\documentclass[11pt,a4paper,reqno]{amsart}  
\usepackage{ifthen,amssymb,mathrsfs,a4wide,color,mathtools}
\mathtoolsset{showonlyrefs}
\let\spacecal=\mathscr
\usepackage[all]{xypic}
\usepackage[backref=page]{hyperref}
\usepackage{tikz-cd}
\usepackage{needspace}
\DeclareMathAlphabet{\mathpzc}{OT1}{pzc}{m}{it}
\hypersetup{
 colorlinks,
 citecolor=green,
 linkcolor=blue,
 urlcolor=Blue}
\makeatletter
\@namedef{subjclassname@2020}{%
  \textup{2020} Mathematics Subject Classification}
\makeatother
\newcommand{\red}[1]{\textcolor{red}{#1}}

\definecolor{aoe}{rgb}{0.0, 0.5, 0.0}
\newcommand{\aoe}[1]{\textcolor{aoe}{#1}}
\let\tempone\itemize
\let\temptwo\enditemize
\let\temponenum\enumerate
\let\temptwonum\endenumerate
\renewenvironment{itemize}{\tempone\addtolength{\itemsep}{0.2\baselineskip}}{\temptwo}
\renewenvironment{enumerate}{\temponenum\addtolength{\itemsep}{0.2\baselineskip}}{\temptwonum}
\let\tempsec\section
\renewcommand{\section}{\par\medskip\tempsec}

\newcommand{\qee}{\parbox{5pt}{\hfill}\hfill $\triangle$}

\newenvironment{rem}{\begin{remqee}}{\qee\end{remqee}}
\newtheorem{thm}{Theorem}[section]
\newtheorem{conj}[thm]{Conjecture}
\newtheorem{corol}[thm]{Corollary}
\newtheorem{lemma}[thm]{Lemma} 
\newtheorem{prop}[thm]{Proposition}
\newtheorem{defin}[thm]{Definition}

\theoremstyle{definition}

\theoremstyle{remark}
\newtheorem{remqee}[thm]{Remark}

\newtheorem{example}[thm]{Example}
\newenvironment{remark}{\begin{remqee}}{\qee\end{remqee}}
\numberwithin{equation}{section}

\newcommand\rk{\operatorname{rk}}
\newcommand\Id{\operatorname{Id}}

\newcommand\coker{\operatorname{coker}}
\newcommand\Spec{\operatorname{Spec}}

\newcommand\End{\operatorname{End}}

\newcommand\Hom{\operatorname{Hom}}
\newcommand\Tor{\operatorname{Tor}}

\newcommand\iso{\kern.35em{\raise3pt\hbox{$\sim$}\kern-1.1em\to}\kern.3em}

\newcommand\sIsom{{\mathcal{I}som}}

\newcommand\rest[2]{#1_{\vert #2}}
\newcommand\Sym{{\mathcal{S}ym}}
\newcommand\F{{\mathcal F}}
\newcommand\Oc{{\mathcal O}}

\newcommand\M{{\mathcal M}}

\newcommand\Z{{\mathbb Z}}

\newcommand\K{{\mathbb K}}

\newcommand\U{{\mathcal U}}
\newcommand\Dcal{{\mathcal D}}
\newcommand\Nc{{\mathcal N}}

\newcommand\Lcl{{\mathcal L}}
\newcommand\Ec{{\mathcal E}}
\newcommand\Fc{{\mathcal F}}
\newcommand\Ac{{\mathcal A}}

\newcommand\X{{\widehat X}}
\newcommand\Gc{{\mathcal G}}
\newcommand\Xcal{{\spacecal X}}
\newcommand\Ycal{{\spacecal Y}}
\newcommand\Sc{{\spacecal S}}
\newcommand\Tc{{\spacecal T}}
\newcommand\Zc{{\spacecal Z}}
\newcommand\Vc{{\spacecal V}}
\newcommand\Wc{{\spacecal W}}

\newcommand\Gsc{{\spacecal G}}

\newcommand\Hc{{\mathcal H}}
\newcommand\Ucal{{\spacecal U}}
\newcommand\Rcal{{\spacecal R}}

\newcommand\Psc{{\spacecal P}}
\newcommand\Qc{{\mathcal Q}}
\newcommand\Ic{{\mathcal I}}
\newcommand\Jc{{\mathcal J}}

\newcommand\bG{{\mathbf  G}}

\newcommand\bP{{\mathbf P}}

\newcommand\bF{{\mathbf  F}}

\newcommand\Der{{{\mathcal D}er}}
\newcommand\Ber{{{\mathcal B}er}}
\newcommand\Homsh{{{\mathcal H}om}}

\newcommand\Xf{{\mathfrak X}}
\newcommand\Yf{{\mathfrak Y}}
\newcommand\Sf{{\mathfrak S}}

\newcommand\Uf{{\mathfrak U}}

\newcommand\Zf{{\mathfrak Z}}

\newcommand\nf{{\mathfrak n}}
\newcommand\Mf{{\mathfrak M}}

\newcommand\As{{\mathbb A}}
\newcommand\Bs{{\mathbb B}}
\newcommand\Vs{{\mathbb V}}
\newcommand\Ss{{\mathbb S}}
\newcommand\Ms{{\mathbb M}}

\newcommand\Ff{{\mathfrak F}}

\newcommand\Homst{{{\mathfrak H}om}}

\newcommand\SSpec{\mathbb{S}\mathrm{pec}\,}

\newcommand\SSym{\mathbb{S}\mathrm{ym}\,}

\newcommand{\rra}{\rightrightarrows}

\newcommand{\bos}{\operatorname{bos}}
\renewcommand{\det}{\operatorname{det}}
\newcommand{\wt}{\widetilde}

\newcommand{\Rf}{\mathfrak R}
\newcommand{\Tf}{\mathfrak T}
\newcommand{\Xb}{\mathbf X}
\newcommand{\Yb}{\mathbf Y}
\newcommand{\Ub}{\mathbf U}

\newcommand{\Bf}{\mathfrak{B}}

\newcommand\GL{\operatorname{GL}}
\newcommand\SGLf{\mathbb{S}\mathcal{G}{L}}
\newcommand\SGL{\mathbb{G}\mathrm{L}}
\newcommand\Gs{\mathbb G}

\newcommand\Gf{\mathfrak G}
\newcommand\Pf{\mathfrak P}

\newcommand\SQuotf{\mathbb{S}\mathcal{Q}{uot}\,}
\newcommand\SQuot{\mathbb{SQ}\mathrm{uot}\,}

\newcommand{\Qcoh}{\mathcal{Q}coh}

\title{Foundations of superstack theory}
\author{Ugo Bruzzo and Daniel Hern\'andez Ruip\'erez}
\address{Departamento de Matem\'atica, Instituto de Ci\^encias Exatas \\
Universidade Federal de Minas Gerais, Av.~Ant\^onio Carlos 6627 \\
Belo Horizonte, 31270-901 MG, Brazil;
IGAP (Institute for Geometry and Physics), Trieste, Italy.}
\address{Departamento de Matem\'aticas and IUFFYM (Instituto Universitario de F\'{\i}sica Fundamental   y Matem\'a\-ticas),  Universidad de Salamanca, Plaza
de la Merced 1-4, 37008 Salamanca, Spain; Real Academia de Ciencias Exactas, F\'{i}sicas y Naturales, Spain.}

\begin{document}
 
\begin{abstract}
Motivated by the construction of moduli spaces in algebraic supergeometry, we initiate a systematic study of stacks within this framework. After defining a superstack as a stack over the \'etale site of superschemes, we introduce quotient superstacks and classifying superstacks of group superschemes. Drawing on existing literature, we establish that, in analogy with superschemes, every superstack possesses an underlying ordinary stack, which we term its bosonic reduction.

We then progressively develop the theory by introducing algebraic superspaces, Deligne--Mumford superstacks, and algebraic superstacks, and prove the representability of their diagonals. We investigate the topology and the dimension of algebraic superstacks and characterize several properties of their morphisms, showing in particular that topological properties, separatedness and properness are detected by the bosonic reduction, whereas smoothness, \'etaleness and flatness are not. Furthermore, we introduce quasi-coherent sheaves  and sheaves of relative differentials in the super context.

A key result provides a generalization of classical criteria to determine when an algebraic superstack is Deligne--Mumford. Specifically, this is characterized in terms of the unramifiedness of the diagonal, or equivalently, of the stabilizers being discrete and reduced ordinary group schemes; this rests on the fact that unramified morphisms of superschemes are automatically bosonic.
We also characterize algebraic superspaces as the algebraic superstacks that are superspaces, and deduce that they are sheaves for the fppf topology. We then discuss principal superbundles with structure group a group algebraic superspace, quotient superstacks by such groups, supertorsors, supergerbes, coarse moduli superspaces and universal families.

We discuss some examples --- superstacks of coherent sheaves on projective superschemes, moduli superspaces and superstacks of smooth and stable SUSY curves with punctures, moduli superstacks of stable supermaps.

Two appendices provide supporting material: the first collects basic definitions of group superschemes and principal superscheme bundles, while the second analyzes several properties of  superscheme morphisms 
that underpin the study of superstack morphisms in the main text.
\end{abstract}
 
\maketitle

\noindent\parbox{\textwidth}{\small \baselineskip=14pt
\noindent\parbox{.8\textwidth}{\hrulefill} \\
\noindent \today  \\
\noindent AMS Subject classification  (2020)  Primary: 14D23; Secondary:    14A20, 14M30, 18F20  
\noindent U.B.: Research partly supported by PRIN 2022BTA242 ``Geometry of algebraic structures: moduli, invariants, deformations,''  INdAM-GNSAGA, and  Brazilian CNPq Bolsas de Produtividade em Pesquisa 313333/2020 and 305343/2025-4 \\
  D.H.R: Research partly supported by the  research projects ``Espacios finitos y functores integrales'', MTM2017-86042-P (Ministerio de Econom\'{\i}a, Industria y Competitividad) and
 GIR 4139 of the University of Salamanca.
 }

\newpage

 \tableofcontents

\section{Introduction}
\label{sec:intro}

The notion of a stack naturally arises in the study of moduli spaces when parameterizing geometric objects that possess nontrivial automorphisms. In classical geometry, objects with internal symmetries create singularities in the moduli space, or even prevent the existence of a moduli space.  By treating these spaces as stacks, one preserves the internal symmetry data of each object, typically by weighting each point by the inverse of the order of its automorphism group --- an essential step for consistent counting.

{Two of the most influential examples in modern geometry are the moduli stack of stable algebraic curves \cite{DelMum69} and the moduli stack of stable maps. These constructions provide powerful tools  for enumerative geometry, for instance through the definition of Gromov-Witten invariants.}

{Beyond pure mathematics, these notions are deeply intertwined with string theory in theoretical physics. In this context, the stack of stable maps provides the rigorous configuration
space necessary to perform the Feynman path integral. In string theory, particles are replaced by one-dimensional strings that sweep out two-dimensional worldsheets. The stack compactifies the space of all possible worldsheet shapes, including those that pinch or degenerate. This allows physicists to sum over all topological configurations, which generate the quantum corrections (specifically, worldsheet instanton effects) that modify   scattering amplitudes. Consequently, the language of stacks is not merely a formal convenience, but a requirement for the mathematical consistency of the quantum strings.}

Largely motivated by supersymmetry and superstring theory, the classical notions of Riemann surface and algebraic curve have been extended to those of super Riemann surface and supersymmetric (SUSY) curve.
{In classical geometry, a Riemann surface is a one-dimensional complex manifold. However, to describe a superstring, which incorporates both force-carrying particles (bosons) and matter particles (fermions), the underlying worldsheet geometry must be enhanced. This is achieved through supergeometry, where the standard  ``even'' complex coordinates   are augmented with  ``odd''   variables. These anticommuting coordinates provide a formal mathematical language for the fermionic degrees of freedom inherent in the string motion.} The moduli spaces of super Riemann surfaces --- the supermoduli spaces --- are the natural domains of integration of superstring perturbation theory \cite{Witten12}. A remarkable feature of these spaces, discovered by Donagi and Witten \cite{DoW13} (see \cite{Correa-Noja, Donagi-Noja} for recent developments), is that they are in general \emph{not projected}, i.e., they do not admit a projection onto their underlying bosonic space; as a consequence, they cannot be reconstructed from their bosonic reductions and a genuinely super-geometric treatment is unavoidable.

Various constructions of the moduli superspaces of these objects have been proposed, beginning with the analytic construction of the moduli space of super Riemann surfaces of genus $g\ge 2$ as an orbifold by Crane and Rabin \cite{CraRab88}, and in the seminal paper by LeBrun and Rothstein \cite{LeRoth88}; see also Deligne's letter to Manin \cite{Del87}. In \cite{DoHeSa97} the second author, together with Dom\'{\i}nguez P\'erez and C. Sancho de Salas, constructed a moduli space for SUSY curves with Neveu–Schwarz punctures and level structures as an algebraic superspace --- that is, as the quotient of an \'etale equivalence relation of superschemes. More recently, Codogni and Viviani provided a construction of the supermoduli of SUSY curves as a Deligne--Mumford  superstack \cite{CodViv17}, without incorporating level structures.\footnote{This paper used the representability of the functor of isomorphisms of SUSY curves, which was afterwards proved  in \cite{FKPpubl} based on  a result in \cite{BrHRPo20}.}
The moduli space constructed in \cite{DoHeSa97} is in fact, in the absence of punctures, an explicit coarse moduli space for that  Deligne--Mumford moduli superstack.

In the  paper \cite{BrHR21}, the present authors constructed a moduli space of SUSY curves with both Neveu–Schwarz and Ramond--Ramond  punctures as an algebraic superspace. The work \cite{FKPpubl} by Felder, Kazhdan, and Polishchuk defined a moduli space of stable SUSY curves as a Deligne--Mumford superstack, thus providing a compactification of the supermoduli of smooth SUSY curves. Stable supermaps were introduced by the present authors and Manin in \cite{BrHRMa23}, and their moduli space was constructed as a superstack. {The paper \cite{BHRSM25}   shows that this moduli space is a Deligne--Mumford  superstack, and   includes further results.} Related constructions, from different viewpoints, are the genus zero supermoduli spaces with Ramond punctures of Ott and Voronov \cite{Ott-Vor}, the operadic approach to the moduli of SUSY curves of Ke\ss ler, Manin and Wu \cite{KeMaWu22}, and the super stable maps of genus zero with bosonic target of Ke\ss ler, Sheshmani and Yau \cite{KSY-GW}.

The lack of a systematic literature on the subject, and the importance of a basic theory of superstacks for the construction of moduli spaces of objects in algebraic supergeometry, and other foreseeable applications, have  prompted us to write these notes on the basic aspects of the theory of superstacks. We provide an introduction to stacks in the context of algebraic supergeometry and define algebraic and Deligne--Mumford stacks in that setting. Superstacks were already   introduced {by Codogni and Viviani} in \cite{CodViv17}, but  in a more restricted framework than ours --- namely, over the site of smooth superschemes with the \'etale topology. However, one often needs to work in more general  settings, which call for broader definitions: for instance, the base superschemes of the families of stable SUSY curves or of stable supermaps of \cite{FKPpubl,BrHRMa23,BHRSM25} are not assumed to be smooth. Some of these generalizations are straightforward, since the notion of a stack over a site is already quite general and, in particular, applies to the site of {\emph{all} }superschemes, smooth or not, with the \'etale topology. This class of superstacks was already considered by  Felder, Kazhdan and Polishchuk \cite{FKPpubl}.

The issue of the algebraicity of superstacks is more delicate. As in the classical case of stacks over the site of schemes with the \'etale topology, there are various possible definitions of what an algebraic or Deligne--Mumford superstack might be. In what follows, we aim to adopt the definitions that most naturally and efficiently generalize the standard ones for stacks, and we prove that the alternative definitions that are customary in the classical setting --- e.g., in terms of the representability of the diagonal --- are equivalent to ours also in the super case. Whenever a classical argument (typically taken from \cite{LauMor-Bai00,Ol16,CasWis17,Alp24} or from the Stacks Project \cite{Stacks}) extends to superschemes with only formal changes, we say so and give precise references; when the extension requires some new input from supergeometry, we provide complete proofs. Much of this new input concerns properties of superscheme morphisms which, to the best of our knowledge, had not been previously discussed in the literature, and is collected in Appendix \ref{sec:app}.

A guiding principle throughout these notes is to understand which properties of a superstack, or of a superstack morphism, are already detected by the underlying bosonic stack --- the \emph{bosonic reduction} --- and which are not. On the one hand, the topological space of a superstack coincides with that of its bosonic reduction (Proposition \ref{prop:points}), and a morphism of algebraic superstacks is open, closed, universally closed, quasi-compact, of finite type, separated or proper if and only if its bosonic reduction is (Propositions \ref{prop:univclosed} and \ref{prop:separpropbos}); likewise, a superscheme morphism is affine or finite if and only if its bosonic reduction is (Proposition \ref{prop:finitebos}). On the other hand, smoothness, \'etaleness and flatness are not entirely  detected by the bosonic reduction, and this is what makes the algebraicity of a superstack a nontrivial issue; we give in Appendix \ref{sec:app} precise criteria (in terms of a vanishing Tor  condition) for a superscheme morphism to be flat or \'etale once its bosonic reduction is (Lemma \ref{lem:flatnilp2} and Proposition \ref{prop:etaleifbosis}). A somewhat unexpected phenomenon, which plays a key role in the characterization of Deligne--Mumford superstacks, is that unramified morphisms are automatically bosonic (Proposition \ref{prop:unram}); in particular, the stabilizers of the points of a Deligne--Mumford superstack are ordinary (non-super) group schemes.

We now describe the contents of these notes in more detail.

\noindent{\bf Superstacks.} A superstack is simply a stack over the \'etale site $\Sf_{et}$ of superschemes (only occasionally we shall consider superstacks in the fppf topology).  Section \ref{subsec:sstacks} introduces superstacks together with the 2-Yoneda Lemma, fiber products, the inertia superstack, and representability of morphisms (by superspaces, i.e.\! \'etale sheaves, and by superschemes). Already here the diagonal plays a prominent role: a superstack is a superspace if and only if its diagonal is a monomorphism (Propositions \ref{prop:injdiag} and \ref{prop:isodiag}). Section \ref{subsec:quot} studies quotient superstacks $[\Xf/\Gsc]$ and classifying superstacks $\Bf\Gsc$, proving they are superstacks when $\Gsc$ is affine and fppf. In Section \ref{sec:bosred}, largely following \cite{CodViv17}, we characterize the bosonic stack $\Xf_{bos}$ underlying a superstack $\Xf$ and show that it is a sub-superstack of $\Xf$ (Proposition \ref{prop:prop:bosred}) through which every morphism from a bosonic superstack to $\Xf$ factors.

\noindent{\bf Algebraic superspaces, Deligne--Mumford and algebraic superstacks.} The main thread of the notes is to define algebraic superspaces (Section \ref{subsec:algsstacks}), then algebraic and Deligne--Mumford superstacks (Section \ref{subsec:dmsstacks}). Algebraic superspaces are the super analogues of the algebraic spaces of Artin and Knutson \cite{Ar71,Knut71}: \'etale sheaves admitting a superschematic surjective \'etale morphism from a superscheme, or, equivalently, quotients of \'etale equivalence relations of superschemes (Proposition \ref{prop:quot}). Their bosonic reduction is an algebraic space (Proposition \ref{prop:represbos}). Deligne--Mumford and algebraic superstacks are superstacks admitting a representable surjective \'etale, resp.\! smooth, morphism from a superscheme. A main result of Section \ref{subsec:properties} is the representability of the diagonal (Proposition \ref{prop:represdiag}), which yields the alternative definition of algebraic and Deligne--Mumford superstacks in terms of the diagonal (Corollary \ref{cor:represdiag2}) and shows that the stabilizers of the points are \emph{group algebraic superspaces}, which need not be superschemes (Remark \ref{rem:groupalgss}). Properties of superscheme morphisms which are local on target and source for the \'etale (resp.\! smooth) topology extend to morphisms of Deligne--Mumford (resp.\! algebraic) superstacks.

The rest of Section \ref{sec:algsstacks} deals with topological properties, governed by the bosonic reduction, and with the dimension $m\vert n$ of an algebraic superspace or superstack (Section \ref{subsec:dimsstack}); since there are two sensible notions of odd dimension (Appendix \ref{sec:app}), we show that the definition is well posed for either, as they agree for smooth morphisms (Proposition \ref{prop:smoothsplit3}). The following section treats quasi-finiteness (Definition \ref{def:qfsstacks}), separatedness and properness (Section \ref{subsec:separprop}), and noetherian superstacks; separatedness and properness are detected by the bosonic reduction (Proposition \ref{prop:separpropbos}).

\noindent{\bf Quasi-coherent sheaves.} Sections \ref{subsec:quasi-coh} and \ref{subsec:quasi-coh2} introduce quasi-coherent sheaves on Deligne--Mumford and algebraic superstacks, via the small \'etale and the lisse-\'etale site respectively, and explain why the two notions agree for Deligne--Mumford superstacks. Section \ref{subsec:quasi-coh} also develops relative differentials for morphisms of Deligne--Mumford superstacks and algebraic superspaces, based on descent for quasi-coherent sheaves (Proposition \ref{prop:descentqcalgsup}), including the vanishing of the relative differentials of unramified morphisms (Proposition \ref{prop:propreldiff}). Following \cite{Ol16}, we construct the sheaf of relative differentials $\Omega_{\Xcal/\Xf}$ of a smooth presentation (Lemma \ref{lem:propadhoc}), the main technical ingredient in the characterization of Deligne--Mumford superstacks.

\noindent{\bf Characterization of Deligne--Mumford superstacks and algebraic superspaces.} In Section \ref{subsec:dmdiag} we prove that an algebraic superstack is Deligne--Mumford if and only if its diagonal is unramified, or, equivalently, if and only if the stabilizer of every point is a discrete and reduced ordinary group scheme (Proposition \ref{prop:unramdiag1}). The proof adapts the argument of Laumon and Moret-Bailly \cite{LauMor-Bai00} and relies on the fact that unramified algebraic superspaces are bosonic (Lemma \ref{lem:unram}). Section \ref{subsec:algdiag} shows that algebraic superspaces are exactly the algebraic superstacks that are superspaces (Proposition \ref{prop:algsp}); consequently, quotients of \'etale equivalence relations of algebraic superspaces are algebraic superspaces (Proposition \ref{prop:quot2}), yielding \'etale descent (Corollary \ref{cor:algssdescent}), and algebraic superspaces and quotient superstacks are fppf sheaves and stacks (Proposition \ref{prop:fppfsheaf}, Corollary \ref{cor:fppfsheaf}).

\noindent{\bf Further constructions and moduli problems.} The next section extends principal superbundles to structure groups that are group algebraic superspaces, proves \'etale descent for them (Proposition \ref{prop:psbdescent}), and shows that the quotient $[\Xf/\Gf]$ of an algebraic superspace by a group algebraic superspace is a superstack, algebraic when $\Gf$ is smooth, and an algebraic superspace when moreover the action is free (Propositions \ref{prop:class}, \ref{prop:quotstack2} and \ref{prop:algebraicspacequotient}). We compare \'etale and fppf supertorsors (Definition \ref{def:supertorsor}) with principal superbundles: the notions agree for smooth groups, while in the fppf case we only get a fully faithful functor (Proposition \ref{prop:torsorsPsb}), as a super analogue of Artin's theorem is not yet available (Remark \ref{rem:ArtinThm}). We also define supergerbes (Definitions \ref{def:supergerbe} and \ref{def:bandedsg}), coarse moduli superspaces and geometric quotients (Definitions \ref{def:coarse} and \ref{def:geomquot}; a Keel-Mori theorem for Deligne--Mumford superstacks has been recently proved in \cite{Peng}), and construct the universal family over an algebraic superstack via a generalized 2-Yoneda Lemma (Lemma \ref{lem:gen2yoneda}, Definition \ref{def:univfam}).

\noindent{\bf Examples.}
In Section \ref{sec:examples} the general theory is applied to some moduli superstacks, drawn from our previous work and from the literature, to establish their algebraicity, Deligne--Mumford property, smoothness or properness, and to determine their bosonic reductions.
\begin{itemize}  \item The superstack of coherent sheaves on a projective superscheme, from a theory being developed by the second author with L.\! \'Alvarez C\'onsul and T.\! G\'omez \cite{LTD26}, proving  it is algebraic, after reviewing the construction of Quot superschemes.
\item The moduli space of smooth SUSY curves with Neveu-Schwarz and Ramond--Ramond punctures and a level structure (Section \ref{subsec:RR-SUSY}), following \cite{BrHR21}: an algebraic superspace of dimension $(3g-3+\nf_R, 2g-2+\nf_R/2)$ whose bosonic reduction is the moduli scheme of Ramond--Ramond spin curves.
\item  The superstack of punctured stable SUSY curves (Section \ref{subsec:Stcurves}), following Felder, Kazhdan and Polishchuk \cite{FKPpubl}: a proper and smooth Deligne--Mumford superstack whose bosonic stack is the moduli stack of punctured stable spin curves \cite{Corna87,JaKiVa05,MoZh24}. Properness follows from that of the bosonic reduction, whereas smoothness and the Deligne--Mumford property require an explicit \'etale presentation built from the sheaf of odd infinitesimal deformations.
\item The moduli superstack of {\em stable supermaps} into a projective superscheme (Section \ref{subsec:Stmaps}), following \cite{BHRSM25,BrHRMa23}: it is Deligne--Mumford by Proposition \ref{prop:unramdiag1}. Unexpectedly, its bosonic stack is not the moduli stack of stable spin maps to the bosonic reduction of the target, but only admits an affine morphism to it with affine cones as fibers; hence the superstack is proper if and only if this morphism is an isomorphism, e.g., when the target is bosonic.
\end{itemize}

\noindent{\bf Appendices.} Appendix \ref{sec:supergr} recalls group superschemes, algebraic supergroups, actions and principal superscheme bundles. Appendix \ref{sec:app} discusses the notions of odd dimension of a superscheme and the equivalence of the resulting notions of smoothness (Propositions \ref{prop:smoothsplit2} and \ref{prop:smoothsplit3}); studies formally smooth, \'etale and unramified morphisms and their relation with the bosonic reduction, proving in particular that unramified superscheme morphisms are bosonic (Proposition \ref{prop:unram}) and characterizing unramified morphisms and monomorphisms locally of finite type (Propositions \ref{prop:unramprop} and \ref{prop:mono}). Further topics are local properties of morphisms for the \'etale, smooth and fppf topologies, quasi-finite, radicial and quasi-affine morphisms, a factorization result in the spirit of Zariski's Main Theorem (Proposition \ref{prop:qaffine2}), and descent theory: fpqc descent for quasi-coherent sheaves (Proposition \ref{prop:fpqcdescent2}), the fact that superschemes are fppf sheaves (Proposition \ref{prop:fpfpmordescent}), and effective descent for sub-superschemes, affine, quasi-affine and locally quasi-finite separated morphisms, principal bundles and polarized superschemes (Section \ref{subsec:fppf}).

As a reference for stack theory   we follow mainly Alper's notes \cite{Alp24},  the paper by Casalaina-Martin  and Wise
\cite{CasWis17}   and the books by Laumon and Moret-Bailly  \cite{LauMor-Bai00} and Olsson \cite{Ol16}. See also \cite{CodViv17} for some facts and results about superstacks. For superscheme theory, we mainly use \cite{BrHRPo20} and also  \cite{BrHR21}. Note that, following  \cite{BrHRPo20}, the ideal sheaf locally generated by the odd sections of the structure sheaf of a superscheme is assumed to be locally finitely generated.  

Once the foundations of superstack theory are established, several avenues for further development arise. Some of them have already been mentioned: a super version of Artin's theorem on the algebraicity of fppf quotients, which would also yield the equivalence between fppf supertorsors and principal superbundles.  A particularly salient direction is the deformation theory of superstacks, and the construction of (perfect) obstruction theories \cite{SPOT}. Following the framework established by Behrend and Fantechi \cite{BeFant97} for classical stacks, such a theory will be essential for defining virtual fundamental classes in the super setting. This will, in turn, open the way to an algebro-geometric definition of Gromov-Witten invariants, and other enumerative invariants for superschemes (super Gromov-Witten invariants were defined in \cite{KSY-GW} in the symplectic setting, with a bosonic scheme (ordinary scheme) as target).

\smallskip \noindent {\bf Acknowledgements.} We thank Nadia Ott for pointing out a mistake in a previous version, and Luis \'Alvarez-Consul, Barbara Fantechi, Tom\'as G\'omez and Andrea T.~Ricolfi for discussions.
D.H.R.~thanks GNSAGA-INdAM for support, and SISSA for support and hospitality. 
Claude Fable 5.1 has been used to reformat the Introduction.
\section{Superstacks}
\label{sec:sstacks}
 We start with the basic definitions and properties of superstacks. 
\subsection{Pre-superstacks and superstacks}
\label{subsec:sstacks}
We shall denote by $\Sf$ the category of superschemes. 
We equip it with the   \emph{\'etale topology}, that is,  is the Grothendieck topology    whose coverings are the surjective \'etale morphisms \cite[Def.\! 2.14]{BrHRMa23}\cite[Def.\,7]{DoHeSa97}. We denote by $\Sf_{et}$ the resulting site.

\begin{remark}\label{rem:subcanonical} 
  Note that $\Sf_{et}$ is \emph{subcanonical}, that is, every  functor  $\Sf^\circ\to(Sets)$ representable by a superscheme is a sheaf with respect to the \'etale topology. Actually, one can see that it is a sheaf also for the coarser fpqc topology of superschemes, and then for the fppf topology as well. The fppf case is actually Proposition \ref{prop:fpfpmordescent}.
\end{remark}

\begin{defin}\label{def:sstack} A pre-superstack (resp.\! a superstack) is a pre-stack (resp.\! a stack) over $\Sf_{et}$ \cite[Def.\, 3.31]{CasWis17}. This definition can be spelled out as follows: a pre-superstack (resp.\! a superstack) is a category fibered in groupoids
$
p_\Xf\colon \Xf \to \Sf_{et}
$
over the site $\Sf_{et}$ of superschemes with the \'etale topology,  such that the isomorphisms between two objects in $\Xf(\Sc):=p_\Xf^{-1}(\Sc)$ form a sheaf for every superscheme $\Sc$ (resp.\! ditto,  and that every descent datum is effective). A morphism of pre-superstacks or superstacks is a morphism of categories  fibered in grupoids over $\Sf$, that is, a functor commuting with the projections to $\Sf$. If $\Tc$ is a superscheme, changing $\Sf$ for the category $\Sf/\Tc$ of superschemes over $\Tc$ and $\Sf_{et}$ for the corresponding \'etale site $(\Sf/\Tc)_{et}$, we obtain the notions of pre-superstack and superstack over $\Tc$. 
\end{defin}

A first example of a superstack is the one associated to a sheaf on $\Sf_{et}$: consider a presheaf    $\F\colon \Sf^\circ\to(Sets)$. One defines a category  fibered in grupoids $p_\Ff\colon \Ff \to \Sf$ by letting $\Ff(\Sc)$ to be the set $\F(\Sc)$ considered as a category; the   morphisms in $\Ff$ are defined in the obvious way. Note that $p_\Xf\colon\Ff\to\Sf$ is a category fibered in setoids, namely, groupoids whose objects have no other morphisms than the identity.

The following result was proved in \aoe{\cite[Prop.\ 3.33]{CasWis17}} in full generality for stacks over an arbitrary site, and therefore also applies to superstacks.
\begin{prop} $\Ff\to\Sf_{et}$ is a pre-superstack (resp.\! a superstack) if and only if $\F$ is a separated presheaf (resp.\! a sheaf).\qed
\label{prop:sheafstack}\end{prop}

Since for every superscheme $\Xcal$ the functor of   points $\Xcal\colon \Sf^\circ\to(Sets)$ is an \'etale sheaf, $\Xcal$ gives rise to
 a superstack.  If $\Xf$ is a superstack equivalent, as a fibered category,  to the superstack  associated with a superscheme $\Xcal$, 
 we say that \emph{$\Xf$ is representable by the superscheme $\Xcal$}; we will confuse notationally $\Xf$ and $\Xcal$. As a fibered category, the superstack $\Xcal$ is equivalent to the category  fibered in grupoids $\Sf/\Xcal\to \Sf$.

There is a version of the Yoneda Lemma for superstacks. To formulate it, let $p_\Xf\colon \Xf\to\Sf$, $p_\Yf\colon \Yf\to\Sf$ be superstacks. We denote by $\Homst_{SSt}(\Xf,\Yf)$ the fibered category whose objects are the morphisms $\Xf \to \Yf$ of superstacks and whose morphisms are the natural isomorphisms. By \cite[Lemma 2.16]{CasWis17} one has:
\begin{lemma}[2-Yoneda]\label{lem:2yoneda} Let $p_\Xf\colon\Xf\to\Sf$ be a superstack. For every superscheme $\Sc$ the natural transformation
$$ 
\Homst_{SSt}(\Sc,\Xf)  \to \Xf(\Sc)\,,
$$
defined on objects by $F \mapsto F(\Sc\xrightarrow{\Id_\Sc} \Sc)$ and on morphisms by mapping a natural transformation of functors $\phi\colon F \to F'$  to the morphism $F(\phi)\colon F(\Sc\xrightarrow{\Id_\Sc} \Sc)\to F'(\Sc\xrightarrow{\Id_\Sc} \Sc)$ in $\Xf(\Sc)$,  is an equivalence of categories.
\qed
\end{lemma}

By virtue of the 2-Yoneda Lemma, 
specifying a superstack morphism $ \Sc \to \Xf$ is the same as specifying an object in  $\Xf(\Sc)$. We will   often confuse the two notions  in the notation.  

The fiber product $\Xf\times\Yf$ of two superstacks is their fiber product as fibered categories; it is also a superstack. This is a particular case of \cite[Def. 4.2 and Lemma 4.4]{CasWis17}. For convenience we recall these definitions and results. 

\begin{defin}\label{def:sstackprod} Let $f\colon \Xf \to \Zf$, $g\colon \Yf \to \Zf$ be morphisms of  categories fibered in grupoids over $\Sf$. Their  fiber product    $\Xf\times_{f,\Zf,g}\Yf$  is the category fibered over $\Sf$ whose objects are quadruples $(\Sc,x,y,\alpha)$, where $\Sc$ is a superscheme, while $x\in\Xf(\Sc)$, $y\in\Xf(\Sc)$, and $\alpha\colon f(x)\iso g(y)$ is an isomorphism in $\Zf(\Sc)$. A morphism $(\Sc,x,y,\alpha)\to(\Sc',x',y',\alpha')$ is a triple $(\varphi, u, v)$, where $\varphi\colon \Sc\to \Sc'$ is a superscheme  morphism  and $u\colon x\to x'$ and $v\colon y\to y'$ are morphisms in $\Xf$ lying over $\varphi$ such that $\alpha'f(u)=g(v)\alpha$.
\end{defin}
\begin{prop}\label{prop:sstackprod}  If $f\colon \Xf \to \Zf$, $g\colon \Yf \to \Zf$ are morphisms of superstacks, the fiber product $\Xf\times_{f,\Zf,g}\Yf$ is a superstack.
\qed
\end{prop}

\begin{remark}\label{rem:prod}
For every superstack $\Xf$ there is a natural superstack morphism $q_\Xf\colon \Xf \to \Spec \Z$. Given superstacks $\Xf$, $\Yf$,  we denote by  $\Xf\times\Yf$ the fiber  product $\Xf\times_{q_\Xf,\Spec\Z,q_\Yf}\Yf$.
\end{remark}

\begin{defin}\label{def:sspc} 
{A superspace is a superstack   that is equivalent to the superstack   associated with a sheaf on $\Sf_{et}$, or, equivalently, 
a superstack that is 
  fibered in setoids (see Proposition \ref{prop:sheafstack}). }
\end{defin}

As a straightforward consequence of this definition we have:
\begin{prop} A fiber product of morphisms of superspaces is a superspace.
\qed
\end{prop} 

This motivates the following:
\begin{defin}\label{def:repmor} A morphism $f\colon \Xf \to \Yf$ of superstacks is representable by superspaces if for every superspace $\Zf$ and every superstack morphism $g\colon\Zf\to \Yf$, the fiber product superstack $\Xf\times_{g,\Yf ,f}\Xf$ is   a superspace. Analogously, $f$ is said to be representable by superschemes, or superschematic, if for every superscheme $\Sc$ and every superstack morphism $g\colon\Sc\to \Yf$ the fiber product superstack $\Sc\times_{g,\Yf ,f} \Xf$ is   a superscheme.
\end{defin}

\begin{prop}\label{prop:sch-spc} A superschematic morphism $f\colon \Xf \to \Yf$ of superstacks is representable by superspaces.
\end{prop}
\begin{proof} One can assume that $\Yf$ is a superspace and prove that if $f\colon \Xf \to \Yf$ is superschematic, then $\Xf$ is a superspace. For every superscheme $\Sc$ and every morphism $g\colon \Sc\to \Yf$, the objects of $(\Xf\times_{f,\Yf,g}\Sc)(\Sc)$ are pairs $(x,t)$ where $x\in\Xf(\Sc)$, $t\in \Sc(\Sc)=\Hom(\Sc,\Sc)$, and $f(x)=g(t)$, because $\Yf(\Sc)$ is a setoid (it has no   morphisms other  than the identities). This proves that $\Xf(\Sc)$  is a subcategory of $(\Xf\times_{f,\Yf,g}\Sc)(\Sc)$; since the latter is a setoid, $\Xf(\Sc)$ is also a setoid and $\Xf$ is a superspace.
\end{proof}

As there are different notions of  representable morphism in the literature, we shall  specify whether representability is understood with respect to superspaces or superschemes.

\begin{lemma}\label{lem:faith} Let $f\colon \Xf\to\Yf$ be a superstack morphism such that for every superscheme $\Sc$, the functor $f(\Sc)\colon \Xf(\Sc) \to \Yf(\Sc)$ is faithful. Then $f$ is representable by superspaces.
\end{lemma}
\begin{proof} Let $\Zf$ be a superspace and $\Zf \to \Yf$ a stack morphism. For every superscheme $\Sc$ the induced functor $(\Xf\times_\Yf\Zf)(\Sc) \to \Zf(\Sc)$ is faithful. Since the category $\Zf(\Sc)$ is (equivalent to) a setoid because $\Zf$ is a superspace, the category $(\Xf\times_\Yf\Zf)(\Sc)$ is equivalent to a setoid as well. Thus, $\Xf\times_\Yf\Zf$ is a superspace.
\end{proof}

The diagonal morphism $\Delta_\Xf \colon \Xf\to \Xf\times\Xf$ of a superstack $\Xf$ has the important property that it encodes the isomorphisms between the objects of the category $\Xf$; namely, for every superscheme $\Sc$ and objects $x_1$, $x_2$ in $\Xf(\Sc)$, if we consider the latter as morphisms $x_1\colon \Sc\to \Xf$, $x_2\colon \Sc \to \Xf$ by the 2-Yoneda Lemma, there is a 2-cartesian diagram
\begin{equation}\label{eq:diagonal}
\xymatrix{ 
\mathfrak{I}som_{\Xf(\Sc)}(x_1,x_2) \ar[r]\ar[d] & \Sc \ar[d]^{(x_1,x_2)} \\
\Xf\ar[r]^{\Delta_\Xf}& \Xf\times\Xf
}
\end{equation}

More generally, if $f\colon \Xf\to \Yf$ is a superstack morphism, and $(x_1,x_2,\alpha)$ --- where $\alpha\colon f(x_1) \iso \allowbreak f(x_2)$ is an isomorphism in $\Yf(\Sc)$ --- is an object of $(\Xf\times_\Yf \Xf)(\Sc)$,  we have a 2-cartesian diagram
\begin{equation}\label{eq:diagonal2}
\xymatrix{ 
\mathfrak{I}som_{\Xf(\Sc),\alpha}(x_1,x_2) \ar[r]\ar[d] & \Sc \ar[d]^{(x_1,x_2,\alpha)} \\
\Xf\ar[r]^{\Delta_f}& \Xf\times_\Yf\Xf
}
\end{equation}
where $\mathfrak{I}som_{\Xf(\Sc),\alpha}(x_1,x_2)$ is the pre-image of $\alpha$ under the natural map $f\colon \mathfrak{I}som_{\Xf(\Sc)}(x_1,x_2)\to \mathfrak{I}som_{\Yf(\Sc)}(f(x_1),f(x_2))$.

 \subsubsection*{Inertia superstack} As for the ordinary stacks, there is an object that in a sense parameterizes the automorphisms of the points of a superstack
 $\Xf$; it is called the \emph{inertia superstack}, and is defined as follows.
\begin{defin}\label{def:inertia} The inertia superstack of a superstack $\Xf$ is the fiber product $$\mathfrak{I}_\Xf:= \Xf\times_{\Delta_\Xf,\Xf\times\Xf,\Delta_\Xf}\Xf,$$  i.e.,  $\mathfrak{I}_\Xf$ is defined by the 2-cartesian diagram
$$
\xymatrix{
\mathfrak{I}_\Xf \ar[r]\ar[d] & \Xf \ar[d]^(.45){\Delta_\Xf}
\\
\Xf\ar[r]^(.4){\Delta_\Xf} & \Xf\times\Xf
}
$$
In a similar way one defines the inertia superstack of a superstack $\Xf \to \Sc$ over a superscheme $\Sc$. 
\end{defin}

 A direct computation shows what are the objects of the inertia superstack:
\begin{prop}\label{prop:inertia} Let $\Tc$ be a superscheme. The objects of $\mathfrak{I}_\Xf(\Tc)$ are   the pairs $(x,\alpha)$ where $x\in \Xf(\Tc)$ is an object of $\Xf$ over $\Tc$, and $\alpha\colon x \iso x$ is an automorphism in $\Xf(\Tc)$.
\qed
\end{prop}

We go back to studying properties of superstack morphisms. 

\begin{defin}\label{def:injective}
A morphism $f\colon \Xf\to\Yf$ of superstacks is a monomorphism or  it is injective (resp.\! an isomorphism) if for each superscheme $\Sc$ the functor $f(\Sc)\colon \Xf(\Sc)\to\Yf(\Sc)$ is fully faithful (resp.\! an equivalence of categories).  In the first instance, we   say that $\Xf$ is a sub-superstack of $\Yf$ (via $f$) and write $\Xf \subseteq \Yf$.
\end{defin}

\begin{remark}\label{rem:inj}
For a   superscheme morphism $f\colon\Xcal\to\Ycal$  there is already a notion of injectivity, namely, that the induced map between the underlying topological spaces $f_{bos}\colon X\to Y$ is injective. For  superscheme morphisms  that are injective in the sense of Definition \ref{def:injective}, we will always use the term monomorphism.
\end{remark}

To define the notions of open and closed sub-superstack it is   convenient to include a requirement of representability.

\begin{defin}\label{def:openclosed} A sub-superstack $\Xf\subseteq \Yf$ is open (resp.\! closed) if, for every superscheme $\Vc$ and every superstack morphism $\Vc\to \Yf$, the superstack $\Vc\times_\Yf\Xf$ is representable by an open (resp.\! closed) superscheme of $\Vc$.
\end{defin}

\begin{lemma}\label{lem:injbasechange} Let $f\colon \Xf\to\Yf$ be a superstack morphism. 
\begin{enumerate}
\item If $f$ is injective, for every superstack morphism $g\colon \Zf\to\Yf$ the base change morphism $f_\Zf\colon \Xf_\Zf:=\Xf\times_{f,\Yf,g}\Zf \to \Zf$ is injective.
\item If, for every superscheme $\Sc$ and every morphism $g\colon\Sc\to \Yf$, the base change morphism $f_\Sc\colon \Xf\times_\Yf \Sc\to \Sc$ is injective, then $f$ is injective.
\end{enumerate}
\end{lemma}
\begin{proof} (1) is straightforward. For (2), take a superscheme $\Sc$ and fix an object $X$ of $\Xf(\Sc)$. Then $f(X)$ is an object of $\Yf(\Sc)$, and by the 2-Yoneda Lemma \ref{lem:2yoneda}, it corresponds to a morphism $g\colon \Sc\to \Yf$. Now make a base change by $g$, so that one has a morphism $\Xf_\Sc:=\Xf\times_{f,\Yf,g}\Sc\to \Sc$ that is injective by hypothesis. Consider another object $X'\in\Xf(\Sc)$ and a morphism $\alpha\colon f(X') \to f(X)=g$. Then, we have objects  $(X,\Id,\Id)$ and $(X',\Id,\alpha)$ of $\Xf_\Sc(\Sc)$ and a morphism (the identity) in $\Sc(\Sc)$ between $f_\Sc(X',\Id,\alpha)$ and $f_\Sc(X,\Id,Id)$. This morphism has to be of the form $(\phi,\Id)$, where $\phi\colon X'\to X$ is a morphism such that $f(\phi)=\alpha$. Since $f_\Sc\colon \Xf_\Sc(\Sc)\to \Sc(\Sc)$ is fully faithful, $\phi$ is uniquely determined by $\alpha$, and  this   means that $\Xf(\Sc) \to \Yf(\Sc)$ is fully faithful.
\end{proof}

\begin{lemma}\label{lem:injdiag}
If $f\colon \Xf\to\Yf$ is a superstack morphism,  the natural morphism $\iota\colon\Xf\times_\Yf \Xf \to \Xf\times\Xf$ is representable by superspaces. If $\Yf$ is a superspace, then $\iota$ is injective.
\end{lemma}
\begin{proof} For every superscheme $\Sc$, the functor $\iota(\Sc)\colon(\Xf\times_\Yf \Xf)(\Sc)\to (\Xf\times\Xf)(\Sc)$ maps $(x_1,x_2,\alpha)$ to $(x_1,x_2)$ where $x_1$, $x_2$ are objects of $\Xf(\Sc)$ and $\alpha\colon f(x_1)\iso f(x_2)$ is an isomorphism in $\Yf(\Sc)$. A morphism $(x_1,x_2,\alpha)\to (x'_1,x'_2,\alpha')$ is a pair of isomorphisms $\phi_i\colon x_i\to x'_i$ such that  $\alpha' f(\phi_1)=f(\phi_2) \alpha$. Since the map 
$$
\Hom_{\Xf(\Sc)}((x_1,x_2,\alpha), (x'_1,x'_2,\alpha'))\to \Hom_{\Yf(\Sc)}((x_1,x_2), (x'_1,x'_2))
$$
 sends $(\phi_1,\phi_2)$ to $(\phi_1,\phi_2)$, we see that $\iota(\Sc)$ is faithful, and then $\iota\colon\Xf\times_\Yf \Xf \to \Xf\times\Xf$ is representable by superpaces by Lemma \ref{lem:faith}. If $\Yf$ is a superspace, $\Yf(\Sc)$ is a setoid. Then $\alpha$ and $\alpha'$ have to be the identity, so that $\iota(\Sc)$ is fully faithful;   now  one applies Definition \ref{def:injective}.
\end{proof}

\begin{prop}\label{prop:injdiag} 
\begin{enumerate}
\item
A superstack $\Xf$ is a superspace if and only if the diagonal morphism $\Delta_\Xf\colon \Xf\to\Xf\times \Xf$ is injective.
\item A superstack morphism $f\colon \Xf\to\Yf$ is representable by superspaces if and only if the diagonal $\Delta_f\colon \Xf \to \Xf\times_\Yf \Xf$ is injective.
\end{enumerate}
\end{prop}
\begin{proof}
(1) This is \cite[Lemma 4.8]{CasWis17}:  the injectivity of the diagonal means, according to equation \eqref{eq:diagonal}, that for every  objects $x$, $y$ in $\Xf(\Sc)$, $ \mathfrak{I}som_{\Xf(\Sc)}(x,y)$ is a subobject of $\Sc$, that is,   there is at most one isomorphism between $x$ and $y$. Thus, $\Xf(\Sc)$ is equivalent to a setoid.

(2) Suppose that $\Delta_f$ is injective. By Lemma \ref{lem:injbasechange}, to show that $f\colon \Xf\to\Yf$ is representable by superspaces we can assume that $\Yf$ is a superspace; we need to prove that  $\Xf$ is a superspace as well. Actually $\iota\colon \Xf\times_\Yf \Xf \to \Xf\times \Xf$ is injective by Lemma \ref{lem:injdiag}, so that  $\Delta_\Xf=\iota \circ\Delta_f$ is injective as well and we may apply (1) to conclude. Conversely, if $f\colon \Xf\to\Yf$ is representable by superspaces, then for  every morphism $g\colon\Sc \to \Yf$ from a  superscheme $\Sc$, the superstack $\Xf_\Sc:=\Xf\times_{f,\Yf,g} \Sc$ is a superspace. By (1), its diagonal  morphism $\Delta_{\Xf_\Sc}$ is injective. Since $\Delta_{\Xf_\Sc}$ is the composition of  the relative diagonal $\Delta_{f_\Sc}\colon\Xf_\Sc \to \Xf_\Sc\times_\Sc \Xf_\Sc$ and the natural morphism $\Xf_\Sc\times_\Sc \Xf_\Sc\to \Xf_\Sc\times \Xf_\Sc$, we see that $\Delta_{f_\Sc}$ is   injective as well.  Then, $\Delta_f$ is injective  by Lemma \ref{lem:injbasechange}. 
\end{proof}

\begin{prop}\label{prop:isodiag} A morphism $f\colon \Xf \to \Yf$ of superstacks is injective if and only if  the diagonal $\Delta_f\colon \Xf \to \Xf\times_\Yf \Xf$ is an isomorphism.
\end{prop}
\begin{proof} Assume that $f$ is injective and consider an object $(x_1,x_2,\alpha\colon f(\Sc)(x_1)\iso f(\Sc)(x_2))$ of $(\Xf\times_\Yf \Xf)(\Sc)$ for a superscheme $\Sc$. Since $f(\Sc)$ is fully faithful, there is an isomorphism $\beta\colon x_1\iso x_2$ such that $\alpha=f(\Sc)(\beta)$. This proves that $\Delta_f(\Sc)\colon \Xf(\Sc) \to (\Xf\times_\Yf \Xf)(\Sc)$ is essentially surjective. Analogously, we see that it is fully faithful, and then it is an equivalence of categories. The converse is similar.
\end{proof}
 
\begin{corol}\label{cor:injrepres} Every injective superstack morphism is representable by superspaces.
\end{corol}
\begin{proof} It follows from Propositions \ref{prop:injdiag} and \ref{prop:isodiag}.
\end{proof}
 
We conclude this section by introducing a notational convention that is standard in the classical (bosonic) setting:
if $\Sc=\SSpec \As$ is an affine superscheme and $\Xf$ is a pre-superstack, we sometimes write $\Xf(\As)$ for the category $\Xf(\Sc)$.

\subsection{Quotient superstacks}
\label{subsec:quot}

Quotient superstacks are defined in the same way as quotient stacks. We are going to consider quotient superstacks associated with a right action of a group superscheme on a superspace and not only on a superscheme.

If $\Sc$ is a superscheme and $\Gsc\to\Sc$ is an $\Sc$-group superscheme, Definition \ref{def:action} about an action of  $\Gsc$  on a superscheme $\Xcal$ makes sense  even after replacing  $\Xcal$ with an $\Sc$-superspace $\Xf$. Also the notion of $\Gsc$-equivariant morphism (Definition \ref{def:equivariant}) is readily extended to superspaces acted on by $\Gsc$. 

\begin{defin}\label{def:quotstack} Let  $\mu\colon \Gsc\times_\Sc\Xf \to \Xf$ be a left action of $\Gsc\to\Sc$ on an $\Sc$-superspace $\Xf$.
 The quotient category   is the category fibered in grupoids
$p_{[\Xf/\Gsc]}\colon [\Xf/\Gsc] \to \Sf/\Sc$
given by:
\begin{enumerate}
\item objects of $[\Xf/\Gsc](\Tc)$ over an $\Sc$-superscheme $\Tc\to\Sc$ are diagrams
$$
\xymatrix{ \Psc \ar[d]^\pi\ar[r]^f & \Xf \\
\Tc &
}
$$
where $\pi\colon\Psc\to\Tc$ is a principal $\Gsc$-superscheme-bundle (Definition \ref{def:principalbundle}) and $f\colon \Psc\to\Xf$ is a $\Gsc$-equivariant morphism of $\Sc$-superspaces;
\item  morphisms are pairs of $\Sc$-superscheme morphisms $\phi\colon \Psc'\to\Psc$, $\psi\colon \Tc'\to\Tc$, where $\phi$ is $\Gsc$-equivariant, such that the diagram
$$
\xymatrix{
\Psc' \ar[d]^{\pi'}\ar[r]_\phi \ar@/^1.5pc/[rr]^{f'} \ar@{}[rd] |{\square} & \Psc\ar[d]^\pi\ar[r]_f & \Xf \\
\Tc'\ar[r]_\psi& \Tc
}
$$
commutes and the   square is cartesian.
\end{enumerate}
One easily sees that   $[\Xf/\Gsc]\to\Sf/\Sc$ is a pre-superstack.
\end{defin}

\begin{defin}
The classifying pre-superstack of a group superscheme $\Gsc\to\Sc$ is the quotient pre-superstack $\Bf\Gsc:=[\Sc/\Gsc]\to\Sf/\Sc$, where $\Gsc$ acts on $\Sc$   trivially.
\end{defin}

Using Proposition \ref{prop:pbdescent} about descent for principal superscheme bundles together with the fact that morphisms to a superspace glue uniquely in the \'etale topology (because superspaces are \'etale sheaves), we  have:
\begin{prop}\label{prop:quotstack} If $\Gc\to\Sc$ is fppf and affine, the quotient CFG $[\Xf/\Gsc]\to\Sf/\Sc$ is a superstack. In particular, the classifying pre-superstack $\Bf\Gsc$ is a superstack. 
\qed
\end{prop}
 
\begin{remark}\label{rem:quotstack} We shall see in Proposition \ref{prop:fppfsheaf} that an algebraic superspace is a sheaf for the fppf topology of superschemes. Using the same argument,  $\Bf\Gsc$ is also an fppf superstack. We will explicitly  state this property  as Corollary \ref{cor:fppfsheaf}.
\end{remark}

For any $\Sc$-superscheme $\Tc$, given an object of $\Xf(\Tc)$, that is, a morphism $\phi\colon\Tc \to \Xf$ of $\Sc$-spaces due to the 2-Yoneda Lemma \ref{lem:2yoneda},  we have an object in $[\Xf/\Gsc](\Tc)$ given by the action
$$
\xymatrix{
\Gsc\times_\Sc \Tc \ar[r]^{\Id\times \phi} \ar[d]^{p_2}  &\Gsc\times_\Sc\Xf \ar[r]^\mu & \Xf\\
\Tc
}
$$
This procedure yields  a superstack morphism
$$
q \colon \Xf\to [\Xf/\Gsc]
$$
which is interpreted as taking the quotient by the action $\mu$. 

 We   study the main properties of the quotient morphism $q$. Let $\psi\colon \Tc \to [\Xf/\Gsc]$ be a superstack morphism from an $\Sc$-superscheme  $\Tc$. By the 2-Yoneda Lemma \ref{lem:2yoneda}, $\psi$ is given by a principal superscheme $\Gsc$-bundle $\pi\colon\Psc_\psi\to\Tc$ together with a $\Gsc$-equivariant morphism $f_\psi\colon \Psc\to\Xf$.

\begin{prop}\label{prop:quotmorphism} There is a 2-cartesian diagram
$$
\xymatrix{ \Psc\ar[r]^{f_\psi} \ar[d]^\pi & \Xf \ar[d]^q 
\\
\Tc\ar[r]^(.4)\psi & [\Xf/\Gsc]\,.
}
$$
In particular,
the quotient morphism $q \colon \Xf \to [\Xf/\Gsc]$ is surjective and superschematic (Definition \ref{def:repmor}). 
\end{prop}
\begin{proof} We shall use the 2-Yoneda Lemma \ref{lem:2yoneda} when needed without explicit mention. Let $\Vc$ be an $\Sc$-superscheme. If $\gamma\colon \Vc \to \Psc$ is a morphism over $\Sc$, we get morphisms $t=\pi\circ \gamma\colon \Vc \to \Tc$ and $x=f_\phi\circ\gamma\colon \Vc \to \Xf$. Moreover, $\gamma$ induces a section $\sigma_\gamma\colon\Vc \to t^\ast\Psc=\Psc\times_{\pi_\phi,\Tc,t}\Vc$ of  $t^\ast\Psc \to \Vc$. Then, there is a trivialization $\alpha\colon \Gsc\times_\Sc\Vc \iso t^\ast\Psc$ of principal superscheme $\Gsc$-bundles  (see Definition \ref{def:trivial})  and a commutative diagram
$$
\xymatrix{
&&&\Xf
\\
 \Gsc\times_\Sc\Vc \ar[rrru]^{\mu\circ (1\times x)}\ar[r]_(.6)\alpha
\ar[rd]_{p_2} & t^\ast\Psc \ar[r]\ar[d] & \Psc\ar[ru]_{f_\phi}\ar[d]^{\pi_\phi} & 
\\
 & \Vc\ar[r]^{t}& \Tc &
}
$$
where the square is cartesian. We then have an  isomorphism $\alpha\colon \psi(t) \iso q(x)$ in $[\Xf/\Gsc](\Vc)$, so that $(t,x,\alpha)$ is an object of $(\Tc\times_{\psi,[\Xf/\Gsc],q}\Xf)(\Vc)$ (Definition \ref{def:sstackprod}). In this way we obtain a stack morphism $\Psc \to \Tc\times_{\psi,[\Xf/\Gsc],q}\Xf$ which is easily seen to be an isomorphism.
\end{proof}

 A similar argument proves the following result.
 \begin{prop}\label{prop:quotmorphism2} There is a 2-cartesian diagram of $\Sc$-algebraic superstacks:
 $$
 \xymatrix{
 \Gsc\times_\Sc \Xf \ar[r]^{p_1}\ar[d]^{p_2} & \Xf\ar[d]^q
 \\
 \Xf \ar[r]^q & [\Xf/\Gsc]\,.
 }
 $$
 \qed
 \end{prop}

\subsection{Bosonic reduction of a superstack}
\label{sec:bosred}

Here we will see that every superstack has a natural underlying ordinary stack, which  we   call its \emph{bosonic reduction}. Moreover, every stack can be understood as a superstack. This section is mostly based on   \cite{CodViv17}.

As before,  we denote by $\Sf_{et}$ the \'etale site of superschemes. The  \'etale site $(Sch)_{et}$ of ordinary schemes (i.e., even superschemes) is a subsite of $\Sf_{et}$: there is a functor $\iota\colon (Sch)_{et} \hookrightarrow \Sf_{et}$ defined  by associating with every scheme $S$ the scheme itself $\iota(S)$ considered as a superscheme. Usually, we   simply write $S$ for $\iota(S)$.

If $p_\Xf\colon \Xf \to \Sf$ is a superstack, we can consider the  category 
fibered in grupoids $\Xf\times_{p_\Xf,\Sf,\iota}(Sch)_{et}$ (Definition \ref{def:sstackprod}). Taking the second projection
$p_2\colon \Xf\times_{p_\Xf,\Sf,\iota}(Sch)_{et}\to (Sch)_{et}$, we have a category  fibered in grupoids over the \'etale site $(Sch)_{et}$ of schemes. By \cite[Lemma 3.4]{CodViv17},  that is actually a \emph{stack}.
\begin{defin}\label{def:bosred} The bosonic reduction of a superstack  $p_\Xf\colon \Xf \to \Sf$ is
the stack
$$
\Xf_{bos} := \Xf\times_{p_\Xf,\Sf,\iota}(Sch)_{et} \xrightarrow{p_2}(Sch)_{et}\,.
$$
\end{defin} 
 
For every scheme $S$ there is an equivalence of categories $\Xf_{bos}(S)\simeq \Xf(S)$.
If $f\colon \Xf \to \Yf$ is a superstack morphism, for every scheme $S$ we have a functor 
$$
f(S)\colon \Xf(S)\to \Yf(S)\,.
$$
 Using the equivalences of categories $\Xf_{bos}(S)\simeq \Xf(S)$ and $\Yf_{bos}(S)\simeq \Yf(S)$, we get a functor $f_{bos}(S)\simeq f(S)\colon\Xf_{bos}(S)\to \Yf_{bos}(S)$. 

\begin{defin}\label{def:bosredmor} The bosonic reduction of $f\colon \Xf \to \Yf$ is the stack morphism $f_{bos}\colon \Xf_{bos}\to\Yf_{bos}$ defined by the functors $f_{bos}(S)$ for every scheme $S$.
\end{defin}

In the same way that a scheme can be considered as a (even) superscheme, every stack can be considered as a superstack. We can formalize this by first considering the functor $p_+\colon\Sf_{et}\to (Sch)_{et}$ that maps a superscheme $\Sc$ to the even superscheme $\Sc_{ev}=\Sc/\tau$, where $\tau$ is the natural involution that acts as the identity on the even sections and as multiplication by $-1$ on the odd ones. Then, $\Sc_{ev}=(S,\Oc_{\Sc,+})$, where $\Oc_{\Sc,+}$ is the sheaf of the even sections of $\Oc_\Sc$, and there is a projection $p_+\colon\Sc\to\Sc_{ev}$ that is topologically the identity.

For every stack $p_{\Xb}\colon\Xb \to (Sch)_{et}$  we can consider the fiber product  $\Xb\times_{p_{\Xb},(Sch)_{et},p_+} \Sf_{et}$. Taking the second projection we get a category  fibered in grupoids $\Xb\times_{p_{\Xb},(Sch)_{et},p_+} \Sf_{et}\to \Sf_{et}$ over the \'etale site of superschemes. As above,  by \cite[Lemma 3.4]{CodViv17}
this category  is  a superstack.

\begin{defin} The superstack associated to a stack $p_{\Xb}\colon\Xb \to (Sch)_{et}$  is the superstack
$$
\iota(\Xb):=\Xb\times_{p_{\Xb},(Sch)_{et},p_+} \Sf_{et}\xrightarrow{p_2} \Sf_{et}\,.
$$
\end{defin}
For every superscheme $\Sc$ we have an equivalence of categories $\iota(\Xb)(\Sc)\simeq \Xb(\Sc_{ev})$. Moreover, if $f\colon \Xb\to\Yb$ is a stack morphism, for every superscheme $\Sc$ we have a functor
$$
f(\Sc_{ev})\colon \Xb(\Sc_{ev})\to \Yb(\Sc_{ev})\,.
$$
Using the equivalences $\iota(\Xb)(\Sc)\simeq \Xb(\Sc_{ev})$ and $\iota(\Yb)(\Sc)\simeq \Yb(\Sc_{ev})$ we have a functor
$$
\iota(f)(\Sc)\simeq f(\Sc_{ev})\colon \iota(\Xb)(\Sc)\to \iota(\Yb)(\Sc)\,.
$$
and then a superstack morphism
$$
\iota(f)\colon \iota(\Xb)\to\iota(\Yb)\,.
$$
As for schemes, we usually write   $\Xb$ for the superstack $i(\Xb)$.

If $f\colon \iota(\Xb) \to \Yf$ is a superstack morphism, for every scheme $S$ one has  a functor $$f(S)\colon\iota(\Xb)(S)\to \Yf(S).$$  Since $\iota(\Xb)(S)\simeq \Xb(S)$ and $\Yf(S)\simeq \Yf_{bos}(S)$, we get a stack morphism $\bar f\colon \Xb\to\Yf_{bos}$. Conversely, given a stack morphism $\bar f\colon \Xb\to\Yf_{bos}$, for every superscheme $\Sc$ we have a functor $\bar f(\Sc_{ev})\colon \Xb(\Sc_{ev})\to\Yf_{bos}(\Sc_{ev})$, that is, a functor $\iota(\Xb)(\Sc) \to \Yf(\Sc)$, and then a superstack morphism $f\colon\iota(\Xb)\to \Yf$. We then have an equivalence of categories
\begin{equation}\label{eq:adjoint}
\Homst_{SSt}(\iota(\Xb), \Yf)\simeq \Homst_{St}(\Xb, \Yf_{bos})\,.
\end{equation}
where  $\Homst_{SSt}(-,-)$ and $\Homst_{St}(-,-)$ are, respectively, the fibered categories whose objects are    superstack  or stack  morphisms, and whose morphisms are the natural isomorphisms between them. One could  say that equation \eqref{eq:adjoint} establishes a sort of adjunction between $(-)_{bos}$ and $\iota$.

\begin{remark} Actually, to formalize this adjunction, we should consider the 2-categories of superstacks and stacks, and view $(-)_{bos}$ and $\iota$ as 2-functors. We are not going to  follow this approach as we only need equation \eqref{eq:adjoint}.
\end{remark}
In particular, for every superstack $\Xf$, the identity morphism $\Xf_{bos}\iso\Xf_{bos}$ gives rise to a natural superstack morphism $i_\Xf\colon \iota(\Xf_{bos}) \to \Xf$.

\begin{defin}
We say that a superstack $\Xf$ is bosonic when  $i_\Xf\colon \iota(\Xf_{bos}) \to \Xf$ is an isomorphism.
\end{defin}
One also sees that the bosonic reduction $f_{bos}\colon \Xf_{bos}\to\Yf_{bos}$ of a superstack morphism $f\colon \Xf\to\Yf$ is the stack morphism  that corresponds to the composition $f\circ \iota_\Xf \colon \Xf_{bos}\to\Yf$ via equation \eqref{eq:adjoint}.

For every superscheme $\Sc$, the category $\Homst_{SSt}(\Sc, \iota(\Xf_{bos}))$ is equivalent to the subcategory of $\Homst_{SSt}(\Sc, \Xf)$  of the morphisms that factors through $p_+\colon \Sc\to\Sc_{ev}$. By the 2-Yoneda Lemma \ref{lem:2yoneda}, we get:

\begin{prop}\label{prop:prop:bosred} The superstack morphism $i_\Xf\colon i(\Xf_{bos}) \to \Xf$ is a monomorphism, so that it is representable by superspaces.
\qed\end{prop}

\section{Algebraic superstacks}
\label{sec:algsstacks}
In this section we define algebraic superstacks and study their first properties. To introduce them, one must first define and study algebraic superspaces, albeit the latter are a special case of the former.

\subsection{Algebraic superspaces}
\label{subsec:algsstacks}
We start  by defining a class of  properties of superschematic superstack morphisms (Definition \ref{def:repmor}).

\begin{defin}\label{def:morphismsprops} Let $\mathbf P$ be a property of  superscheme  morphisms that is stable under base change (e.g., surjective, proper, smooth, \'etale, unramified --- see Proposition \ref{prop:basech}). A superschematic superstack morphism $\Xf \to \Yf$ has property $\mathbf P$ if for every morphism $\Tc\to\Yf$ from a superscheme, the base-change  superscheme morphism $\Xf\times_\Yf\Tc \to \Tc$ as property $\mathbf P$.
\end{defin}

 We can now give the following:
\begin{defin}\label{def:algsupersps} 
An algebraic superspace is a superspace (i.e., an \'etale sheaf) $\Xf$ such that there exist a superscheme $\Ucal$ and a superschematic surjective \'etale  morphism $p\colon \Ucal\to\Xf$ (Definition \ref{def:morphismsprops}). A morphism of algebraic superspaces is a morphism of superspaces.
We say that $p\colon \Ucal\to\Xf$ is an \emph{\'etale presentation} of $\Xf$.
\end{defin}

There is a useful characterization of the superspaces that are superschemes.
 \begin{prop}\label{prop:sschmchar} Let $\mathbf P$ be  one of the following properties of  superscheme morphisms: open immersion, closed immersion, affine, quasi-affine (Definition \ref{def:quasiaffine}), or locally quasi-finite and separated. 
Let $\Xcal\to \Ycal$ be a surjective \'etale superscheme morphism and $\Zf \to \Ycal$ a sheaf morphism from a superspace. If the fiber product $\Xcal\times_\Ycal \Zf$ is a superscheme and the induced morphism $\Xcal\times_\Ycal \Zf \to \Xcal$ has property $\mathbf P$, then $\Zf$ is a superscheme and $\Zf \to \Ycal$ has property $\mathbf P$.
\end{prop}
\begin{proof} The proof is the same as  that  of \cite[Prop.\! 2.3.17]{Alp24} using descent for fppf  superscheme morphisms   (see Subsection \ref{subsec:fppf}).
\end{proof}

We have already defined when a superstack morphism is representable by sheaves (or superspaces) and by superschemes (Definition \ref{def:repmor}). For the theory of algebraic superstacks, which we will consider in Section \ref{subsec:dmsstacks}, we need representability by algebraic superspaces. This will be the default definition of representability. More precisely:
\begin{defin}\label{def:repres} A morphism $f\colon\Xf\to\Yf$ of superstacks is representable if, for every morphism $\Tc\to\Yf$ from a superscheme $\Tc$, the fiber product superstack $\Xf\times_\Yf \Tc$ is an algebraic superspace. 
\end{defin}

One easily proves the following:
\begin{prop}\label{prop:sch-spc2}
Let $f\colon\Xf\to\Yf$ be a superstack morphism.
\begin{enumerate}
\item If  $f$  is superschematic, then it is representable.
\item If $f$ 
 is representable, then  for every morphism $g\colon \Zf\to\Yf$ from an algebraic superspace $\Zf$, the fiber product $\Xf\times_\Yf \Zf$ is an algebraic superspace. 
\end{enumerate}
\end{prop}

\begin{proof} {The first claim follows from the definitions. For the second claim, note 
that by the definition of an algebraic space, there exist a superscheme $\Ucal$ and a surjective \'etale morphism $p: \Ucal \to \Zf$ .  
Consider the following 2-cartesian diagram of morphisms of superstacks:
$$ \xymatrix{
\Xf \times_{\Yf } \Ucal  \arrow[r]^q \arrow[d] & \Xf  \times_{\Yf } \Zf \arrow[d]  \\
\Ucal  \arrow[r]^p  & \Zf  \arrow[d]^g \\
& \Yf }
$$
Since $f$ is representable and $\Ucal$ is a superscheme, the fiber product $\Xf \times_{\Yf} \Ucal$ is an algebraic superspace, so that there is a surjective \'etale morphism $ h\colon \mathcal V \to \Xf \times_{\Yf} \Ucal$, where $\mathcal V$ is a superscheme.
Moreover, the morphism $q$ is obtained as a base change of the morphism $p$,  hence it is  surjective  and \'etale. Thus the composition $q\circ h$ is the sought-for \'etale presentation.}
\end{proof}

 As happens for algebraic spaces, properties of  superscheme morphisms  that are stable under base change  and composition  and  local on the source for the \'etale topology (Definition \ref{def:localproperties}) can be generalized to morphisms of algebraic superspaces, or, more generally, to representable morphisms of superstacks. The precise statement is as follows:
 
\begin{defin}\label{def:morphismsprops2} Let  $\mathbf P$ be a property of  superscheme morphisms  that is stable under base change  and composition  and  local on the source for the \'etale topology (Definition \ref{def:localproperties}) --- for instance, being surjective, \'etale, smooth or unramified (Proposition \ref{prop:localproperties}). A representable  morphism $\Xf\to\Yf$ of superstacks (Definition \ref{def:repres}) has property  $\mathbf P$ if for  every morphism $\Tc\to\Yf$ from a superscheme $\Tc$ and \'etale presentation $\Ucal \to \Xf\times_\Yf \Tc$ from a superscheme, the composition superscheme morphism  $\Ucal \to \Xf\times_\Yf \Tc \to \Tc$ has property $\mathbf P$. 
\end{defin}
 
 Note  that \cite[Lemma B.12]{CasWis17} is essentially categorical so that it   also holds  true for the \'etale site $\Sf_{et}$ of superschemes.  Then, using Proposition \ref{prop:sch-spc2}, we have:

\begin{prop}\label{prop:diag} Let $\Xf$ a superstack. The following conditions are equivalent:
\begin{enumerate}
\item the diagonal morphism $\Delta_\Xf\colon \Xf \to \Xf\times \Xf$ is representable (resp.\! superschematic);
\item for every superscheme $\Sc$ and objects  $x,y\in\Xf(\Sc)$,  the sheaf $\mathfrak{I}som_{\Xf(\Sc)}(x,y)$ is (representable by) an algebraic superspace (resp.\! a superscheme);
\item for every superscheme $\Sc$ and object  $x\in\Xf(\Sc)$, the morphism $x\colon \Sc\to \Xf$ given by the 2-Yoneda Lemma \ref{lem:2yoneda} is representable (resp.\! superschematic).
\end{enumerate}
\qed
\end{prop}

As in the bosonic case, there is an alternative definition of algebraic superspace.
The notion of \'etale equivalence relation of superschemes is a verbatim generalization  of that of \'etale equivalence relation of schemes, see e.g.~\cite[5.2.1]{Ol16}. Alternatively, it is the   application to superschemes equipped with the \'etale topology of the general definition of $\tau$-equivalence relation given in \cite[Def.\! 5.3]{Knut71}.  We recall it for the reader's convenience.

\begin{defin}\label{def:equivrel} Let $\Sc$ be a superscheme. An \'etale equivalence relation on an $\Sc$-superscheme $\Ucal$  is a  superscheme monomorphism 
$$ 
\Rcal \to \Ucal\times_\Sc\Ucal
$$
such that 
\begin{enumerate}
\item for every $\Sc$-superscheme $\Tc$, $\Rcal(\Tc) \hookrightarrow \Ucal(\Tc) \times \Ucal(\Tc)$  is an equivalence relation of sets on $\Ucal(\Tc)$;
\item the two maps $(q_1,q_2)\colon \Rcal \rightrightarrows \Ucal$ 
induced by the two projections $\Ucal\times_\Sc\Ucal \to \Ucal$ are \'etale.
\end{enumerate}
In the same way one defines fppf or smooth equivalence relations.
\end{defin}
 
\begin{prop}\label{prop:quot}
An algebraic superspace is a superspace that can be expressed as the quotient sheaf of an \'etale equivalence relation of superschemes.
\end{prop}
\begin{proof} 
Let $\Xf$ be an algebraic superspace and $p\colon \Ucal\to\Xf$ a superschematic \'etale presentation. Then $\Rcal:=\Ucal\times_\Xf\Ucal$ is a superscheme and $\Xf$ is the quotient sheaf $\Ucal/\Rcal$ of the \'etale equivalence relation of superschemes $\Rcal\rightrightarrows \Ucal$ defined by the two projections of $\Rcal$ onto $\Ucal$.  

For the converse
we proceed  as in \cite[Thm.\! 3.4.13]{Alp24}. Let $(q_1,q_2)\colon \Rcal \rightrightarrows \Ucal$ be an \'etale equivalence relation of superschemes and let $\Xf$ be the quotient \'etale sheaf considered as a stack. We  divide the proof into steps.

\noindent (1) \emph{The quotient morphism $q\colon \Ucal \to \Xf$ is \'etale, surjective and representable.}

  If $\phi\colon \Tc\to \Xf$ is a morphism from a superscheme, we have to prove that $\Ucal_\Tc=\Ucal\times_{q,\Xf,\phi} \Tc$ is an algebraic superspace and that $q_\Tc\colon\Ucal_\Tc \to \Tc$ is \'etale and surjective. By definition of the quotient sheaf, there exists an \'etale cover $\pi\colon\Tc' \to \Tc$ such that $\Tc'\xrightarrow{\pi} \Tc \xrightarrow{q} \Xf$ lifts to a morphism $\phi'\colon\Tc'\to \Ucal$. 
We can form a commutative cube
$$
 \xymatrix{
 & \Ucal_{\Tc'} \ar[dl] \ar[rr] \ar'[d] [dd] && \Tc' \ar[dl]^{\phi'} \ar [dd]^\pi \\
\Rcal \ar[rr]^(.4){q_1} \ar[dd]^{q_2} && \Ucal \ar[dd]^(.3)q &  \\
&\Ucal_\Tc \ar[dl] \ar'[r]^{q_\Tc} [rr] &&  \Tc \ar[dl]^\phi \\
\Ucal \ar[rr]^q && \Xf & 
}
$$
where the top, bottom, front and  left squares are cartesian. One has that $ \Ucal_{\Tc'}$ is a superscheme, being a product of superscheme morphisms. Since $\pi\colon\Tc'\to\Tc$ is an \'etale and surjective  superscheme morphism, the base-change morphism $\Ucal_{\Tc'}\to \Ucal_\Tc$ is \'etale, surjective, and superschematic, so that $\Ucal_\Tc$ is an algebraic superspace. Moreover, $\Ucal_{\Tc'}\to \Tc'$ is \'etale and surjective because $q_1\colon\Rcal \to \Ucal$ is so. By \'etale descent $\Ucal_\Tc\to \Tc$ is \'etale and surjective as well and we conclude this step.

\noindent (2) \emph{$\Xf$ is an algebraic superspace.}

We only need to prove that the \'etale and surjective morphism $\Ucal \to \Xf$ is superschematic. By Proposition \ref{prop:diag} it is enough to see that the diagonal $\Delta_\Xf\colon \Xf \to \Xf\times \Xf$ is superschematic. Let $\Tc \to \Xf\times\Xf$ be a morphism from a superscheme $\Tc$ and let $\phi'\colon \Tc'\to\Tc$ be an  \'etale cover such that $\Tc' \to \Xf\times\Xf$ lifts to morphism $\Tc'\to\Ucal\times\Ucal$. Define $\spacecal{Q}$ and $\mathfrak Q$ by the 2-cartesian cube
$$
 \xymatrix{
 & \spacecal{Q} \ar[dl] \ar[rr]\ar'[d] [dd] && \Tc' \ar[dl] \ar [dd]^\pi \\
\Rcal \ar[rr]^{\delta}  \ar[dd] && \Ucal\times\Ucal \ar[dd] &  \\
&\mathfrak{Q} \ar[dl] \ar'[r] [rr] &&  \Tc \ar[dl] \\
\Xf \ar[rr]^{\Delta_\Xf} && \Xf\times\Xf & 
}
$$
Note that 
the superscheme morphism $\delta\colon\spacecal R \to  \Ucal\times \Ucal$ is locally quasi-finite and separated:
 it is locally  quasi-finite because its composition with the first projection $\Ucal\times\Ucal\to \Ucal$ is $q_1$, which is locally quasi-finite (see \cite[1.3.4]{EGAIV-IV}), and it is separated because it is a monomorphism and then the diagonal $\Delta_\delta$ is an isomorphism. Thus, $\spacecal{Q} \to \Tc'$ is a locally quasi-finite and separated superscheme morphism, so that $\mathfrak{Q}$ is a superscheme by Proposition \ref{prop:sschmchar}.
\end{proof}

\begin{prop}\label{prop:affinequot}
Let $(s,t)\colon \Rcal \rightrightarrows \Ucal$ be an \'etale equivalence relation of affine superschemes. If the morphisms $s$ and $t$ are finite, then the quotient \'etale sheaf (or algebraic superspace by Proposition \ref{prop:quot}) $\Ucal/\Rcal$ is a superscheme.
\end{prop}
\begin{proof} Let $X$ be the set obtained by identifying in $U$ the points $s(y)$ and $t(y)$ for points $y\in R$, and endow it with the quotient topology. If $p\colon U\to X$ is the quotient map, we have two morphisms of $\Z_2$-graded sheaves $(s^\sharp,t^\sharp)\pi_\ast \Oc_\Ucal \rightrightarrows p_\ast (s_\ast \Oc_\Rcal)\simeq p_\ast (t_\ast \Oc_\Rcal)$. If we denote by $\Oc_\Xcal$ the subsheaf of $\Oc_\Ucal$ defined by the sections $a$ such that $s^\sharp(a)=t^\sharp(a)$, one   sees that $\Xcal:=(X,\Oc_\Xcal)$ is the cokernel of the equivalence relation $(s,t)\colon \Rcal \rightrightarrows \Ucal$  in the category of $\Z_2$-graded ringed spaces \cite[V.\! 1]{SGA3-1}.
We have only to prove that $\Xcal$ is   an affine superscheme.
Let $\Oc_X=\Oc_\Xcal/\Jc_\Xcal$, where $\Jc_\Xcal$ is the ideal generated by the odd sections, be the bosonic reduction of $\Oc_\Xcal$. Now, the bosonic reduction $(s_{bos},t_{bos})\colon R \rightrightarrows U$ is an \'etale equivalence relation of affine schemes and $s_{bos}$, $t_{bos}$ are finite, so that the bosonic reduction $(X,\Oc_X)$ is an affine scheme and   is the quotient of $R \rightrightarrows U$ in the category of schemes by \cite[V. Thm.\! 4.1]{SGA3-1} --- see also \cite[Tag 03BM]{Stacks}. 
It follows that $\Xcal$ is an affine superscheme and that it is the quotient of $(s,t)\colon \Rcal \rightrightarrows \Ucal$ in the category of superschemes. Moreover, $\Xcal$ is also isomorphic to the quotient \'etale sheaf $\Ucal/\Rcal$.
\end{proof}

\begin{prop}\label{prop:represbos} The bosonic reduction of an algebraic superspace is an algebraic space.
\end{prop}
\begin{proof} Let $\Xf$ be an algebraic superspace. Since for every scheme $S$ there is an equivalence of categories $\Xf_{bos}(S)\simeq\Xf(S)$, one sees that $\Xf_{bos}$ is a sheaf on the \'etale site of schemes. Take an \'etale presentation $p\colon\Ucal\to \Xf$ (Definition \ref{def:algsupersps}). Then $p_{bos}\colon U \to \Xf_{bos}$ is schematic and surjective. Moreover, it is \'etale by Corollary \ref{cor:localsmooth}, so that $\Xf_{bos}$ is an algebraic space.
\end{proof}

\begin{corol}\label{cor:represbos} Let $f\colon \Xf\to\Yf$ be a representable superstack morphism. Then, the bosonic reduction $f_{bos}\colon \Xf_{bos}\to\Yf_{bos}$ is a  representable  (by algebraic spaces) stack morphism.
\end{corol}
\begin{proof} Let $g\colon U\to\Yf_{bos}$ be a stack morphism from a scheme $U$, and $i\colon  \Yf_{bos}\to \Yf$ the natural morphism. Then the fiber product $\Xf\times_{f,\Yf,i\circ g}U$ is an algebraic superspace. By Proposition \ref{prop:represbos}, $(\Xf\times_{f,\Yf,i\circ g}U)_{bos}$ is an algebraic space. Since 
$(\Xf\times_{f,\Yf,i\circ g}U)_{bos}\simeq \Xf_{bos}\times_{f_{bos},\Yf_{bos},g} U$, we finish.
\end{proof}

\paragraph{\bf Zariski topology of algebraic superspaces.} Often it is useful to equip algebraic superspaces with the Zariski topology.

\begin{defin}\label{def:Zariski} The   Zariski topology  of an algebraic superspace $\Xf$  is defined as follows:
\begin{enumerate}
\item objects are open immersions $\iota\colon\Uf\hookrightarrow \Xf$ of algebraic superspaces (Definition \ref{def:openclosed});
\item coverings are families $\{\Uf_i\hookrightarrow \Xf\}$ of open immersions such that $\coprod_i\Uf_i\to\Xf$ is surjective (Definition \ref{def:morphismsprops2}).
\end{enumerate}
\end{defin}

A standard argument of superscheme glueing yields:
\begin{prop}\label{prop:recollement} An algebraic superspace $\Xf$ is a superscheme if and only it has   a Zariski covering by superschemes, that is, if there exists a Zariski covering  $\{\Ucal_i\hookrightarrow \Xf\}$ where all the $\Ucal_i$'s are superschemes.
\qed
\end{prop}

Proposition \ref{prop:qfsepdescent}, which is a descent result for superschemes,  can be quite straightforwardly extended to algebraic superpaces. 
\begin{prop}[Locally quasi-finite and separated descent for algebraic superspaces]\label{prop:qfsepdescentalgsup} Let $\{\Sf_i \to \Sf\}$ be an fppf  covering of algebraic superspaces (Definition \ref{def:finitetype}) and $\{\Xf_i/\Sf_i, \varphi_{ij}\}$ descent data relative to $\{\Sf_i \to \Sf\}$. If every morphism $\Xf_i\to\Sf_i$ is locally quasi-finite and separated, the descent data are effective, that is, there exist a   morphism  of algebraic superpaces $\Yf\to \Sf$ and isomorphisms $\Xf_i \iso \Sf_i\times_\Sf \Yf$.
\qed
\end{prop}

\subsection{Deligne--Mumford  and algebraic superstacks}
\label{subsec:dmsstacks}

In this Subsection we give the definitions of the types of superstacks that are most important in geometry, namely, Deligne--Mumford and algebraic superstacks.

\begin{defin}\label{def:DMalgsstack} A Deligne--Mumford  (resp.\! algebraic) superstack is a superstack $\Xf$ such that there exists a superscheme $\Ucal$ and a representable surjective \'etale (resp.\! smooth) morphism $p\colon\Ucal\to \Xf$ (Definitions \ref{def:repres}, \ref{def:morphismsprops2}). A morphism of Deligne--Mumford  or algebraic superstacks is a superstack morphism.
The morphism $p\colon\Ucal\to \Xf$  is called an \'etale (resp.\! smooth) presentation. \end{defin}

In the literature, smooth or \'etale \emph{presentations} are also called smooth or \'etale \emph{atlases}. We will note use this terminology in these notes.

The following result proves the existence of a class of algebraic superstacks that  plays a relevant role in moduli theory.
\begin{prop}   Let $\Gsc\to\Sc$ be an affine and smooth $\Sc$-group superscheme acting on an $\Sc$-algebraic superspace $\Xf$. The quotient superstack $[\Xf/\Gsc]$ (Proposition 
\ref {prop:quotstack}) is an algebraic superstack.
\end{prop}
\begin{proof}  By Proposition \ref{prop:quotmorphism} the quotient  morphism $q\colon \Xf \to [\Xf/\Gsc]$ is superschematic and surjective, and it follows from Definition \ref{def:morphismsprops2} that it is   smooth as well. If $\Ucal \to \Xf$ is an \'etale presentation of $\Xf$, the composition $\Ucal \to \Xf \xrightarrow{q} [\Xf/\Gsc]$ is a smooth (and superschematic) presentation. 
\end{proof}

One  has the following results, whose  proofs are  straightforward.
\begin{prop}\label{prop.firstprop}
Superschemes and  algebraic superspaces are Deligne--Mumford  superstacks.  Deligne--Mumford  superstacks are algebraic superstacks.
Open and closed substacks (Definition \ref{def:openclosed}) of an algebraic superspace (resp.\! Deligne--Mumford  superstack, resp.\! algebraic superstack) are algebraic superspaces (resp.\!  Deligne--Mumford  superstacks, resp.\! algebraic superstacks). Fiber products exist for algebraic superspaces, Deligne--Mumford  superstacks and algebraic superstacks.\qed
\end{prop}

We can   extend Definition \ref{def:repres}:
\begin{defin} A superstack morphism $f\colon \Xf\to \Yf$ is representable by algebraic superstacks (resp.\! Deligne--Mumford superstacks)  if for every algebraic superstack (resp.\! DM superstack) $\Zf$ and every stack morphism $ \Zf\to\Yf$, the fiber product $\Zf\times_\Yf \Xf$ is an algebraic (resp.\! a Deligne--Mumford) superstack.
\end{defin}

\begin{prop}\label{prop:stackfiber}	
If $\Yf$ is an algebraic (resp.\! a Deligne--Mumford) superstack and $f\colon \Xf\to\Yf$ is a superstack morphism representable by algebraic (resp.\! Deligne--Mumford) superstacks, then $\Xf$ is an algebraic (resp.\!  Deligne--Mumford) superstack.
\end{prop}
\begin{proof} 
For the algebraic superstack case, 
let $\mathcal V \to \Yf$ be a smooth presentation and take  the fiber product
$ \Xf_{\mathcal V}  = \Xf\times_\Yf {\mathcal V}$.  Note that the projection
$p_1\colon \Xf_{\mathcal V}\to\Xf$ by base change is smooth and surjective. 
As $f$ is representable by algebraic superstacks,
$ \Xf_{\mathcal V} $ is an algebraic superstack, i.e., it has a smooth presentation $p\colon \Ucal \to \Xf_{\mathcal V}$. 
The composition $p_1\circ p \colon \Ucal \to \Xf$ provides the required smooth presentation.
For the Deligne--Mumford case, just replace ``smooth'' with ``\'etale.''
\end{proof}

\begin{remark}\label{rem:algssp} Clearly, algebraic superspaces are algebraic stacks  and superspaces.  We will see later (Proposition \ref{prop:algsp}) that an algebraic stack that is a superspace is actually an algebraic superspace.
\end{remark}

\begin{prop} The bosonic reduction of a Deligne--Mumford (resp.\! algebraic) superstack is a Deligne--Mumford (resp.\! algebraic) stack.
\end{prop}
\begin{proof} Let $p\colon \Ucal\to \Xf$ be a smooth (resp.\! \'etale) presentation, that is, a representable, surjective and smooth (resp.\! \'etale) superscheme morphism. By Proposition \ref{prop:represbos} and Corollary \ref{cor:localsmooth} $p_{bos}\colon U \to \Xf_{bos}$ is representable, surjective and smooth (resp.\! \'etale).
\end{proof}

 Let $\mathbf P$ be  a property
of superschemes which is \'etale (resp.\! smooth) local, that is,  if $\Xcal \to \Ycal$ is an \'etale (resp.\! 
smooth) surjection of superschemes, then $\Xcal$ has $\mathbf P$ if and only if $\Ycal$ has $\mathbf P$. 
One can mimic  \cite[Def.\! 3.3.7]{Alp24} and give the following definition. 
\begin{defin}\label{def:locnoeth}  An algebraic (resp.\! Deligne--Mumford) superstack  $\Xf$ 
has property $\mathbf P$ if for a smooth (resp.\! \'etale) presentation $\Ucal\to \Xcal$ (equivalently, for all presentations) the superscheme
$\Ucal$ has $\mathbf P$. This allows us to define locally noetherian algebraic superstacks.
\end{defin}

\subsection{Properties of algebraic stack morphisms}\label{subsec:properties}

 We can generalize Definition \ref{def:morphismsprops2} to define properties of morphisms of algebraic superstacks. 
 
\begin{defin}\label{def:morphismsprops3} Let $\mathbf P$ be a property of superscheme morphisms stable by composition and base change and \'etale (resp.\! smooth) local on the target and on the source (Definition \ref{def:localproperties}). A morphism $f\colon \Xf \to \Yf$ of Deligne--Mumford (resp.\! algebraic) superstacks has the property $\mathbf P$ if there is an \'etale (resp.\! smooth) presentation $p\colon \Vc \to \Yf$ and an \'etale  (resp.\! smooth) presentation $\Ucal \to \Xf\times_{f,\Yf,p}\Vc$ such that the composition $\Ucal \to \Vc$ has $\mathbf P$. This is equivalent to claim that  $\Ucal \to \Vc$ has $\mathbf P$ for any such presentations.\footnote{The proof is similar to the corresponding one in the bosonic setting, see \cite[Lemma 6.24]{CasWis17}.}
Using Proposition \ref{prop:localproperties} one can then define the following properties:
\begin{enumerate}
\item for a morphism of algebraic superstacks: being surjective, flat, locally of finite type, locally of finite presentation, and smooth;
\item for a morphism of Deligne--Mumford superstacks:  those in (1), and moreover being locally quasi-finite, \'etale, and unramified.
\end{enumerate}
\end{defin}

 Proceeding as in \cite[Lemma 6.25]{CasWis17}, which is completely formal, we have:
\begin{prop}\label{prop:morphismsprops3}
The properties $\mathbf P$ of  morphisms $f\colon \Xf \to \Yf$ of Deligne--Mumford (resp.\! algebraic) superstacks defined in Definition \ref{def:morphismsprops3} are stable by composition and base change. Moreover they are \'etale (resp.\! smooth) local in the target and the source, in the following sense: if $\pi\colon\Ycal' \to \Ycal$ is a surjective and \'etale (resp.\! smooth) superstack morphism and $f'\colon \Xcal'=\Xcal\times_{f,\Ycal,\pi}\Ycal' \to \Ycal'$ is the pullback morphism, then $f$ has $\mathbf P$ if and only if $f'$ has $\mathbf P$. 
\end{prop}

 As a consequence of Corollary \ref{cor:localsectsmooth} and Definition \ref{def:morphismsprops3}, a smooth morphism of algebraic superstacks,  locally in the \'etale topology, always has sections.
 
\begin{prop}\label{prop:localsectsmooth} 
Let $f\colon \Xf \to \Yf$ be a smooth morphism of algebraic superstacks. There exist a superscheme $\Zc$ and an \'etale covering $\Zc\to \Yf$ such that $f_\Zc\colon \Xf\times_\Yf\Zc \to \Zc$ has a section.
\qed
\end{prop}

\begin{prop}\label{prop:represdiag} (Representability of the diagonal)
\begin{enumerate}
\item
The diagonal of an algebraic superspace is superschematic;
\item
the diagonal of an algebraic superstack is representable.
\end{enumerate}
\end{prop}
\begin{proof} {
The proof is a straightforward adaptation of    that of \cite[Theorem 3.2.1]{Alp24}.}

(1) We need to show that, for every superstack morphism $g\colon \Tc \to \Xf\times\Xf$,  
the superstack  $\mathfrak Q:= \Tc\times_{g,\Xf\times\Xf,\Delta}\Xf$ 
is actually a superscheme. Let $p\colon \Ucal \to \Xf$ be an \'etale
presentation, and define $\spacecal R = \Ucal \times_\Xf \Ucal$; this  is a superscheme, since $p\colon \Ucal \to \Xf$ is superschematic.  
Moreover, $\Ucal\times\Ucal\to \Xf\times \Xf$ is surjective, \'etale and superschematic. 

Consider the cube diagram
\begin{equation}\label{eq:cube} \xymatrix{
 & \spacecal Q' \ar[dl] \ar[rr] \ar'[d] [dd] && \spacecal R \ar[dl] \ar [dd] \\
\Tc' \ar[rr] \ar[dd] && \Ucal\times \Ucal \ar[dd] &  \\
& \mathfrak Q \ar[dl] \ar'[r] [rr] &&  \Xf \ar[dl] \\
\Tc \ar[rr] && \Xf\times\Xf & 
}
\end{equation}
where all faces are 2-cartesian. 
$\spacecal Q'$ is a superscheme as it is a fiber product of superschemes, while at this stage $\mathfrak Q$ is a superstack, which we need to prove to be a superscheme.

A crucial remark   is that the superscheme morphism $\delta\colon\spacecal R \to  \Ucal\times \Ucal$ is locally quasi-finite and separated: it is locally of finite type because its composition with the first projection $\Ucal\times\Ucal\to \Ucal$ is $q_1$, which is of finite type, see \cite[1.3.1]{EGAIV-IV} ---  and  it is separated because it is a monomorphism and then the diagonal $\Delta_\delta$ is an isomorphism. It follows that it is quasi-finite by Proposition \ref{prop:mono}, and then the same is true for $\spacecal Q' \to \Tc'$  by base change.  This allows us to apply the fppf descent for locally quasi-finite separated  superscheme morphisms  (Proposition \ref{prop:qfsepdescent})  to conclude that $\mathfrak Q$ is a superscheme.

{(2) In this case,   with $p\colon \Ucal \to \Xf$   a smooth
presentation, $\mathfrak R = \Ucal \times_\Xf \Ucal$ is an algebraic superspace but need no longer be a superscheme, so we have to modify the argument.
Let   $\mathfrak T''$ be the base change of $\Tc\to \Xf\times\Xf$ along  $\Ucal\times\Ucal\to \Xf\times \Xf$; it is an algebraic superspace, and the morphism $\mathfrak T''\to\Tc$ is surjective and smooth. If $\spacecal V \to \mathfrak T''$ is an \'etale presentation, 
then $\spacecal V\to \Tc$ is a surjective smooth superscheme morphism, which up to an \'etale cover $\Tc'\to\Tc$ has a section,
that is, there is a morphism $\Tc'\to \spacecal V\times_\Tc\Tc'$ which is a section of the projection $ \spacecal V\times_\Tc\Tc'\to\Tc'$
  (Corollary \ref{cor:localsectsmooth}). The composition $$\Tc'\to\spacecal V \to \mathfrak T'' \to \Ucal\times\Ucal$$
fits into the diagram
\begin{equation}\label{eq:cube2} \xymatrix{
 & \mathfrak Q' \ar[dl] \ar[rr] \ar'[d] [dd] && \mathfrak R \ar[dl] \ar [dd] \\
\Tc' \ar[rr] \ar[dd] && \Ucal\times \Ucal \ar[dd] &  \\
& \mathfrak Q \ar[dl] \ar'[r] [rr] &&  \Xf \ar[dl] \\
\Tc \ar[rr] && \Xf\times\Xf & 
}
 \end{equation}
 where $\mathfrak Q'$ is defined as the fiber product $\Tc'\times_{\Ucal\times\Ucal} \mathfrak R$, so that it is  an algebraic superspace.
Note that the front and rear faces may not be 2-cartesian.}

{By equation \eqref{eq:diagonal}  we identify $\mathfrak Q$ with $\mathfrak Isom_{\Xf(\Tc)}(a,a)$, where $a\colon\Tc\to\Xf$, and then, $\mathfrak Q$ is a sheaf. The morphism $\mathfrak Q'\to\mathfrak Q$ is surjective, \'etale and schematic, as it is a base change of $\Tc'\to\Tc$. If $\spacecal W \to \mathfrak Q'$ is an \'etale presentation, the composition $\spacecal W \to \mathfrak Q'\to\mathfrak Q$ is an \'etale presentation of $\mathfrak Q$, so that $\mathfrak Q$ is an algebraic superspace.}
 \end{proof}

 By the same token,  one has:

\begin{corol}\label{cor:represdiag} The diagonal $\Delta_f\colon \Xf \to \Xf\times_\Yf \Xf$ of a representable morphism $f\colon\Xf \to\Yf$ of algebraic superstacks is superschematic.
\qed
\end{corol}

Another consequence is that, as in the bosonic case, there is an alternative definition of algebraic or Deligne--Mumford superstack:
\begin{corol}\label{cor:represdiag2} A superstack $\Xf$ is algebraic (resp.\! Deligne--Mumford) if  and only if  it fulfills the following conditions:
\begin{enumerate}
\item there exists a surjective smooth (resp.\! \'etale) morphism $\Ucal \to \Xf$;
\item the diagonal is representable, or, equivalently, for every superscheme $\Sc$ and objects $x,y\in\Xf(\Sc)$, the sheaf $\mathfrak{I}som_{\Xf(\Sc)}(x,y)$ is an algebraic superspace (Proposition \ref{prop:diag}).
\end{enumerate}
\end{corol}
\begin{proof} Assume that $\Xf$ fulfills the two properties of the statement. Then $\Ucal \to \Xf$ is representable by Proposition \ref{prop:diag}, so that it is a smooth (resp.\! \'etale) presentation of $\Xf$ and we finish (Definition \ref {def:DMalgsstack}). The converse follows directly from Proposition \ref{prop:represdiag}. 
\end{proof}

The automorphisms of the points of a stack (stabilizers) are an important tool to study stacks, so we pause to define them. Stabilizers are going to be \emph{group algebraic superspaces,} so we need the definition of such objects.

\begin{defin}\label{def:groupalgss}
A \emph{group algebraic superspace} is a group object in the category of algebraic superspaces.
\end{defin}

\begin{rem}\label{rem:groupalgss}
Classical results of Raynaud \cite{Raynaud-FaisceauxAmples}
and Anantharaman \cite{Anantha} imply that every quasi-separated group algebraic space locally of finite type over a field, or over a normal noetherian base scheme of dimension at most $1$,
 is automatically a scheme (see also \cite[Tag 0B8G]{Stacks}). Thus we may expect that
  to obtain a group algebraic superspace that is \emph{not} a superscheme, one must   either drop quasi-separatedness or the finite-type condition, or consider a base (super)scheme whose (even) dimension is greater than $1$. We give a simple example of the first kind, which is a super version of a standard example in the ordinary bosonic case \cite[Tag 02Z7]{Stacks} (see also \cite[Tag 06E9]{Stacks}).

Let $k$ be a field of characteristic $0$, and let $\Gs_{a,k}=\As^{1|1}_k$ be the additive supergroup over $k$ (Example \ref{ex:ga}). The constant group superscheme $\Z_k=\coprod_{n\in\Z}\Spec k$ is a closed subgroup superscheme of $\Gs_{a,k}$ via $n\mapsto (n,0)$, and it acts on $\Gs_{a,k}$ by translating the even coordinate, $n \cdot (x, \theta) = (x + n, \theta)$. Let $\tilde{\Gs}_{k}=\Gs_{a,k}/ \Z_k$ be the quotient sheaf for the \'etale topology; being the quotient of a sheaf of abelian groups by a subgroup sheaf, it is a sheaf of abelian groups. We prove that $\tilde{\Gs}_{k}$ is an algebraic superspace, and then a group algebraic superspace (Definition \ref{def:groupalgss}), by showing that it is the quotient of an \'etale equivalence relation of superschemes (Proposition \ref{prop:quot}). The relevant equivalence relation is
 \[
\Rcal = \Gs_{a,k}  \times_{k} \Z_k  = \coprod_{n \in \Z} \Rcal_n\,,\qquad s(g,n)=g\,,\quad t(g,n)=g+n\,,
\]
 where each component $\Rcal_n$ is isomorphic to $\As^{1|1}_k$. Since $k$ has characteristic $0$, the action is free: if $T$ is a superscheme and $(g,n)\in \Rcal(T)$, with $n$ a locally constant $\Z$-valued function on $T$, then $n$ is determined by $t(g,n)-s(g,n)=n\cdot 1_{\Oc_T}$; hence $(s,t)\colon \Rcal\to\Gs_{a,k}\times_k\Gs_{a,k}$ is a monomorphism, and  $s, t \colon \Rcal \rightrightarrows  \Gs_{a,k}$  is an equivalence relation. The restrictions $s_n=s_{\vert \Rcal_n}$ and $t_n=t_{\vert \Rcal_n}$ are, respectively, the identity of $\As^{1|1}_k$ and the translation $(x,\theta)\mapsto (x+n,\theta)$. Both are isomorphisms, hence \'etale. Since the disjoint union of \'etale morphisms is \'etale, $s$ and $t$ are \'etale, so that  $\Rcal \rightrightarrows \Gs_{a,k}$ is an \'etale   equivalence relation and
 $\tilde{\Gs}_{k}$ is an algebraic superspace by Proposition \ref{prop:quot}.

Now, the bosonic reduction of $\tilde{\Gs}_k$ is an algebraic space (Proposition \ref{prop:represbos}), with \'etale presentation $\As^1_k=(\Gs_{a,k})_{bos}\to(\tilde{\Gs}_k)_{bos}$; since  $\Xf_{bos}(S)\simeq\Xf(S)$ for every scheme $S$, one has $\As^1_k\times_{(\tilde\Gs_k)_{bos}}\As^1_k=\Rcal_{bos}=\As^1_k\times_k\Z_k$, so that $(\tilde{\Gs}_k)_{bos}$ is the quotient $\As^1_k/\Z$ of the affine line by the translation action of $\Z$, i.e., the classical group algebraic space  of \cite[Tag 02Z7]{Stacks}. It is not a scheme. Indeed, the preimage in $\As^1_k$ of a nonempty open subspace of $\As^1_k/\Z$ is a nonempty $\Z$-stable open subset, whose complement is a finite $\Z$-stable set of closed points; since the $\Z$-orbits are infinite, the complement is empty, so that every nonempty open subspace of $\As^1_k/\Z$ is the whole space. If $\As^1_k/\Z$ were a scheme, it would then be affine, equal to $\Spec \Gamma(\As^1_k/\Z,\Oc)=\Spec k[x]^{\Z}=\Spec k$, and $\As^1_k\to \Spec k$ would be \'etale, which is absurd. (The algebraic space $\As^1_k/\Z$ is not even quasi-separated, since $\As^1_k\times_{\As^1_k/\Z}\As^1_k=\As^1_k\times_k\Z_k \to \As^1_k\times_k\As^1_k$ is not quasi-compact, cf.~\cite[Tag 06E9]{Stacks}; for the same reason $\tilde\Gs_k$ is not quasi-separated.)
Thus, $\tilde{\Gs}_{k}$  is a group algebraic superspace that is not a superscheme: if it were, its bosonic reduction would be a scheme, which is not the case. \end{rem}

\begin{defin}\label{def:stabilizer} Let $k$ be a field and  $x\colon \Spec k \to \Xf$ a $k$-valued point of an algebraic superstack $\Xf$. The stabilizer of $x$ is the sheaf $\mathfrak{A}ut_{\Xf(k)}(x):=\mathfrak{I}som_{\Xf(k)}(x,x)$.\end{defin}
By equation \eqref{eq:diagonal}, there is a 2-cartesian diagram
\begin{equation}\label{eq:stabilizer2}
\xymatrix{ 
\mathfrak{A}ut_{\Xf(k)}(x) \ar[r]\ar[d] & \Spec(k) \ar[d]^{(x,x)} \\
\Xf\ar[r]^{\Delta_\Xf}& \Xf\times\Xf
}
\end{equation}
so that $\mathfrak{A}ut_{\Xf(k)}(x)$ is a group algebraic superspace (see Definition \ref{def:groupalgss}) by the representability of the diagonal (Proposition \ref{prop:represdiag}). By the same reference, the stabilizers of the field-valued points of an algebraic superspace are group superschemes or \emph{supergroups}.

\subsection{Topological properties}

{The topological space $\Xf$ of a superstack is defined   in analogy with the definition given for stacks, see for instance \cite[Def.\! 9.4]{CasWis17} --- namely, it is the topological space}
$$
|\Xf | =\varinjlim_{\Sc\to\Xf} |\Sc |
$$
where the limit is taken over all morphisms from superschemes $\Sc$ and $|\Sc |$ denotes the  topological space underlying $\Sc$.

A morphism $f\colon \Xf \to \Yf$ of superstacks induces a map $|f|\colon |\Xf|\to|\Yf|$.

 As  the topology of a superscheme is all contained in the underlying bosonic object, the same happens for superstacks.
\begin{prop}\label{prop:points} One has
$|\Xf | =\displaystyle\varinjlim_{S\to\Xf} |S|$, where the limit is taken over all morphisms from schemes $S$, so that
$$
|\Xf | \simeq  | \Xf_{bos} |.
$$
\qed
\end{prop} 
 Let us write $\Xb=\Xf_{bos}$. Following \cite[Rem.\! 9.5]{CasWis17}, the topological space $|\Xb |$ can be also described as follows:
the points of $|\Xb |$ are field-valued morphisms $x\colon \Spec k\to \Xb$, where two morphisms $x_1\colon \Spec k_1\to \Xb$ and $x_2\colon \Spec k_2\to \Xb$ are identified if there exist field extensions $k_1\to k_3$, $k_2\to k_3$ such that $x_{1\vert \Spec k_3}$ and $x_{2\vert \Spec k_3}$ are isomorphic in $\Xb(k_3)$. This is the definition of the topological space underlying a stack as in 
\cite[Def.~3.3.17]{Alp24} or in the Stacks Project \cite[Tags  
  04XE, 04XG, 04XL]{Stacks}. A subset $U\subseteq |\Xb |$ is open if there exists an open substack $\Ub\hookrightarrow\Xb$ such that $U=|\Ub |$.

 Since the field-valued points of a superstack $\Xf$ are the same as the field-valued points of its bosonic reduction $\Xf_{bos}$ (see equation \eqref{eq:adjoint}), we have the same characterization of the topological space associated to a superstack, namely, 
\begin{prop} The points of $|\Xf |$ are the classes of field-valued morphisms $x\colon \Spec k\to \Xf$, where two morphisms $x_1\colon \Spec k_1\to \Xf$ and $x_2\colon \Spec k_2\to \Xf$ are identified if there exist field extensions $k_1\to k_3$, $k_2\to k_3$ such that $x_{1\vert \Spec k_3}$ and $x_{2\vert \Spec k_3}$ are isomorphic in $\Xf(\Spec k_3)$. A subset $U\subseteq |\Xf |$ is open if  there exists an open substack $\Uf\hookrightarrow\Xf$ such that $U=|\Uf |$.
\qed
\end{prop}

\begin{defin}\label{def:locfinitepts} A point of a superstack $\Xf$ is of finite type if it has a representing morphism $x\colon \Spec k \to \Xf$ of finite type. Equivalently, if it is of finite type as a point of the bosonic reduction $\Xf_{bos}$. We denote by $|\Xf |_{ft}=|\Xf_{bos}|_{ft}$ (Proposition \ref{prop:points}) the set of its finite type points.
\end{defin}

\begin{prop}\label{prop:ftdense} Let $\Xf$  be  an algebraic stack and  $Z$   a  closed subset of $|\Xf|$. Then $Z\cap |X_f|_{ft}$ is dense in $Z$.
\end{prop}
\begin{proof} Since $|\Xf | \simeq  | \Xf_{bos} |$, we can apply directly the result holding in the case of stacks, see \cite[Tag 06G2]{Stacks}.
\end{proof}

Using \cite[Tag 04XL]{Stacks}, one sees the following property.
\begin{prop}\label{prop:continuous} For every morphism $f\colon \Xf \to \Yf$ of superstacks one has 
$|f|=|f_{bos}|$ as  a  map from $ | \Xf_{bos} |$ to $ | \Yf_{bos} |$, so that  $|f|$ is continuous.
\qed
\end{prop}

\begin{defin}\label{def:univclosed}
A morphism $f\colon \Xf \to \Yf$ of algebraic superstacks is closed (resp.\! open) if $|f|$ is a closed (resp.\! open) map. 
A morphism $f\colon \Xf \to \Yf$ of algebraic superstacks is \emph{universally closed} if for every stack morphism $\Zf \to \Yf$, the morphism $\Xf\times_\Yf \Zf \to \Zf$ is closed.
\end{defin} 
 
 \begin{defin}\label{def:quasi-compact}
An algebraic superstack $\Xf$ is quasi-compact, connected,
or irreducible if the topological space $|\Xf |$ is so. 
\end{defin}

 \begin{defin}\label{def:finitetype} Let $f\colon \Xf \to \Yf$ be a morphism of algebraic superstacks.
\begin{enumerate}
\item $f$ is quasi-compact if for every affine superscheme $\SSpec\As$ and every morphism $\SSpec \As \to \Yf$, the fiber product $\Xf\times_\Yf\SSpec \As$ is quasi-compact (Definition \ref{def:quasi-compact}).
\item $f$ is  of finite type if it is locally of finite type  (Definition \ref{def:morphismsprops3}) and quasi-compact.
\item $f$ is fppf if it is faithfully flat, locally of finite presentation  (Definition \ref{def:morphismsprops3})  and quasi-compact.
\end{enumerate}
\end{defin}

 We have seen that the topological properties of a superstack are the same as those of its bosonic reduction. Moreover,   a superscheme morphism is locally of finite type if and only if the bosonic reduction has the same property; this follows from the fact that for a superscheme, the ideal sheaf generated by the odd sections is locally of finite type.  Then we have:    
\begin{prop}\label{prop:univclosed}  A morphism $f\colon \Xf \to \Yf$ of algebraic superstacks is closed, open, universally closed, quasi-compact or of finite type if and only if the bosonic reduction $f_{bos}\colon \Xf_{bos}\to\Yf_{bos}$ has the same property.
\end{prop}
  
Moreover, one has:
\begin{prop}\label{prop:open}
Every flat  superstack morphism $f\colon \Xf \to \Yf$,   locally of finite presentation, where $\Xf$ is an algebraic superspace, is open.
\end{prop}
\begin{proof} The result follows from \cite[Tag  04XL]{Stacks}, using Proposition \ref{prop:represbos}.
\end{proof}

\subsection{Dimension of a superstack}
\label{subsec:dimsstack}  

In this section we define the notion of dimension of an algebraic superspace and of  a superstack. 
Quite naturally, the dimension will be a pair of numbers (even and odd dimension), for which we use the notation
 $m|n$, where $m,n\in \Z\cup\infty$.  All algebraic superspaces, superstacks and superschemes considered in this  section will be locally Noetherian.
 The definition of dimension of a superstack relies on the notion of dimension of a superscheme; actually for the latter there are (at least) two possibilities, as
 we shall see in Appendix \ref{sec:app}, so that also for superstacks there are two choices. In this subsection, the term ``odd-dimension'' of a superscheme will refer to any of them, chosen consistently from the beginning. Also some properties rely on the fact that for a smooth superscheme morphism both notions of odd dimension coincide by Proposition \ref{prop:smoothsplit3}.

Let $\Xf$ be an algebraic superspace over a superscheme $\Sc$.
The definition of the dimension of  $\Xf$  is a straightforward adaptation of the definition for algebraic spaces (see e.g.\! \cite[Defs.\! 66.9.1, 66.9.2]{Stacks}).
\begin{defin} \label{dimalgsp}\
 \begin{enumerate}\item  The dimension of $\Xf$ at a point $x\in\vert\Xf\vert$
is the pair $\dim_x\Xf= \dim_u \Ucal$ where
$p\colon\Ucal\to\Xf$ is an \'etale presentation of $\Xcal$,  with $\Ucal$ a superscheme, and $p(u) = x$.
\item  The dimension of $\Xf$ is the pair
$$ \sup_{x\in\vert\Xf\vert} \operatorname{even} 
\dim_x\Xf  \, \bigm\vert   \sup_{x\in\vert\Xf\vert} \operatorname{odd} 
\dim_x\Xf\,.$$
\end{enumerate}
 \end{defin} 
This definition is well-posed as different choices of $\Ucal$ and $p$ in part 1 produce the same dimension. 
 Indeed, if $p'\colon \Ucal\to \Xf$ is another \'etale presentation, the fiber product $\tilde p\colon\Ucal\times_\Xf \Ucal' \to \Xf $ is an \'etale presentation of $\Xf$. This follows from the 2-cartesian diagram:
$$
\xymatrix{ \Ucal\times_\Xf \Ucal' \ar[r]^(.6){p'_\Ucal}\ar[d]_{p_{\Ucal'}} \ar[rd]^{\tilde p} & \Ucal\ar[d]^p \\
\Ucal' \ar[r]_{p'} & \Xf
}
$$
Here $\Ucal\times_\Xf \Ucal'$ is a superscheme because $p$ is superschematic and $\Ucal'$ is a superscheme, $p'_{\Ucal}$ is \'etale and surjective because it is obtained from $p'$ by base change and finally $\tilde p$ is superschematic, surjective and \'etale because is the composition of two morphisms with these properties.

If $u'\in \vert \Ucal'\vert$ is a point with $p'(u')=x$, we have 
$$
\dim_{u'}\Ucal'=\dim_{u\times u'}\Ucal\times_\Xf \Ucal' = \dim_{u}\Ucal\,,
$$ because $p_{\Ucal'}$ and $p'_{\Ucal}$ are \'etale  superscheme morphisms. This holds true for both choices of the notion of odd dimension, as \'etale superscheme morphisms have relative dimension $0\vert 0$ for both definitions (see Definition \ref{def:etale} and Proposition \ref{prop:smoothsplit3}).

\begin{defin} Let $f\colon\mathfrak U \to \Xf$ be a representable morphism of algebraic stacks, locally of finite type. 
Fix a morphism $\Spec k \to\Xf$ representing $f(u)$.  The relative dimension of $f$ at $u$ is
$$ \dim_u f = \dim_{u'} (\mathfrak U \times_\Xf \Spec k)$$
where $u'$ is any point  in $\vert \mathfrak U \times_\Xf \Spec k\vert$ that projects to $u$.
\end{defin}
This definition is well posed for the following reasons:
\begin{itemize} \item the point $u'$ exists because the natural set-theoretic map
$$  \vert \mathfrak U \times_\Xf \Spec k\vert \to \vert \mathfrak U \vert \times_{\vert\Xf\vert}\Spec k$$
is surjective; a proof of this fact works as in \cite[Lemma 100.4.3]{Stacks}.\footnote{For the reader's convenience, we provide here a proof   in greater generality., i.e., we prove that if
$$\xymatrix{ \Zf\times_\Yf\Xf \ar[d]\ar[r] & \Xf\ar[d] \\ \Zf\ar[r] & \Yf}$$
is a 2-cartesian diagram of morphisms of algebraic superstacks, then the naturally induced map of topological spaces
$ \vert  \Zf\times_\Yf\Xf \vert \to \vert\Zf\vert\times_{\vert\Yf\vert}\vert\Xf\vert $ is surjective.
We utilize the definition of the underlying topological spaces in terms of morphisms from spectra. Now if $k_1$ and $k_2$ are fields,     morphisms $\Spec k_1\to\Xf$ and $\Spec k_2\to\Zf$  yield the same point in $\vert\Yf\vert$ if and only if  there is a common extension $k_3$ of $k_1$ and $k_2$ such that $\Spec k_3\to\Spec k_1 \to \Xf \to \Yf$ and $\Spec k_3\to\Spec k_2 \to \Zf \to \Yf$ are 2-isomorphic, which is the condition for the existence of a morphism $\Spec k_3 \to  \Zf\times_\Yf\Xf$.
}
\item $ \mathfrak U \times_\Xf \Spec k $ is an algebraic superspace because $f$ is representable, so that its dimension is defined as in Definition \ref{dimalgsp};
\item the definition does not depend on the choices made because the relative dimension of a superscheme morphism  which is locally of finite type   is invariant under base change.
\end{itemize}
The following Lemma follows quite straightforwardly from the previous definitions.
\begin{lemma} \label{lem:itwillbeok} Let $f\colon\mathfrak U\to \mathfrak V$ be a morphism of algebraic superspaces, locally of finite type.
Then for every $u \in \vert \mathfrak U \vert $
$$ \dim_u \mathfrak U = \dim_{f(u) }\mathfrak  V + \dim_u f.$$
\end{lemma}
We can eventually define the dimension of a (locally Noetherian) algebraic superstack.
\begin{defin} \ 
\begin{enumerate}
\item Let $\Xf$ be a locally Noetherian algebraic superstack,
and let $x\in \vert\Xf\vert$. The dimension of $\Xf$ at $x$ is defined as 
$$ \dim_x\Xf = \dim_u \, \mathcal U  - \dim_u p$$
where $p\colon \mathcal U \to \Xf$ is a smooth presentation from a superscheme $\mathcal U$,
and $u\in \mathcal U$ is such that $p(u)=x$.
\item The dimension of $\Xf$ is defined as in part 2 of Definition \ref{dimalgsp}.
\end{enumerate}
\end{defin}
The independence of this definition from the choices involved in it follows from Lemma \ref{lem:itwillbeok}, proceeding as in the case of algebraic superspaces.

\begin{rem} If $\Xf$ is of finite type over a field, or over a noetherian ring of finite Krull dimension,
both even and odd dimensions are finite.
\end{rem} 

\section{Properties of morphisms}
We study   some more properties of morphisms of superstacks. Some of these properties will be used in Section \ref{sec:charstack} to characterize Deligne--Mumford superstacks and algebraic superspaces.

\subsection{Quasi-finite morphisms of algebraic superstacks}

Using Definition \ref{def:morphismsprops3} we can extend the notion of locally quasi-finite morphism to morphisms of algebraic spaces. Note that, according to Definition \ref{def:morphismsprops3},  a morphism $f\colon \Xf \to \Yf$ of algebraic superspaces is locally quasi-finite if for every \'etale presentations $\Ucal\to \Yf$ and $\Vc\to\Xf\times_\Yf\Ucal$ the superscheme morphism $\Vc \to \Ucal$ is locally quasi-finite (Definition \ref{def:qf}).
We can   extend that definition to morphisms of  algebraic superstacks as in  \cite[Def.\! 3.3.34]{Alp24}). We first define this for representable morphism of algebraic superstacks and then extend it to general morphisms.

\begin{defin}\label{def:qfsstacks}
\ 
\begin{enumerate}
\item A representable morphism $f\colon \Xf\to\Yf$ of algebraic superstacks is locally quasi-finite if it is locally of finite type  (Definition \ref{def:morphismsprops3}), and for every morphism $\Ucal\to \Yf$ from a superscheme, the induced morphism $f_\Ucal\colon \Xf\times_\Yf \Ucal \to \Ucal$ of algebraic superspaces  is locally quasi-finite.
\item A  morphism $f\colon \Xf\to\Yf$ of algebraic superstacks is locally quasi-finite if it is locally of finite type, the diagonal morphism $\Delta_f\colon \Xf \to \Xf\times_\Yf\Xf$ ---  that is representable by Proposition \ref{prop:represdiag} --- is locally quasi-finite, and for every field-valued point $\Spec k \to \Yf$, the topological space $|\Xf\times_\Yf \Spec k|$ is discrete.
\item A morphism of algebraic superstacks is quasi-finite if it is locally quasi-finite and quasi-compact (Definition \ref{def:finitetype}).
\end{enumerate}
\end{defin}

 The following statement provides some examples of locally quasi-finite morphisms.

\begin{prop}\label{prop:locfinite} \ 
\begin{itemize}
\item An \'etale or unramified representable morphism of algebraic stacks is locally quasi-finite;
\item a representable injective morphism  of algebraic stacks that is locally of finite type    is locally quasi-finite.
\end{itemize}
\end{prop}
\begin{proof} (1) Follows from Corollary \ref{cor:unramprop}.

(2) One   reduces to the case of  an injective morphism $f\colon \Xf\to\Ucal$,  locally of finite type, from an algebraic superspace to a superscheme. We have then to prove that if $p\colon \Vc \to \Xf$ is an \'etale (superschematic) presentation of $\Xf$, the composition $f\circ p\colon \Vc\to\Ucal$ is locally quasi-finite. Since it is locally of finite type, we have to prove that it has discrete fibers. First note that $\Vc\times_\Xf\Vc$ is a superscheme and that the two projections $p_i\colon\Vc\times_\Xf\Vc\to \Vc$ are \'etale; by (1) they are locally quasi-finite (Corollary \ref{cor:unramprop}). Since $f$ is injective, its diagonal is an isomorphism (Proposition \ref{prop:isodiag}), and then $\Vc\times_\Xf\Vc \simeq \Vc\times_\Ucal\Vc$. Then, we have a cartesian diagram of superscheme morphisms
$$
\xymatrix{
\Vc\times_\Ucal\Vc \ar[r]^{p_2}\ar[d]^{p_1} & \Vc\ar[d]^{f\circ p} \\
 \Vc\ar[r]^{f\circ p} & \Ucal
}
$$
As $p_1$ (or $p_2$) has discrete fibers, so has $f\circ p$, thus finishing the proof.
\end{proof}

\subsection{Separated and proper morphisms of algebraic superstacks}
\label{subsec:separprop}

We start by defining properness and separatedness for morphisms of algebraic superspaces and use them to extend those notions to morphisms of algebraic superstacks. This is the same procedure one follows in the case of algebraic stack morphisms. Hence, as
in \cite[Definition 3.8.1]{Alp24}, one  defines by iteration the notions of separatedness and properness for  algebraic superstack morphisms. Analogously  to Definition \ref{def:qfsstacks}, we proceed in four steps.

\begin{defin}\label{def:separprop}  \
\begin{enumerate}
\item 
A representable morphism $f\colon \Xf \to \Yf$ of algebraic superstacks is separated if the diagonal morphism $\Delta_f\colon \Xf \to \Xf\times_\Yf \Xf$, which is superschematic by Corollary  \ref{cor:represdiag}, is proper according to Definition  \ref{def:morphismsprops}.
\item A representable morphism $f\colon \Xf \to \Yf$ of algebraic superstacks is proper if it is universally closed (Definition \ref{def:univclosed}), separated, and of finite type (Definition \ref{def:finitetype}).
\item A morphism $f\colon \Xf \to \Yf$ of algebraic superstacks is separated if the representable morphism $\Delta_f\colon \Xf \to \Xf\times_\Yf\Xf$ (Proposition \ref{prop:represdiag}) is proper.
\item A morphism $f\colon \Xf \to \Yf$ of algebraic superstacks is proper if it is universally closed, separated, and of finite type.
\end{enumerate}
\end{defin}

 We can also define:
\begin{defin}\label{def:quasisep} A morphism $f\colon \Xf \to \Yf$ of algebraic superstacks is quasi-separated if the diagonal $\Delta_f\colon \Xf\to\Xf\times_\Yf\Xf$ and the second diagonal $\Xf \to \Xf\times_{\Delta_f,\Xf\times_\Yf\Xf,\Delta_f}\Xf$ are quasi-compact.
\end{defin}

\begin{prop}\label{prop:separpropbos} A morphism $f\colon \Xf \to \Yf$ of algebraic superstacks is separated or proper if and only if the bosonic reduction $f_{bos}\colon \Xf_{bos}\to\Yf_{bos}$ has the same property.
\end{prop}
\begin{proof}
The proof that $f_{bos}\colon \Xf_{bos}\to\Yf_{bos}$ is separated or proper if $f\colon \Xf \to \Yf$ is so  is straightforward. For the converse we follow the same steps as in Definition \ref{def:separprop} and \cite[Sect.\! 3.8.1]{Alp24}, relying on the fact   that the statement is true for  superscheme morphisms.
\begin{enumerate}
\item Assume that $f$ is representable, so that $\Delta_f$ is superschematic (Corollary  \ref{cor:represdiag}). Moreover, $f_{bos}$ is also representable (by algebraic spaces) and then $\Delta_{f_{bos}}$ is schematic. If $f_{bos}$ is separated, $\Delta_{f_{bos}}$ is proper by  \cite[Sect.\! 3.8.1]{Alp24}. Since $\Delta_{f_{bos}}=(\Delta_f)_{bos}$, we have that $\Delta_f$ is proper and then $f$ is separated by (1) of Definition \ref{def:separprop}.
\item Assume again that $f$ is representable. By \cite[Definition 3.8.1]{Alp24}, if $f_{bos}$ is proper, then it is universally closed, separated and of finite type. By Proposition \ref{prop:univclosed} and (1), $f$ is also universally closed, separated and of finite type, so that it is proper by (2) of Definition \ref{def:separprop}.
\item In the general case, $\Delta_f$ is representable. If $f_{bos}$ is separated, $\Delta_{f_{bos}}$  is proper by \cite[Definition 3.8.1]{Alp24} and since  $(\Delta_f)_{bos}=\Delta_{f_{bos}}$, $\Delta_f$ is proper by (2). 
\item Finally, if $f_{bos}$ is proper, it is universally closed, separated and of finite type, so that $f$ is also proper by (3) and Proposition \ref{prop:univclosed}.
\end{enumerate}
\end{proof}

\begin{prop}\label{prop:injsep1} Every injective  representable morphism $g\colon \Xf\to\Yf$ of algebraic superstacks, locally of finite type, is superschematic.
\end{prop}  
\begin{proof} We can assume that $\Yf\simeq \Ycal$ is a superscheme and that $g\colon \Xf \to \Ycal$ is an injective morphism of algebraic spaces,  locally of finite type. Then, $g$ is locally quasi-finite by  Proposition \ref{prop:locfinite}. Moreover, due to  Proposition \ref{prop:recollement}, to prove that $\Xf$ is a superscheme we can assume that $\Ycal$ is affine.

We have to prove that for every morphism $\Tc\to \Ycal$, where $\Tc$ is a superscheme, the fiber product $\Tc\times_\Ycal\Xf$ is a superscheme. If $\Xf$ has an open cover $\{\Xf_i\}$ such that each fiber product $\Tc\times_\Ycal\Xf_i$ is a superscheme,  the claim will be proved  invoking  Proposition \ref{prop:recollement}. We may assume that each $\Xf_i$ has an affine presentation: to obtain this,
start from an \'etale presentation $p\colon\Wc\to \Xf$ and let $\{\Wc_i\}$ be an affine cover of $\Wc$, and take $\Xf_i = p(\Wc_i)$. 
So we can assume that $\Xf$ has an affine presentation $p\colon\Wc\to\Xf$.

At this stage we have   a locally quasi-finite morphism of affine superschemes $q=g\circ p\colon \Wc \to \Ycal$.
The algebraic superspace $\Xf$ is the quotient of the \'etale equivalence relation of superschemes $\Rcal:=\Wc\times_\Xf \Wc \rightrightarrows \Wc$ and the two projections $(s,t)\colon\Rcal\rightrightarrows \Wc$ are \'etale and surjective. 
Moreover, since $g$ is a monomorphism, one has $\Rcal=\Wc\times_\Xf \Wc\simeq \Wc\times_\Ycal\Wc$, so that we have a cartesian diagram
$$
\xymatrix{ \Rcal \ar[r]^(.4){(s,t)} \ar[d] & \Wc\times\Wc \ar[d]^{q\times q}
\\
\Ycal \ar@{^(->}[r]^{\Delta_\Ycal}& \Ycal\times\Ycal
}
$$
This proves that $\Rcal$ is an affine superscheme,  $s$ and $t$ are quasi-finite and affine, and $(s,t)\colon\Rcal \to \Wc\times\Wc$ is a closed immersion.

We prove that $\Xf$ is a superscheme  using Proposition \ref{prop:affinequot};  we have only to prove that $s$ and $t$ are finite. Since the  finiteness of an affine superscheme morphism only depends  on the bosonic reduction (Proposition \ref{prop:finitebos}), we need to prove that the bosonic reductions  $(s_{bos},t_{bos})\colon R \rightrightarrows W$ are finite. Actually they are quasi-finite and \'etale  scheme morphisms, in particular they are  flat and of finite presentation, so that, by the  Zariski Main Theorem \cite[Th\'eor\`eme 8.12.6]{EGAIV-IV}, they factor as compositions
$$
\xymatrix{
R \ar@{^(->}[r]^i &R_s  \ar[r]^{\bar s} &W
}
\qquad
\xymatrix{
R \ar@{^(->}[r]^j &R_t \ar[r]^{\bar t} &W
}
$$
where $i$ and $j$ are open immersions and $\bar s$, $\bar t$ are finite. Now we have a diagram
$$
\xymatrix{
R \ar@{^(->}[rr]^{(s,t)_{bos}} \ar@{^(->}[rd]^{(i,j)} && W\times W \\
& R_s \times R_t \ar[ur]^{\bar s\times \bar t}& 
}
$$
Since $(s,t)_{bos}$ is a closed immersion and $(\bar s, \bar t)$ is finite, $(i,j)$ is proper, so that it is an isomorphism with a closed and open subscheme of $R_s\times R_t$. Thus, there exist open and closed subschemes $R'_s\hookrightarrow R_s$, $R'_t \hookrightarrow R_t$ and factorizations
$$
\xymatrix{
R \ar@{^(->}[rr]^{s_{bos}} \ar@{^(->}[rd]^{\simeq} && W \\
& R'_s \ar[ur]_{\bar s_{| R'_s}}& 
}
\quad 
\xymatrix{
R \ar@{^(->}[rr]^{t_{bos}} \ar@{^(->}[rd]^{\simeq} && W \\
& R'_t \ar[ur]_{\bar t_{| R'_t}}& 
}
$$
It follows that $s_{bos}$ and $t_{bos}$ are finite, as it was to be proved.
\end{proof}

Building up on previous definitions we can   give the notion of {\em noetherian superstack.}

\begin{defin}\label{def:noetherian}
We say that an algebraic superstack is noetherian if it is  locally noetherian
(Definition \ref{def:locnoeth}), quasi-separated (Definition \ref{def:quasisep}), and quasi-compact (Definition \ref{def:quasi-compact}).
\end{defin}

\section{Elements of the theory of quasi-coherent sheaves}
We develop a basic theory of quasi-coherent sheaves on superstacks. We   consider first  the case of Deligne--Mumford superstacks, and then  algebraic superstacks;
we shall basically extend the treatment of the ordinary case in \cite{Alp24}.

\subsection{Quasi-coherent sheaves on Deligne--Mumford  superstacks}\label{subsec:quasi-coh}
In this subsection $\Xf$ will always denote a Deligne--Mumford  superstack. We shall denote by $\Xf_{et}$ the small \'etale site of $\Xf$, that is, the category of superschemes that are \'etale over $\Xf$. A covering of an $\Xf$-superscheme $\Ucal$ is a collection of \'etale morphisms 
$\{\Ucal_i\to\Ucal\}$ over $\Xf$ such that $\coprod_i\Ucal_i\to\Ucal$ is surjective. 
 We can then consider the category $\operatorname{Ab}(\Xf)$ of sheaves of $\Z_2$-graded abelian groups over the \'etale site $\Xf_{et}$, with homogeneous (degree 0) morphisms.

We start by defining the structure sheaf of a Deligne--Mumford superstack,   the sheaf of its relative differentials, and the notion of module over the structure sheaf.
\begin{defin}\label{def:sheafexamples}\ 
\begin{enumerate}
\item
The structure sheaf $\Oc_\Xf$ of $\Xf$ is  the sheaf of $\Z_2$-graded abelian groups on $\Xf_{et}$ defined by
$$ \Oc_\Xf(\Ucal \to \Xf) = \Gamma(\Ucal,\Oc_\Ucal)$$
if $\Ucal\to\Xf$ is an \'etale superscheme over $\Xf$.  It is a sheaf of superrings.
\item
A $\mathbb Z_2$-graded    $\Oc_\Xf$-module 
 is a  $\mathbb Z_2$-graded module object for $\Oc_\Xf$ in the category of sheaves on $\Xf_{et}$.
\item
 If $\Xf\to\Sc$  is a Deligne--Mumford  superstack over a superscheme $\Sc$, its sheaf of relative differentials is defined as the  $\mathbb Z_2$-graded    $\Oc_\Xf$-module
$$ \Omega_{\Xf/\Sc}(\Ucal\to\Xf) = \Gamma(\Ucal,\Omega_{\Ucal/\Sc}).$$ 
\end{enumerate}
\end{defin}

If  $\operatorname{Mod}(\Oc_\Xf)$ is the category
of $ \mathbb Z_2$-graded    $\Oc_\Xf$-modules, given a morphism of  Deligne--Mumford  superstacks
$f\colon\Xf\to\Yf$, there are push-forward, inverse image and pullback functors
$$ f_\ast \colon \operatorname{Ab}(\Xf) \to \operatorname{Ab}(\Yf), \qquad f^{-1} \colon \operatorname{Ab}(\Yf) \to \operatorname{Ab}(\Xf)$$
and  
$$f_\ast \colon \operatorname{Mod}(\Oc_\Xf) \to \operatorname{Mod}(\Oc_\Yf), \qquad
f^\ast \colon \operatorname{Mod}(\Oc_\Yf) \to \operatorname{Mod}(\Oc_\Xf).$$

As usual, the two functors $f_\ast$ coincide in the appropriate sense, $f^\ast \Gc = f^{-1} \Gc \otimes_{f^{-1}\Oc_\Yf} \Oc_{\Xf}$
for any $\Oc_{\Yf}$-module $\Gc$, 
and  $(f_\ast,f^{-1})$ and $(f_\ast,f^\ast)$ are pairs of adjoint functors.
Note that if
$f$ is \'etale then $f^\ast\Gc\simeq f^{-1}\Gc$.

\subsubsection*{Quasi-coherent sheaves} If $\Ucal\to\Xf$ is an \'etale superscheme over $\Xf$, and $\Fc$ is a $\mathbb Z_2$-graded module on $\Xf$, 
we shall denote by $\Fc_{\vert \Ucal,et}$ the restriction of $\Fc$ to the small \'etale site of $\Ucal$, and by $\Fc_{\vert\Ucal, Zar}$ the restriction to $\Ucal$ in the Zariski topology.

\begin{defin} \label{def:qcDM} \ 
\begin{enumerate}
\item
An $\Oc_\Xf$-module $\Fc$ is \emph{quasi-coherent} if
\begin{enumerate}\item[a)] $\Fc_{\vert\Ucal, Zar}$ is quasi-coherent sheaf for every \'etale superscheme $\Ucal\to\Xf$ over $\Xf$;
\item[b)] for every \'etale morphism $f\colon\Ucal\to\mathcal V$ of \'etale $\Xf$-superschemes, $f^\ast(\Fc_{\vert\mathcal V, Zar}) \to
\Fc_{\vert\Ucal, Zar}$ is an isomorphism.
\end{enumerate}
\item A quasi-coherent sheaf $\Fc$ on $\Xf$ is \emph{locally free of rank $m|n$} (i.e., it is a \emph{supervector bundle}) if
 $\Fc_{\Ucal, Zar}$ is so for every \'etale superscheme $\Ucal$ over $\Xf$. 
\item
If $\Xf$ is locally noetherian, $\Fc$ is \emph{coherent} if $\Fc_{\vert\Ucal, Zar}$ is coherent for every \'etale superscheme $\Ucal$ over $\Xf$.
\end{enumerate} 
\end{defin}

\begin{remark} The structure sheaf $\Oc_\Xf$ is an even line bundle on $\Xf$ (i.e., a locally free sheaf of rank $1|0$), and if $\Xf$  is locally noetherian, it is coherent. If $\Xf\to\Sc$  is a Deligne--Mumford  superstack over a superscheme $\Sc$, the sheaf of relative differentials  $ \Omega_{\Xf/\Sc}$ is quasi-coherent. By Proposition  \ref{prop:smoothsplit3}, it is locally free of rank $m|n$ if $\Xf\to\Sc$ is smooth of relative dimension $m|n$.
\end{remark}

 As in the case of superschemes, one easily sees that if $f\colon\Xf\to\Yf$ is a morphism of Deligne--Mumford superstacks and $\Fc$ is a quasi-coherent sheaf on $\Xf$, the sheaf $f_\ast\Fc$ on $\Yf$ is quasi-coherent.

\subsubsection*{Descent for quasi-coherent sheaves on algebraic superspaces}
  Since algebraic superspaces are Deligne--Mumford superstacks, at this point we know what quasi-coherent sheaves on algebraic superspaces are. 
Next we provide  without proof a descent theorem for
quasi-coherent sheaves on algebraic superspaces. This is a straitghforward  
extension to algebraic superspaces of Proposition \ref{prop:fpqcdescent};
the version for ordinary stacks is   
in \cite[Tag 04W7]{Stacks}. 

\begin{prop}[Descent for quasi-coherent sheaves on algebraic superspaces]
\label{prop:descentqcalgsup}
Let $\{\Xf_i\to\Xf\}$ be an fpqc (Definition \ref{def:finitetype}) covering of algebraic superspaces.
Then any descent datum of quasi-coherent sheaves for $\{\Xf_i\to\Xf\}$  is effective.
\qed
\end{prop}

\subsubsection*{Relative differentials} Generalizing the treatment in
\cite{Ol16}, we give a definition of the sheaf of relative differentials for a morphism of Deligne--Mumford superstacks.  Let $\mathbf T_\Xf$ be the topos associated with the small \'etale site of a Deligne--Mumford superstack,  that is, the category $\operatorname{Ab}(\Xf)$ of $\Z_2$-graded abelian sheaves on $\Xf_{et}$. Let $ \As\to\Bs$ be a superring morphism   in  $\mathbf T_\Xf$, and let $\mathbb M$ be a $\Z_2$-graded $\Bs$-module. There is a standard construction in superalgebra  which defines the  $\Z_2$-graded $\Bs$-module $\operatorname{Der}_\As(\Bs,\mathbb M)$ of
$\As$-linear $\Z_2$-graded derivations of $\Bs$ with values in $\mathbb M$.  
Moreover, there exists a unique
$\Z_2$-graded $\Bs$-module $\Omega_{\Bs/\As}$, equipped with an $\As$-linear $\Z_2$-graded derivation
$d\colon\Bs\to \Omega_{\Bs/\As}$, such that for every $\Z_2$-graded $\Bs$-module $\mathbb M$
the natural morphism
\begin{equation}\label{eq:univdiff}
\Hom_\Bs( \Omega_{\Bs/\As},\mathbb M) \to \operatorname{Der}_\As(\Bs,\mathbb M)
\end{equation}
is an isomorphism. 
The construction is as follows. Let $\tau\colon \Bs\to\Bs$ be   the identity on the even sections and the multiplication by $-1$ on the odd ones. For any morphism 
$\phi\colon \mathbb{N}\to\mathbb{M}$ of $\Z_2$-graded $\As$-modules with homogeneous decomposition $\phi=\phi_++\phi_-$, write $\phi^\tau =\phi_+-\phi_-$.   
A  morphism $D\colon \Bs \to \mathbb M$  of $\Z_2$-graded $\As$-modules is a \emph{$\Z_2$-graded derivation or superderivation over $\As$}  if the diagram
$$
\xymatrix{ \Bs\otimes_\As\Bs \ar[r]^\cdot \ar[d]_{D\otimes 1+1\otimes D^\tau+\tau\otimes D}& \Bs \ar[d]^D\\
\mathbb M\otimes_\As\Bs\oplus \Bs\otimes_\As\mathbb M \ar[r]^(.7)+& \mathbb M
}
$$
is commutative. If  $\Delta$ is the diagonal ideal of $\As\to\Bs$, that is, the kernel of the product morphism $\Bs\otimes_\As\Bs \to \Bs$, and $\Omega_{\Bs/\As}:=\Delta/\Delta^2$, there is a morphism
$$
d\colon \Bs \to \Omega_{\Bs/\As}
$$
given by the composition of the morphism $\Id_\Bs\otimes 1-1\otimes\Id_\Bs\colon \Bs\to \Delta$ with the quotient morphism $\Delta \to \Omega_{\Bs/\As}$. One checks that $d$ is actually a superderivation over $\As$, homogeneous of degree zero, and that equation \eqref{eq:univdiff} holds, so that $\Omega_{\Bs/\As}$ is uniquely determined.

If $f\colon\Xf\to\Yf$ is a morphism of Deligne--Mumford superstacks, by letting $\As=f^{-1}\Oc_\Yf$ and $\Bs=\Oc_\Xf$, one obtains a sheaf on $\Xf$, that we denote $\Omega_{\Xf/\Yf}$ or $\Omega_f$ --- the sheaf of relative differentials of the morphism $f$.  If $\Yf$ is a superscheme, we get the sheaf of relative differentials given in Definition \ref{def:sheafexamples}. Moreover, if both $\Xf$ and $\Yf$ are superschemes,  one gets the sheaf of relative differentials as  defined in \cite{BrHR21}, Definition 2.11. 

We give some basic properties of the sheaf $\Omega_{\Xf/\Yf}$  for morphisms of algebraic superspaces.
First, note that if $\Sc$ is a superscheme and $\F$ is a quasi-coherent sheaf of the big \'etale site of  $\Sc$-superschemes, we can extend $\F$ to define a sheaf of the big \'etale site of algebraic superspaces over $\Sc$, by setting $\F(\Xf)= \ker(F(\Ucal)\rightrightarrows \F(\Rcal))$ for an \'etale presentation $\Ucal \to\Xf$ from a $\Sc$-superscheme $\Ucal$, and with $\Rcal=\Ucal\times_\Xf\Ucal$ (see \cite[II.2.5]{Knut71}). In this way, we can define the structure sheaf $\Oc$ and the differential sheaf $\Omega$ on the  big \'etale site of algebraic superspaces over $\Sc$. If we fix an algebraic superspace $\Xf\to\Sc$, the restriction of these sheaves to the small \'etale site of $\Xf$ recover the structure sheaf $\Oc_\Xf$ and the relative differentials sheaf $\Omega_{\Xf/\Sc}$ of Definition \ref{def:sheafexamples}. From this description one easily sees that the formation of $\Omega_{\Xf/\Sc}$ is compatible with base changes of superschemes.

\begin{lemma}\label{lem:basechomega} Let $\phi\colon \Tc \to \Sc$ a superscheme morphism and $\Xf \to \Sc$ an algebraic superspace morphism. There is an isomorphism $\phi_\Xf^\ast \Omega_{\Xf/\Sc}\simeq \Omega_{\Xf\times_\Sc\Tc/\Tc}$, where $\phi_\Xf\colon \Xf\times_\Sc\Tc\to\Xf$ is the morphism of algebraic superspaces induced by base change.
\qed
\end{lemma}

 Also, for a morphism $f\colon \Xf\to \Yf$ of algebraic superspaces we can give an alternative description of $\Omega_{\Xf/\Yf}$. Let $p\colon\Sc\to\Yf$ be an \'etale presentation of $\Yf$; consider the superscheme $\Tc:=\Sc\times_\Yf\Sc$ and the 2-cartesian diagram
 \begin{equation}\label{eq:diffdef}
\xymatrix{\Xf_\Tc:=\Xf\times_\Sc\Xf \ar@<.5ex>[r]^(.7){\tilde q_1}\ar@<-.5ex>[r]
_(.7){\tilde q_2} \ar[d]^{f_\Tc} & \Xf_\Sc\ar[r]^{p_\Xf} \ar[d]^{f_\Sc}& \Xf \ar[d]^f \\
\Tc:=\Sc\times_\Yf\Sc \ar@<.5ex>[r]^(.7){q_1}\ar@<-.5ex>[r]_(.7){q_2} &\Sc\ar[r]^p & \Yf
}
\end{equation}

 By Lemma \ref{lem:basechomega}, there are isomorphisms of sheaves $\tilde q_1^\ast(\Omega_{\Xf_\Sc/\Sc})\simeq \Omega_{\Xf_\Tc/\Tc}$ and $\tilde q_2^\ast(\Omega_{\Xf_\Sc/\Sc})\simeq \Omega_{\Xf_\Tc/\Tc}$, and then an isomorphism $\tilde q_1^\ast(\Omega_{\Xf_\Sc/\Sc})\simeq \tilde q_2^\ast(\Omega_{\Xf_\Sc/\Sc})$. One checks that  this isomorphism yields a descent datum, so that, by Proposition \ref{prop:descentqcalgsup}, there exists a sheaf $\Gc$ on $\Xf$ such that $p^\ast \Gc\simeq \Omega_{\Xf_\Sc/\Sc}$. Now, $\Gc\simeq \Omega_{\Xf/\Yf}$. In other words:
\begin{lemma} \label{lem:diffdef} The sheaf $\Omega_{\Xf/\Yf}$ of relative differentials is the sheaf defined by descent from $\Omega_{\Xf_\Sc/\Sc}$ via the diagram \eqref{eq:diffdef}. \qed
\end{lemma}

\begin{prop} \label{prop:propreldiff}  \ 
\begin{enumerate} \item
If $f\colon\Xf\to\Yf$ is  a morphism of algebraic superspaces,
the sheaf $\Omega_{\Xf/\Yf}$ of relative differentials is quasi-coherent.
\item \label{item:2}   If
\begin{equation}\label{eq:diagmorphDM}
\xymatrix{
\Xf' \ar[r]^g \ar[d]_{f'} & \Xf \ar[d]^f \\
\Yf' \ar[r]_h & \Yf
}
\end{equation}
is a commutative diagram of algebraic superspaces,
there is a natural morphism
$$ g^\ast\Omega_{\Xf/\Yf} \to \Omega_{\Xf'/\Yf'},$$
which is an isomorphism if the diagram is 2-cartesian.
\item  If $\Xf \xrightarrow{f} \Yf \xrightarrow{g} \Zf $
are morphisms of algebraic superspaces, there is
an exact sequence
\begin{equation} \label{eq:exseqreldiff1}
f^\ast\Omega_{\Yf/\Zf} \to \Omega_{\Xf/\Zf} \to \Omega_{\Xf/\Yf} \to 0.
\end{equation}
In particular by taking $\Zf = \Spec \Z$ one obtains that, given
a morphism of algebraic superspaces $f\colon\Xf\to \Yf$, there is an exact sequence
\begin{equation} \label{eq:exseqreldiff2}
f^\ast\Omega_{\Yf} \to \Omega_{\Xf} \to \Omega_{\Xf/\Yf} \to 0.
\end{equation}
\item  \label{item:reldiff4}  If the diagram \eqref{eq:diagmorphDM} is cartesian, the natural morphism
\begin{equation} \label{eq:addreldiff}
{f'}^\ast\Omega_{\Yf'/\Yf} \oplus {g}^\ast\Omega_{\Xf/\Yf} 
\to \Omega_{\Xf'/\Yf}
\end{equation}
is an isomorphism.
\item  \label{item:reldiff5}   If a morphism of algebraic superspaces $f\colon\Xf\to \Yf$ is  unramified, then 
$\Omega_{\Xf/\Yf}=0$. 
\end{enumerate}
\end{prop}
 \begin{proof} The properties (1) to (4) follows straightforwardly from Lemma \ref{lem:diffdef}. For (5), let $p\colon\Sc\to\Yf$ be an \'etale presentation of $\Yf$. Then, $f_\Sc\colon \Xf_\Sc\to \Sc$ is also unramified (Definition \ref{def:morphismsprops2}) and $p_\Xf^\ast\Omega_{\Xf/\Yf}\simeq \Omega_{\Xf_\Sc/\Sc}$ by (2). Since $p$ is faithfully flat, it is enough to prove that $\Omega_{\Xf_\Sc/\Sc}=0$. Using Definition \ref{def:sheafexamples}, we see that if $\Ucal \to \Xf_\Sc$ is an \'etale morphism from a superscheme, one has $\Omega_{\Xf_\Sc/\Sc}(\Ucal)=\Gamma(\Ucal, \Omega_{\Ucal/\Sc})$;  the latter is zero by Proposition \ref{prop:formunram}, thus finishing the proof.
\end{proof}

\subsubsection*{A useful construction.} In section \ref{subsec:dmdiag} we shall need to consider
the sheaf of relative differentials for a  smooth representable morphism $ f\colon \Xcal \to \Xf$, where $\Xcal$ is a superscheme, and $\Xf$ is an algebraic superstack; the construction we have just developed does not apply to this case.  Again, the ad hoc construction made in \cite[pp.~180--181]{Ol16} for the bosonic case directly generalizes to the super setting, as we are going to sketch. The fiber product $\Zf = \Xcal\times_\Xf\Xcal$
is an algebraic superspace. We can construct the commutative diagram
\def\quno{\raise4pt\hbox{\tiny $q_1$}}\def\qdue{\raise-5pt\hbox{\tiny $q_2$}}
 \begin{equation} \label{eq:monster} \xymatrix{
\Zf\times_{\operatorname{pr_1},\Xcal,  \operatorname{pr_1}}\Zf \ar[d]^{q:=(\operatorname{pr_2},\operatorname{pr_2})}
\ar@<-.4ex>[r]^(0.7){\quno} \ar@<.4ex>[r]_(0.7){\qdue}& \Zf \ar[d]^{\operatorname{pr_2}}  \ar[r]^{\operatorname{pr_1}} & \Xcal \ar[d]^f \\
 \Xcal\times_\Xf\Xcal \ar@<-.4ex>[r] \ar@<.4ex>[r] & \Xcal \ar[r]^f & \Xf 
}
\end{equation}
 
Since we are able to define relative differentials for morphisms of algebraic superspaces, in particular we can consider
the relative differentials 
$ \Omega_q$
and $\Omega_{\operatorname{pr_2}}$. One has an isomorphism
$$ \epsilon \colon q_1^\ast \Omega_{\operatorname{pr_2}} \to  q_2^\ast \Omega_{\operatorname{pr_2}} $$
obtained by combining the  base-change isomorphisms   (Proposition \ref{prop:propreldiff} \eqref{item:2})
$$ q_1^\ast \Omega_{\operatorname{pr_2}} \simeq \Omega_q
\qquad\text{and}\qquad q_2^\ast \Omega_{\operatorname{pr_2}} \simeq \Omega_q
$$
coming from diagram \eqref{eq:monster}. The isomorphism $\epsilon$ satisfies the cocycle condition for the 
triple fiber product 
$$ \Zf\times_{\operatorname{pr_1},\Xcal,  \operatorname{pr_1}}\Zf \times_{\operatorname{pr_1},\Xcal,  \operatorname{pr_1}}\Zf$$
 and therefore
by  descent for quasi-coherent sheaves on algebraic superspaces (Proposition \ref{prop:descentqcalgsup}) there is a quasi-coherent sheaf 
$\Omega_{\Xcal/\Xf}$ 
on $\Xcal$ 
 such that
\begin{equation}\label{eq:descentomega}
\Omega_{\operatorname{pr_2}}\simeq \operatorname{pr_1}^\ast \Omega_{\Xcal/\Xf}\,.
\end{equation}
 
The arguments in \cite{Ol16} generalize straightforwardly to yield the following results. We start by noting that  there is  a morphism
\begin{equation} \label{eq:surjnat}
 \Omega_\Xcal \to 
 \Omega_{\Xcal/\Xf}
 \end{equation} defined as follows.  
 Starting from the composition
 $$ \operatorname{pr_1}^\ast\Omega_\Xcal \to \Omega_\Zf \to \Omega_{\operatorname{pr_2}}$$
 one forms the   commutative diagram
\begin{equation}\label{eq:speremmuben}
\xymatrix{q_1^\ast\circ \operatorname{pr_1}^\ast  \Omega_\Xcal \ar[d]_\eta \ar[r] & q_1^\ast \Omega_{\operatorname{pr_2}} \ar[d]^\epsilon  \\
 q_2 ^\ast\circ \operatorname{pr_1}^\ast   \Omega_\Xcal \ar[r]  & q_2^\ast \Omega_{\operatorname{pr_2}}
 }
  \end{equation}

 The morphism  
  $\eta$  trivially satisfies the cocycle condition giving rise by descent to the sheaf $\Omega_\Xcal$. The horizontal arrows in
  diagram \eqref{eq:speremmuben}  define morphisms between descent data,  yielding the morphism \eqref{eq:surjnat}.
Let us consider the diagram
\begin{equation} \label{eq:olssondiag}
\xymatrix{
\operatorname{pr_2}^\ast \Omega_\Xcal \ar[r] & \Omega_{\Zf} \ar[r] & \Omega_{\operatorname{pr_2}} \ar[r] & 0 \\
\operatorname{pr_1}^\ast \Omega_\Xcal  \ar@{.>}[urr]  \ar@{_{(}->}[r]    & 
\operatorname{pr_1}^\ast \Omega_\Xcal \oplus \operatorname{pr_2}^\ast \Omega_\Xcal \ar[u]_(.3)u   &
}
\end{equation}
where the first row is exact   (see equation \eqref{eq:exseqreldiff2}).
 Since $\Omega_{\operatorname{pr_2}}\simeq 
\operatorname{pr_1}^\ast\Omega_{\Xcal/\Xf}$ (equation \eqref{eq:descentomega}) the dotted arrow yields  a   morphism
$\operatorname{pr_1}^\ast \Omega_\Xcal \to 
\operatorname{pr_1}^\ast\Omega_{\Xcal/\Xf}$ which is the pullback of the morphism \eqref{eq:surjnat}.

\begin{lemma}\label{lem:propadhoc} \needspace{4\baselineskip}
\ 
\begin{enumerate}
\item The sheaf $\Omega_{\Xcal/\Xf}$ is locally free of finite rank.
\item If the diagonal of $\Xf$ is unramified, the  morphism
\eqref{eq:surjnat} is surjective.
\end{enumerate}
\end{lemma}
\begin{proof} (1) This follows from the fact that the quasi-coherent sheaf $\Omega_{\operatorname{pr_2}}$ is itself
locally free (as $f$ is a smooth morphism, and hence $\operatorname{pr_2}$ is smooth as well, see Definition \ref{def:smooth}) of finite rank.

(2)  If the     morphism  $u$  in diagram
\eqref{eq:olssondiag}
 is surjective,  then the  dotted arrow,
i.e., the morphism $\operatorname{pr_1}^\ast \Omega_\Xcal \to 
\operatorname{pr_1}^\ast\Omega_{\Xcal/\Xf}$,
 is surjective as well. Since $\operatorname{pr_1}$ is faithfully flat, the morphism \eqref{eq:surjnat} is surjective. So we need to show that $u$ is surjective.
 
 Applying Proposition \ref{prop:propreldiff} \eqref{item:reldiff4}  to the diagram defining the  product  
 $\Xcal\times\Xcal$,  one obtains 
 $$\Omega_{\Xcal\times\Xcal}
 \simeq  p_1^\ast \Omega_\Xcal \oplus p_2^\ast \Omega_\Xcal  
 $$
 where $p_1$, $p_2$ are the projections onto the factors of the product. 
 Taking this into account, by applying equation  \eqref {eq:exseqreldiff2} to the composition of morphisms
 $$
  \Zf \to \Xcal\times\Xcal \to \Spec \Z
 $$ 
 with obtain the exact sequence 
\begin{equation} \label{eq:exseqZ}  \operatorname{pr_1}^\ast \Omega_\Xcal \oplus \operatorname{pr_2}^\ast \Omega_\Xcal
 \xrightarrow{\ \ u\ \ } \Omega_\Zf \to \Omega_{\Zf/\Xcal\times\Xcal}\to 0.
 \end{equation} 
We observe  that the diagram
 $$\xymatrix{ \Zf \ar[d] \ar[r] & \Xf \ar[d]^\Delta \\
\Xcal\times\Xcal \ar[r] &  \Xf\times\Xf }$$
 is 2-cartesian, so that  if the diagonal of $\Xf$ is unramified,    the morphism  $\Zf \to \Xcal\times\Xcal$ is unramified as well,
and  applying Proposition \ref{prop:propreldiff} \eqref{item:reldiff5} we have $\Omega_{\Zf/\Xcal\times\Xcal}=0$.  Thus equation \eqref{eq:exseqZ} implies that the morphism $u$ is surjective. 
\end{proof}

\subsection{Quasi-coherent sheaves on  algebraic superstacks}\label{subsec:quasi-coh2}
 We generalize  to algebraic superstacks what we did   about  Deligne--Mumford superstacks   in Subsection \ref{subsec:quasi-coh}. Again, we  basically extend the treatment of the ordinary case, following \cite{LauMor-Bai00,Ol16,Alp24}.  The main point is that the \'etale site of $\Xf$ must be extended to the 
 {\em lisse-\'etale site}.  Let $\Xf$ be an algebraic superstack over a superscheme $\Sc$.  
 
As for stacks, an $\Xf$-superspace is a pair $(\Tf,t)$, where $\Tf$ is an algebraic superspace and $t\colon \Tf\to\Xf$ is a morphism over $\Sc$. A morphism of $\Xf$-superspaces $ (\Tf',t') \to (\Tf,t)$  is a pair
$(f,f^\sharp)$
where $f\colon \Tf'\to\Tf$ is a morphism of algebraic superspaces over $\Sc$ and $f^\sharp \colon t' \to t\circ f$ is an isomorphism of functors from $\Tf'$ to $\Xf$. Morphisms of $\Xf$-superspaces are composed in the obvious way: given two such morphisms 
$$ (\Tf'',t'') \xrightarrow{(g,g^\sharp)}  (\Tf',t') \xrightarrow{(f,f^\sharp)} (\Tf,t)$$ 
the composed morphism $  (\Tf'',t'') \to (\Tf,t)$  is given by the composition $f\circ g$ together with the isomorphism of functors
$t'' \to t \circ (f\circ g)$.  In this way one defines the category  of $\Xf$-superspaces; this has a natural full subcategory  $\Sf/\Xf$ where algebraic superspaces are replaced by superschemes. 

\begin{defin} 1. The {\em lisse-\'etale site} $\Xf_{lis-et}$ of $\Xf$ is the full subcategory of $\Sf/\Xf$ whose objects are pairs
$(\Tc,t)$ where $\Tc$ is a superscheme over $\Sc$,  and $t\colon\Tc\to\Xf$ is a smooth morphism over $\Sc$.

2. A {\em covering} of an object $(\Tc,t)$ of $\Xf_{lis-et}$ is a collection  of maps $(f_i,f_i^\sharp)\colon (\Tc_i,t_i)\to(\Tc,t)$
such that the collection $\{f_i\colon\Tc_i\to\Tc\}$ is an \'etale covering.

3. The {\em structure sheaf} $\Oc_\Xf$ of $\Xf$ is the sheaf on $\Xf_{lis-et}$ defined by
$$\Oc_\Xf(\Tc,t) = \Gamma(\Tc,\Oc_\Tc).$$
\end{defin}
We denote by $\operatorname{Mod}(\Oc_\Xf)$
the naturally defined category of  $\Z_2$-graded commutative $\Oc_\Xf$-modules.

If $(\Tc,t)$ is an object in  $\Xf_{lis-et}$, there is a natural inclusion $\Tc_{et} \hookrightarrow \Xf_{lis-et}$,
where $\Tc_{et}$ is the \'etale site of $\Tc$, given by $(h,\Tc'\to\Tc) \mapsto (\Tc', t\circ h)$.  If $\Fc$ is a sheaf on
  $\Xf_{lis-et}$ we shall denote by $\Fc_{(\Tc,t)}$ the restriction of $\Fc$ to  $\Tc_{et}$ whenever
  $(\Tc,t)$ is an object in $\Xf_{lis-et}$.
  
  \begin{remark}  \label{rmk:defsec}  If $\Fc$ is a sheaf on
  $\Xf_{lis-et}$, one can make sense of its sections on smooth morphisms $\Yf\to \Xf$ of algebraic superstacks by defining
  $\Fc(\Yf\to \Xf)$ as the equalizer
  $$ \operatorname{Eq}\big( \Fc(\Ucal \to \Xf) \rightrightarrows (\Fc(\Vc \to \Xf)
  \bigr) $$ 
  where $\Ucal \to \Yf$ is a smooth presentation from a superscheme and $\Vc \to \Ucal\times_\Yf\Ucal$ is an \'etale presentation from a superscheme. This definition turns out to be independent of the choices of the presentations. 
  \end{remark}
  
  \begin{defin} 1. If $\Ac$ is a sheaf of  superrings on $\Xf_{lis-et}$,
  a sheaf $\Fc$ of  $\Z_2$-graded $\Ac$-modules is \emph{cartesian} if for every morphism
  $(f,f^\sharp) \colon (\Tc',t')\to (\Tc,t)$ the morphism of $\Z_2$-graded $\Rcal_{(\Tc',t')}$-modules  
  $f^\ast \Fc_{(\Tc,t)} \to \Fc_{(\Tc',t')}$ is an isomorphism.
  
  2. A sheaf $\Fc$ of $\Z_2$-graded $\Oc_\Xf$-modules is {\em quasi-coherent} if it is cartesian, and
  for every object $(\Tc,t)$ in $\Xf_{lis-et}$, the sheaf  $\Fc_{(\Tc,t)}$ is a quasi-coherent sheaf on the superscheme $\Tc$.
  
  3. If $\Xf$ is locally noetherian (Definition \ref{def:locnoeth}), a quasi-coherent sheaf $\Fc$ on $\Xf$ is \emph{coherent} if each $\Fc_{(\Tc,t)}$ is a coherent sheaf on the superscheme $\Tc$.
  \end{defin}
  
  \begin{remark} The structure sheaf $\Oc_\Xf$ is quasi-coherent, and it is coherent if $\Xf$ is locally noetherian.
  \end{remark}
  \begin{defin} (Push-forwards) Let $\colon\Xf\to\Yf$ be a morphism of algebraic superstacks over a superscheme $\Sc$.
  The functor $f_\ast \colon \operatorname{Mod}(\Oc_\Xf) \to \operatorname{Mod}(\Oc_\Yf)$ is defined by letting,
  for every object $(\Ucal\to\Yf)$ in $\Yf_{lis-et}$,
  $$ (f_\ast \Fc) (\Ucal\to\Yf) = \Fc ( \Xf_\Ucal \to \Xf) $$
  where 
  $\Xf_\Ucal$ is the fiber product $\Xf\times_\Yf\Ucal$, which  is a smooth $\Xf$-superstack (the right-hand side in this definition makes sense due to Remark \ref{rmk:defsec}).
   \end{defin}
  
 \subsubsection*{Relation between quasi-coherent sheaves on algebraic and Deligne--Mumford superstacks}
 In this and in the previous subsections we   defined the notions of quasi-coherent sheaves in the cases of Deligne--Mumford  and algebraic superstacks; in other terms, we defined the category $\operatorname{Qcoh}(\Xf_{et})$ when
 $\Xf$ is a Deligne--Mumford superstack, and the category  $\operatorname{Qcoh}(\Xf_{lis-et})$ when $\Xf$ is an algebraic superstack.
 If $\Xf$ is Deligne--Mumford, the two categories turn out to be equivalent, as in the case of stacks (see \cite[9.1.16--9.1.18]{Ol16}). The reason is that when $\Xf$ is a Deligne--Mumford superstack, the small \'etale site $\Xf_{et}$
 is a full subcategory of the lisse-\'etale site $\Xf_{lis-et}$; this induces a morphism between the corresponding superringed topoi, which at the level of the categories of quasi-coherent sheaves turns out to be an equivalence of categories. 

\section{Characterizing algebraic superstacks by properties of the diagonal}
\label{sec:charstack}
Deligne--Mumford superstack are in many respect nicer and suppler objects than algebraic superstacks, so it is important to have criteria to ascertain if a given algebraic superstack is Deligne--Mumford. An important criterion may given in terms of the unramifiedness  of the diagonal.

\subsection{Characterization of  Deligne--Mumford superstacks}
\label{subsec:dmdiag}
We start with a Lemma that will be of fundamental importance in the sequel, and   is somehow unexpected in a way.
\begin{lemma}\label{lem:unram} Every unramified algebraic superspace (Definition \ref{def:morphismsprops2})  is bosonic, that is, it is an ordinary unramified algebraic space.
\end{lemma}
\begin{proof} Let $\Xf$ be an unramified algebraic superspace and $p\colon \Ucal \to \Xf$ an \'etale presentation. By Proposition \ref{prop:quot}, $\Xf$ is the quotient sheaf of the \'etale equivalence relation of superschemes $\Rcal \rightrightarrows \Ucal$. Moreover, $\Ucal$ and $\Rcal$ are  unramified since $\Xf$ is. By Proposition \ref{prop:unram}, both are bosonic, i.e. ordinary schemes, so that $\Rcal \rightrightarrows \Ucal$ is an \'etale equivalence relation of schemes. The quotient sheaf for the \'etale topology of schemes is an algebraic space $\mathcal X$, and one easily sees that it is also the quotient for the \'etale topology of superschemes, due to the fact that every morphism from a scheme to a superscheme factors through the bosonic reduction of the target. Thus $\mathcal X=\Xf$.
\end{proof}

\begin{lemma}\label{lem:ftpointsDM} Let $\Xf$ be an algebraic superstack. If for every finite type point $x$ of $\Xf$ there is a representable \'etale morphism $\Ucal_x \to \Xf$ from a superscheme such that $x$ is the image a closed point $u$ of $\Ucal_x$, then $\Xf$ is Deligne--Mumford.
\end{lemma}
\begin{proof} Take $\Ucal$ as the disjoint union of the various $\Ucal_x$ and define $p\colon \Ucal \to \Xf$ as the disjoint union of the morphisms $\Ucal_x \to \Xf$. Then $p$ is \'etale and representable. Moreover, it is open by Proposition \ref{prop:open} and it is then surjective by Proposition \ref{prop:ftdense}. Thus $\Xf$ is Deligne--Mumford.
\end{proof}

\begin{prop}\label{prop:unramdiag1} Let $\Xf$ be an algebraic superstack.
\begin{enumerate}
\item
The diagonal $\Delta_\Xf\colon \Xf \to \Xf\times\Xf$ is of finite type  (Definition \ref{def:finitetype}).
\item
$\Xf$ is a Deligne--Mumford superstack if and only if the diagonal $\Delta_\Xf$ is unramified (Definition \ref{def:morphismsprops2}).
\item $\Xf$ is a Deligne--Mumford superstack if and only if, for every point $x$, the stabilizer $\mathfrak{A}ut_{\Xf(k)}(x)$ is a discrete and reduced ordinary group scheme.
\end{enumerate}
\end{prop}
\begin{proof} Let $\Xf$ be an algebraic (resp.\! Deligne--Mumford) superstack. Take a smooth (resp.\! \'etale) presentation $p\colon\Xcal \to \Xf$ from a superscheme $\Xcal$. Since $p\times p\colon\Xcal\times\Xcal \to \Xf\times\Xf$ is also surjective and smooth (resp.\! \'etale), to prove that $\Delta_\Xf\colon \Xf \to \Xf\times\Xf$ is of finite type (resp.\! unramified) it is enough to prove that its base-changed morphism
$$
(\Xcal\times\Xcal)\times_{p\times p,\Xf\times\Xf,\Delta_\Xf}\Xf\iso \Xcal\times_\Xf\Xcal \to \Xcal\times\Xcal
$$
is of the same type. Take an \'etale presentation $\Vc \to \Xcal\times_\Xf\Xcal$ of the algebraic space $\Xcal\times_\Xf\Xcal$.
The second projection $p_2\colon\Xcal\times_\Xf\Xcal\to \Xcal$ is a smooth (resp.\! \'etale) morphism of algebraic superspaces and then the composition $\Vc \to \Xcal\times_\Xf\Xcal\xrightarrow{p_2} \Xcal$ is a smooth (resp.\! \'etale) superscheme morphism, in particular it is of finite type (resp.\! unramified). In the first case, the composition $\Vc \to \Xcal\times_\Xf\Xcal\to \Xcal\times\Xcal$  is of finite type and in the second it is unramified by Proposition \ref{prop:unramcomp}. This proves (1) and the ``only if'' part of (2).

To prove that an algebraic superstack with unramified diagonal is Deligne--Mumford, we follow the proof of \cite[Th\'eor\`eme 8.1]{LauMor-Bai00} or \cite[Thm.\! 8.3.3]{Ol16}.  
Adapting the arguments used there, it suffices to find, for every point $x\in|\Xf|$ of finite type, a superscheme $\Vc$ and a representable \'etale  morphism $\Vc\to\Xf$ such that $x$ is the image of a point $y\in \Vc$.  We adjust to the super setting the rest of the proof of \cite[Th\'eor\`eme 8.1]{LauMor-Bai00}.  
Let   $p\colon \Xcal \to \Xf$  be a smooth presentation from an affine superscheme and let 
 $\Omega_{\Xcal/\Xf}$ be  the sheaf given   in equation \eqref{eq:descentomega};  by Lemma \ref{lem:propadhoc}       there is an epimorphism 
$$
\Omega_\Xcal \to \Omega_{\Xcal/\Xf}
$$
of coherent sheaves on $\Xcal$, and   $\Omega_{\Xcal/\Xf}$ is a locally free graded $\Oc_\Xcal$-module of finite rank.

Let $r|s$ be the rank of $\Omega_{\Xcal/\Xf}$ in a neighborhood of $x$. Since $\Omega_\Xcal \to \Omega_{\Xcal/\Xf}$ is surjective, there exist $r$ global even sections $f_1,\dots,f_r$ and $s$ global odd sections $\theta_1,\dots,\theta_s$ of $\Oc(\Xcal)$ whose differentials at $x$ form a basis of the fiber $\Omega_{\Xcal/\Xf}(x)$ of $\Omega_{\Xcal/\Xf}$ at $x$ as a $\kappa(x)$-graded vector space. Let us consider the diagram 
$$
\xymatrix{
\Xcal \ar[rr]^{\psi=(p,f,\theta)}\ar[rd]^p&&\Xf \times \As^{r|s}\ar[ld]_{\mathrm{p}_1}
\\
&\Xf&
}
$$
The slanted arrows are smooth and representable, and the differential of $\psi$ at the point $x$ is an isomorphism. It follows that $\psi$ is representable and that it is \'etale on an open neighborhood $\Wc\subseteq\Xcal$ of $x$. As in \cite{LauMor-Bai00}, we can find a sub-superscheme $\Ycal$ of $\As^{r|s}$ such that if $\Vc$ is the pre-image of $\Xf\times\Ycal$ under $\psi$, the restriction $p_{|\Vc}\colon \Vc\to \Xf$ is \'etale and $y$ is in its image, thus finishing the proof of (2).

We   prove that (2) is equivalent to (3). As  the diagonal of $\Xf$ is of finite type by (1), it is unramified if and only if its geometrical fibers are formally unramified. Then the diagonal  is unramified  if and only if, for every field $k$ and every point $x\colon \Spec(k)\to \Xf$,  the stabilizer $\mathfrak{A}ut_{\Xf(k)}(x)$ (see equation \eqref{eq:stabilizer2}), which is an algebraic superspace over $k$ by the representability of the diagonal (Proposition \ref{prop:represdiag}), is unramified. By Lemma \ref{lem:unram}, if $\mathfrak{A}ut_{\Xf(k)}(x)$ is unramified, then it is actually an unramified ordinary algebraic space and then, it is a discrete union of spectra of separable extensions of $k$ by \cite[Proposition 29.35.12]{Stacks} (see also \cite[Theorem 3.6.4]{Alp24}). Conversely, if $\mathfrak{A}ut_{\Xf(k)}(x)$ is a discrete union of spectra of separable extensions of $k$, then it is unramified.
\end{proof}

\begin{remark} It should be stressed that the stabilizers of a Deligne--Mumford superstack are ordinary group schemes, i.e., they have no odd (fermionic) part, as we saw in Proposition \ref{prop:unramdiag1} (3).
\end{remark}

\subsection{Characterization of algebraic superspaces}
\label{subsec:algdiag}

An important consequence of  Proposition \ref{prop:unramdiag1} is that algebraic superspaces can be characterized in terms of their  diagonal. In other words, the following Proposition states that, as one expects, algebraic spaces are precisely the algebraic stacks that are spaces.

\begin{prop}\label{prop:algsp} 
Let  $\Xf$  be a superstack. The following conditions are equivalent:
\begin{enumerate}
\item  $\Xf$  is an algebraic superstack and a superspace;
\item $\Xf$ is an algebraic superspace.
\end{enumerate}
\end{prop}
\begin{proof} It is obvious that an algebraic superspace is both  an algebraic superstack and   a superspace. Let $\Xf$  be an algebraic superstack that is a superspace. The diagonal $\Delta_\Xf\colon \Xf\to \Xf\times\Xf$ is injective and superschematic (Proposition \ref{prop:injsep1}). Then, for every morphism $\Tc\to  \Xf\times\Xf$ from a superscheme $\Tc$, the base change $\Xf\times_{\Xf\times\Xf}\Tc \to \Tc$ is a superscheme monomorphism, and by Propositions \ref{prop:unramdiag1} and \ref{prop:mono}, it is unramified, so that $\Delta_\Xf\colon \Xf\to \Xf\times\Xf$ is unramified and then $\Xf$ is a Deligne--Mumford superstack. Now, take an \'etale presentation $p\colon \Xcal\to\Xf$ of $\Xf$. Since the diagonal is superschematic, $p$ is superschematic as well by Proposition \ref{prop:diag}, which proves that $\Xf$ is an algebraic superspace.
\end{proof}

We can  extend Proposition \ref{prop:quot} to see that taking quotients of \'etale equivalence relations of algebraic superspaces does not yield more general objects. The proof adapts that of \cite[Prop.\! 1.3]{LauMor-Bai00}.

\begin{prop}\label{prop:quot2}
The quotient \'etale sheaf of an \'etale equivalence relation of algebraic superspaces is an algebraic superspace.
\end{prop}
\begin{proof}  Let $(q_1,q_2)\colon \Rf \rightrightarrows \Uf$ be an \'etale equivalence relation of algebraic superspaces and let $\Xf$ be the quotient \'etale sheaf. By Proposition \ref{prop:algsp}, it is enough to prove that it is  a Deligne--Mumford superstack. This is done by showing that if $\Ucal \to \Uf$ is a (superschematic) \'etale presentation of the algebraic superspace $\Uf$, the composition $\Ucal \to \Uf \xrightarrow{q}\Xf$  is \'etale, surjective and representable. It is clearly surjective, so  we only need to prove that  it is representable and \'etale.

Let us write $\Rf_\Ucal:=\Rf\times_{\Uf\times\Uf}(\Ucal\times\Ucal)$ and consider the induced \'etale equivalence relation $\Rf_\Ucal\rightrightarrows \Ucal$. One checks that $\Xf$ is also the quotient sheaf of this equivalence relation.  Since $\Rf \to \Uf\times\Uf$ is a monomorphism locally of finite type, so is $\Rf_\Ucal \to \Ucal\times\Ucal$. Then $\Rf_\Ucal$ is a superscheme by Proposition \ref{prop:injsep1} and then $\Rf_\Ucal\rightrightarrows \Ucal$  is an \'etale equivalence relation of superschemes. Thus $\Xf$ is an algebraic superspace by Proposition \ref{prop:quot}. 
\end{proof}

 The following result is an immediate consequence:
\begin{corol}[\'Etale descent for algebraic superspaces]\label{cor:algssdescent} Let $\Ucal\to\Tc$ be an \'etale covering of superschemes and $(q_1,q_2)\colon \Rcal=\Ucal\times_\Tc\Ucal \rightrightarrows \Ucal$. Let $p\colon\Xf\to\Ucal$ be a $\Ucal$-algebraic superspace and $\alpha\colon \Xf\times_{\Ucal,q_1}\Rcal \to \Xf\times_{\Ucal,q_2}\Rcal$ isomorphisms of $\Rcal$-algebraic superspaces satisfying the cocycle condition $q_{12}^\ast\alpha\circ q_{23}^\ast \alpha=q_{13}^\ast\alpha$. Then there exists a $\Tc$-algebraic superspace $\Yf$ and an isomorphism $\phi\colon \Xf \to \Yf\times_\Tc\Ucal$ of algebraic superspaces such that $q_1^\ast\phi=q_2^\ast\phi\circ\alpha$.
\qed
\end{corol}
\begin{proof} Let $j\colon \Rcal \to \Ucal\times_\Tc\Ucal$ be the induced morphism and $\Xcal'=\Rcal\times_{j,\Ucal\times_\Tc\Ucal,(p,p)}(\Xcal\times_\Tc\Xcal)$. Then $\Xcal'$ is an algebraic superspace and we have an \'etale equivalence relation of algebraic superspaces $\Xcal'\rightrightarrows \Xcal$. By Proposition \ref{prop:quot2}, the quotient \'etale sheaf $\Yf$ is an algebraic superspace that satisfies the condition of the statement.
\end{proof}

One also has a property of algebraic superspaces with various consequences that    are important on their own:
\begin{prop}\label{prop:fppfsheaf} An algebraic superspace is a sheaf for the fppf topology of superschemes.
\end{prop}
\begin{proof} Let $\Xf$ be an algebraic superspace and $\Xf^\sharp$ the associated sheaf for the fppf topology. There is a natural morphism of functors $q\colon\Xf \to \Xf^\sharp$, which we consider as a sheaf morphism  for the \'etale topology. We have to prove that $q$ is an isomorphism. Injectivity is proved as in \cite[Thm.\! 5.5.2]{Ol16}: if $(x_1, x_2)$ are morphisms $\Tc \to \Xf$ from a superscheme $\Tc$ whose compositions with $\Xf\to\Xf^\sharp$ coincide, there exists an fppf covering $\Tc'\to \Tc$ and a morphism $\phi'\colon \Tc' \to \Zf:= \Tc\times_{\Xf\times\Xf}\Xf$ such that the following diagram
$$
\xymatrix{ & \Zf \ar[r]\ar[d] & \Xf \ar[d]^\Delta \\
\Tc' \ar[r] \ar[ru]^{\phi'} & \Tc \ar[r]^{(x_1,x_2)} & \Xf\times\Xf
}
$$
  commutes.  Since the diagonal of an algebraic superspace is superschematic (Proposition \ref{prop:represdiag}), $\Zf$ is a superscheme. Thus, it is a sheaf for the fppf topology --- see Remark \ref{rem:subcanonical} --- and then $\phi'$ descend to a superscheme morphism $\phi\colon \Tc \to \Zf$. This implies  that $x_1=x_2$, thus proving the injectivity of $q$.

By Proposition \ref{prop:quot} we can write $\Xf$ as the quotient of an \'etale equivalence relation $\Rcal\rightrightarrows\Ucal$ of superschemes. If $\Ucal' \to \Ucal$ is an \'etale and surjective morphism, we can also write $\Xf$ as the quotient of the induced \'etale equivalence relation $\Rcal'\rightrightarrows \Ucal'$, where $\Rcal'=\Rcal\times_{\Ucal\times\Ucal}(\Ucal'\times\Ucal')$. So we may assume that $\Rcal\rightrightarrows\Ucal$ is the disjoint union of \'etale equivalence relations $\Rcal_i\rightrightarrows\Ucal_i$, where the $\Ucal_i$'s are affine open subsuperschemes of $\Ucal$. If $\Xf_i$ is the quotient algebraic superspace of $\Rcal_i\rightrightarrows\Ucal_i$, then $\Xf_i$ is an open  algebraic sub-superspace of $\Xf$ and $\Xf$ is the union of the $\Xf_i$'s. Thus  we are reduced to the case when there is a surjective \'etale morphism $\pi\colon\Ucal\to\Xf$ from an affine superscheme, so that $\pi$ is quasi-affine. 

 The rest of the proof follows   \cite[Thm.\! A.4]{LauMor-Bai00}. To prove that $q$ is surjective it suffices to show that $q\circ\pi\colon \Ucal \to \Xf^\sharp$ is a surjective morphism of \'etale sheaves. Actually, we are going to see that it is superschematic, quasi-affine, \'etale and surjective, using the fact that  $\pi\colon \Ucal \to \Xf$ has such properties. Thus we have to prove that for any sheaf morphism $x\colon \Tc\to \Xf^\sharp$, the fiber product $\Tc\times_{\Xf^\sharp,q\circ\pi}\Ucal$ is a superscheme, and   the projection $\Tc\times_{\Xf^\sharp,q\circ\pi}\Ucal\to\Tc$ is quasi-affine, \'etale and surjective. By definition of the sheaf associated to a presheaf, there is an fppf cover $\psi\colon\Tc'\to\Tc$ such that the restriction $x\circ\psi$ lifts to a morphism $x'\colon \Tc' \to \Ucal$. Now we have a commutative diagram
$$
\xymatrix@C=6pt{\Ucal\times_{q\circ\pi,\Xf^\sharp,x\circ\psi}\Tc' \ar[dd]^{p'_2}\ar[rr] && \Ucal\times_{q\circ\pi,\Xf^\sharp,x}\Tc \ar[dd]^{p_2}\ar[rr] && \Ucal\ar[dl]_\pi\ar[dd]^{q\circ\pi}\\
& \Ucal\times_{\pi,\Xf,x'}\Tc' \ar[ul]_\simeq \ar[dl]_{\pi_{\Tc'}} \ar[rrru]& & \Xf\ar[dr]_q & \\
\Tc' \ar[rr]^\psi \ar[rrru]^{x'}& & \Tc \ar[rr]^x && \Xf^\sharp
}
$$
where the isomorphism $\Ucal\times_{\pi,\Xf,x'}\Tc' \simeq \Ucal\times_{q\circ\pi,\Xf^\sharp,x\circ\psi}\Tc'$ is due to the injectivity of $q\colon\Xf\to\Xf^\sharp$, and the two squares and the slanted quadrilateral are 2-cartesian. Since $\pi$ is superschematic, quasi-affine, \'etale and surjective, $\Ucal\times_{\pi,\Xf,x'}\Tc'$ is a superscheme and the base change morphism $\pi_{\Tc'}$ is quasi-affine, \'etale and surjective. Thus, $\Ucal\times_{q\circ\pi,\Xf^\sharp,x\circ\psi}\Tc'$ is a superscheme and the projection $p'_2$ is quasi-affine, \'etale and surjective. By fppf descent for quasi-affine morphisms (Proposition \ref{prop:qadescent}), $\Ucal\times_{\Xf^\sharp}\Tc$ is a superscheme and $p_2$ is quasi-affine. Moreover the morphism $p_2$  is also \'etale and surjective, as these conditions are local on the target for the fppf topology of superschemes (Proposition \ref{prop:localproperties}). This finishes the proof.
\end{proof}

 We can now generalize Proposition \ref{prop:quotstack} as announced in Remark \ref{rem:quotstack}.
\begin{corol}\label{cor:fppfsheaf} Let $\Gsc\to\Sc$ be an fppf and affine $\Sc$-group superscheme over a superscheme $\Sc$, and let $\Xf$ be an algebraic superspace with a
$\Gsc$-action. The quotient stack $[\Xf/\Gsc]\to\Sf/\Sc$ is also an fppf superstack. In particular, the classifying superstack $\Bf\Gsc$ is an fppf-superstack.
\end{corol}
 
We note another  result, easily obtained  yet interesting: as morphisms between algebraic superspaces are just morphisms of \'etale sheaves on superschemes, the following holds.

\begin{prop}\label{prop:algsspsubcan}  The site of algebraic superspaces with the \'etale topology is subcanonical, i.e., every functor  from it to $(Sets)$ that is representable by an algebraic superspace is a sheaf.
\qed
\end{prop}

\section{Further constructions and moduli problems}

\subsection{Principal superbundles and quotient superstacks revisited, supertorsors and supergerbes}

In this section we generalize definition \ref{def:principalbundle} to include  algebraic superspaces and not only superschemes. Moreover, we introduce supertorsors and see that, under some mild requirements, they are equivalent to principal superbundles. We finish this section by introducing supergerbes and providing some examples of them.

\subsubsection*{Principal superbundles and quotient superstacks revisited}

First note that for any superscheme $\Sc$, group $\Sc$-superschemes, as defined in Definition \ref{def:supergroup}, are particular instances of group objects in the category of \'etale sheaves on $\Sf/\Sc$, that is, in the category of $\Sc$-superspaces.
There are also a natural notion of left (and right) action of a group superspace  on an  $\Sc$-superspace, and of equivariant morphism of $\Sc$-superspaces acted on by a group superspace. When the group superspace and the superspace upon which it acts are superschemes, we recover Definitions \ref{def:action} and \ref{def:equivariant}.
We can then generalize Definition \ref{def:principalbundle} in the obvious way:
\begin{defin}\label{def:principalbundle2}
Let $\Gf\to\Sc$ be a group  $\Sc$-algebraic superspace, i.e., a group object in the category of $\Sc$-algebraic spaces (see Definition \ref{def:groupalgss}). A principal $\Gf$-superbundle on an $\Sc$-superscheme $\Tc$ is a morphism  $\pi\colon \Pf\to \Tc$ of $\Sc$-algebraic superspaces together with a left  action $\psi\colon \Gf\times_\Tc\Pf \to \Pf$ satisfying the following properties:
\begin{enumerate}
\item if we consider the trivial action of $\Gf$ on $\Tc$, then $\pi\colon \Pf\to \Tc$ is $\Gf$-equivariant;
\item the morphism
$(\psi, p_2)\colon\Gf\times_\Sc \Pf\to \Pf\times_\Tc \Pf$
is an isomorphism.
\end{enumerate}
\end{defin}

Many definitions and properties listed in Subsection \ref{subs:psuper} are easily adapted to this setting. We recall some of them.

Morphisms of principal $\Gf$-superbundles  over the same superscheme $\Sc$ are $\Gf$-equivariant morphisms of superspaces  commuting with the projections to $\Sc$; every morphism of  principal $\Gf$-superbundles over the same superscheme is an automorphism, as in Proposition \ref{prop:alliso}. A principal $\Gf$-superbundle $\Pf\to \Tc$ is \emph{trivial} if there is an isomorphism  $\Pf \iso \Gf\times_\Sc \Tc$ of $\Gf$-superbundles. This is equivalent to the existence of a section of $\Pf\to\Tc$. 
 Moreover, principal superbundles are preserved under base change, that is, if $\phi\colon\Tc'\to\Tc$ is a morphism of $\Sc$-superschemes and $\Pf\to \Tc$ is a principal $\Gf$-bundle, the pullback $\phi^\ast\Pf:=\Pf\times_\Tc \Tc'\to \Tc'$ is a principal $\Gf$-superbundle.

\begin{prop}[\'Etale descent for principal $\Gf$-superbundles]\label{prop:psbdescent} If $\Gf$ is a group  algebraic  superspace,  principal $\Gf$-superbundles satisfy effective descent for the \'etale topology of superschemes. In other words, let $q\colon\Ucal\to\Tc$ be an \'etale covering of $\Sc$-superschemes and $(q_1,q_2)\colon \Rcal=\Ucal\times_\Tc\Ucal \rightrightarrows \Ucal$. Let $\Pf\to \Ucal$ be a principal $\Gf$-superbundle and $\alpha\colon q_i^\ast \Pf \to q_2^\ast \Pf$ an isomorphism of principal $\Gf$-superbundles over $\Rcal$ satisfying the cocycle condition $q_{12}^\ast\alpha\circ q_{23}^\ast \alpha=q_{13}^\ast\alpha$. Then there exist a principal $\Gf$-superbundle $\Pf'\to\Tc$ and an isomorphism $\phi\colon \Pf \to q^\ast\Pf'$ of principal $\Gf$-superbundles such that $q_1^\ast\phi=q_2^\ast\phi\circ\alpha$.
\end{prop}
\begin{proof} By Corollary \ref{cor:algssdescent}, $\Pf\to \Ucal$ descends to an algebraic superspace morphism $\Pf'\to\Tc$  such that there is  an isomorphism $\phi\colon \Pf \iso q^\ast\Pf'$. Moreover, since algebraic superspaces are sheaves for the \'etale topology of \emph{algebraic spaces} (Proposition \ref{prop:algsspsubcan}), the action $\psi\colon \Gf\times_\Tc\Pf \to \Pf$ descends to an action $\psi'\colon \Gf\times_\Tc\Pf' \to \Pf'$ that makes $\Pf'\to\Tc$ into a principal $\Gf$-superbundle.
\end{proof}

We can  modify the Definition \ref{def:quotstack} of a quotient stack  utilizing the notion of principal superbundle.
 
\begin{defin}\label{def:quotstack2}
Let $\Sc$ be a superscheme,   $\Gf\to\Sc$ an $\Sc$-group algebraic   superspace, and $\mu\colon \Gf\times\Xf \to \Xf$  a left action of $\Gf\to\Sc$ on an algebraic $\Sc$-superspace $\Xf$.
The category fibered in grupoids 
$p_{[\Xf/\Gf]}\colon [\Xf/\Gf] \to \Sf_{et}/\Sc$ is
given by:
\begin{enumerate}
\item objects of $[\Xf/\Gf](\Tc)$ over an $\Sc$-superscheme $\Tc\to\Sc$ are diagrams
$$
\xymatrix{ \Pf \ar[d]^\pi\ar[r]^f & \Xf \\
\Tc &
}
$$
where $\pi\colon\Pf\to\Tc$ is a principal $\Gf$-superbundle  (Definition \ref{def:principalbundle2}), and $f\colon \Pf\to\Xf$ is a $\Gf$-equivariant morphism of $\Sc$-superspaces;
\item  morphisms are pairs of $\Sc$-superspace morphisms $\phi\colon \Pf'\to\Pf$, $\psi\colon \Tc'\to\Tc$, where $\phi$ is $\Gf$-equivariant, such that the diagram
$$
\xymatrix{
\Pf' \ar[d]^{\pi'}\ar[r]_\phi \ar@/^1.5pc/[rr]^{f'} \ar@{}[rd] |{\square} & \Pf\ar[d]^\pi\ar[r]_f & \Xf \\
\Tc'\ar[r]_\psi& \Tc
}
$$
commutes and the square is 2-cartesian. 
\end{enumerate}
One easily sees that $p_{[\Xf/\Gf]}\colon [\Xf/\Gf] \to \Sf_{et}/\Sc$  is a pre-superstack which is called the quotient pre-superstack of the action of $\Gf$ on $\Xf$.
\end{defin}

\begin{defin}\label{def:class}
The classifying pre-superstack of  a group algebraic superspace $\Gf\to\Sc$ is the quotient pre-superstack $\Bf\Gf:=[\Sc/\Gf]\to\Sf_{et}/\Sc$, where $\Gf$ acts on $\Sc$   trivially.
\end{defin}

Actually, the quotient pre-superstack is a superstack, as the following Proposition shows.

\begin{prop}\label{prop:class} The quotient pre-superstack $p_{[\Xf/\Gf]}\colon [\Xf/\Gf] \to \Sf_{et}/\Sc$ is a superstack. In particular, the classifying pre-superstack $\Bf\Gf$ is a superstack.
\end{prop}

\begin{proof} We have to prove  that descent data are effective. Let $q\colon\Ucal \to \Tc$ be an \'etale cover of $\Sc$-superschemes and $(\Pf,f)$   an object of the category $[\Xf/\Gf](\Ucal)$, that is, a $\Gf$-principal superbundle $\pi\colon \Pf\to\Ucal$ together with an equivariant superstack morphism $f\colon\Pf\to\Xf$. 
Let $$(q_1,q_2)\colon\Rcal=\Ucal\times_\Tc\Ucal \rightrightarrows \Ucal$$ be the projections and let $\alpha\colon q_1^\ast\Pf \iso q_2^\ast \Pf$ be an isomorphism of $\Gf$-principal superbundles over $\Rcal$ satisfying the cocycle conditions $q_{12}^\ast\alpha\circ q_{23}^\ast \alpha=q_{13}^\ast\alpha$ and $q_1^\ast f = q_2^\ast f \circ\alpha$. By Proposition \ref{prop:psbdescent}  there exist a principal $\Gf$-superbundle $\Pf'\to\Tc$ and an isomorphism $\phi\colon \Pf \to q^\ast\Pf'$ of principal $\Gf$-superbundles such that $q_1^\ast f=q_2^\ast f\circ\alpha$ as morphisms $q^\ast\Pf \to \Xf$. Moreover, since  by Proposition \ref{prop:algsspsubcan} $\Xf$ is a sheaf for the \'etale topology of algebraic superspaces, morphisms to $\Xf$ glue uniquely in the \'etale topology, and   $f\colon \Pf\to\Xf$ descends to an algebraic space morphism $f'\colon \Pf'\to \Xf$.
\end{proof}

One has a superstack morphism
$$
q\colon \Xf \to [\Xf/\Gf]
$$
defined as in Subsection  \ref{subsec:quot}: Let $\Tc$ be a superscheme and $\phi\in\Xf(\Tc)$,  that we understand as a morphism  $\phi\colon\Tc\to \Xf$ of \'etale sheaves on $\Sf/\Sc$. Then
$$
\xymatrix{\Gf\times_\Sc\Tc \ar[d]^{p_2} \ar[r]^{1\times\phi} &\Gf\times_\Sc \Xf\ar[r]^\mu & \Xf
\\
\Tc}
$$
defines an object $q(\phi)$ of $[\Xf/\Gf](\Tc)$.

\begin{prop} \label{prop:quotstack2}
 With the notations of Definition \ref{def:quotstack2}, one has:
\begin{enumerate}
\item let $\Tc$ be an $\Sc$-superscheme and $\psi\colon \Tc \to [\Xf/\Gf]$ a superstack morphism over $\Sc$, and let 
$$
\xymatrix{ \Pf_\psi \ar[d]^\pi\ar[r]^{f_\psi} & \Xf \\
\Tc &
}
$$
be  the diagram corresponding to $\psi$ by the 2-Yoneda Lemma \ref{lem:2yoneda}. There is a 2-cartesian diagram
$$
\xymatrix{ \Pf_\psi \ar[r]^{f_\psi}\ar[d]^{\pi} & \Xf\ar[d]^q
\\
\Tc\ar[r]^\psi &  [\Xf/\Gf] .
}
$$
Thus $q\colon \Xf \to [\Xf/\Gf]$ is surjective and representable. Moreover, if $\Gf\to\Sc$ is fppf, $q$ is fppf as well. Also, if $\Gf\to\Sc$ is  smooth, $q$ is also smooth, and then  the quotient superstack $[\Xf/\Gf]$ is algebraic. In particular, the classifying superstack $\Bf\Gf=[\Sc/\Gf]$ of a smooth  group algebraic space $\Gf\to\Sc$ is an $\Sc$-algebraic superstack.
\item
There is a 2-cartesian diagram
$$
\xymatrix{ \Gf\times_\Sc\Xf \ar[r]^{p_2} \ar[d]^{p_1} & \Xf\ar[d]^q
\\
\Xf\ar[r]^q & [\Xf/\Gf] \,.
}
$$
\item If both $\Gf$ and $\Xf$ are smooth over $\Sc$, then $[\Xf/\Gf]$ is smooth over $\Sc$.
\end{enumerate}
\end{prop}
\begin{proof} 
 (1) and (2) are proved as in Propositions \ref{prop:quotmorphism} and \ref{prop:quotmorphism2}. As for  (3), the smoothness follows from Definition \ref{def:morphismsprops3}. 
\end{proof}

\begin{remark}\label{rem:ArtinThm} In the classical (bosonic)  setting, the quotient stack of the action of an fppf algebraic group  on an algebraic space is algebraic. To extend this result to the super setting one would  need a super analog to \cite[Thm.\! 7.3.1]{Alp24}, or what amount to the same, to Artin's Theorem \cite[Thm.\! 10.1]{LauMor-Bai00};  these theorems state that an fppf stack admitting an fppf representable morphism from an algebraic space  is an algebraic stack. The proof of the corresponding super statement is beyond the scope of these notes. \end{remark} 

 However, we have a related result.
\begin{prop} Let $\Sc$ be a superscheme, $\Gf\to\Sc$ a smooth $\Sc$-group algebraic  superspace acting on  an algebraic $\Sc$-superspace $\Xf$.
If the action of $\Gf$ on $\Xf$ is free, that is, if the natural morphism $\Gf\times_\Sc \Xf \to \Xf\times_\Sc\Xf$ is injective (see Definition \ref{def:action}), then the quotient superstack $[\Xf/\Gf]$ is an algebraic superspace.
\label{prop:algebraicspacequotient}
\end{prop}
\begin{proof} We know that $[\Xf/\Gf]$ is an algebraic superstack by Proposition \ref{prop:quotstack2}. Then  by Proposition \ref{prop:algsp} we have only to prove that it is a superspace,  namely,  that for every $\Sc$-superscheme $\Tc$, the category $[\Xf/\Gf](\Tc)$ is equivalent to a setoid. This is equivalent to saying that any object $\psi$ of $[\Xf/\Gf](\Tc)$ has no automorphisms other than the identity. Since the automorphisms of $\psi$ are a sheaf, we are reduced to prove that this is true after an \'etale covering of the base. Now, let 
$$
\xymatrix{ \Pf_\psi \ar[d]^\pi\ar[r]^{f_\psi} & \Xf \\
\Tc &
}
$$
be  the diagram corresponding to $\psi$ by the 2-Yoneda Lemma \ref{lem:2yoneda}. Here $\Bf_\psi$ is a $\Gf$-principal superbundle and $f_\psi$ is an equivariant morphism. The morphism $\pi\colon \Bf_\psi \to \Tc$ is smooth, and then, it has sections locally for the \'etale topology (Proposition \ref{prop:localsectsmooth}), or, in other words, it is trivialized, as a $\Gf$-principal superbundle, by an \'etale covering (see Definition \ref{def:action}). Thus, we may assume that  $\pi\colon \Bf_\psi \to \Tc$ is trivial. Since the automorphisms of a trivial principal superbundle over $\Tc$ are the elements of $\Gf(\Tc)$, an automorphism $\alpha$ of $\psi$ consists of an element $g\in \Gf(\Tc)$ such that $f_\psi(g(t)\cdot h,t)=f_\psi(h,t)$ for every $\Sc$-superscheme $\Tc'$ and every pair $t\in\Tc'(\Tc)$, $h\in \Gf(\Tc')$. The equivariance of $\psi$ yields  $g(t)\cdot f_\psi(1,t)=f_\psi(1,t)$, so that $g(t)=1$ because the action of $\Gf$ on $\Xf$ is free. It follows that $\alpha$ is the identity,  thus finishing the proof.
\end{proof}

\subsubsection*{Supertorsors}

Torsors were first defined by Grothendieck as a categorical version of principal bundles. Actually, torsors can be defined for any site and any group object in it, see, for instance, \cite[Def.\! 6.4.1]{Alp24};
  in the case of the \'etale site of schemes,  when the group is algebraic they coincide,   in the sense of an   equivalence of categories, with principal bundles \red{(see Proposition \ref{prop:torsorprincipal}).}
Their general definition  includes torsors over the site $\Sf_{et}$ of superschemes for the \'etale topology, and also, over the more general  sites $\Sf_{smooth}$ and $\Sf_{fppf}$ of superschemes for the smooth or the fppf topology.  We can spell out these notions as follows.

Since we are going to consider different topologies, in this subsection we denote simply by $\Sf$ any of the sites $\Sf_{et}$ or $\Sf_{fppf}$ of superschemes for the \'etale or the fppf topology, respectively.

\begin{defin}[Supertorsors] \label{def:supertorsor} Let $\bG$ be a sheaf of groups on the site $\Sf$.  A $\bG$-supertorsor  is a sheaf $\bP$ on $\Sf$  with a left action $\psi\colon \bG\times\bP \to \bP$ such that:
\begin{enumerate}
\item for any superscheme $\Tc$, there is a covering $\{\Tc_i\to \Tc\}$ on the site $\Sf$ such that $\bP(\Tc_i)\neq 0$ for every $i$;
\item the  morphism $(\psi,p_2)\colon \bG\times\bP \to  \bP\times \bP$ is an isomorphism of sheaves on $\Sf$.
\end{enumerate}
\end{defin}

We have relative notions of supertorsors:
\begin{defin}\label{def:relsupertorsor} Let $\Sc$ be a superscheme.
\begin{enumerate}
\item  If $\bG$ is a sheaf of groups on the site $\Sf/\Sc$ of superschemes over $\Sc$, a $\bG$-supertorsor is a $\bG$-torsor over $\Sc$ on the site $\Sf/\Sc$.  
\item If $\bG\to\Sc$ is a sheaf of groups on the big site $\Sf$ of superschemes,   a $\bG$-supertorsor over $\Sc$ is a $(\bG\times \Sc)$-supertorsor on the site $\Sf/\Sc$.
\end{enumerate}
\end{defin}

Thus, we distinguish between three types of supertorsors --- \emph{\'etale}, \emph{smooth}, and \emph{fppf}   --- depending on the topology being considered.

Morphisms of  $\bG$-supertorsors are $\bG$-equivariant morphisms of sheaves.  One easily prove that,  for any of the above definitions, a morphism of $\bG$-supertorsors is an isomorphism. 

\begin{defin}\label{def:trivialsupertorsor} Let $\bG$ be a sheaf of groups on the site $\Sf$ (resp.\! $\Sf/\Sc$ for a superscheme $\Sc$). A $\bG$-supertorsor $\bP$ on $\Sf$ (resp.\! $\Sf/\Sc$) is trivial if is is isomorphic, as a supertorsor, to $\bG$.
\end{defin}

In particular, if $\Sc$ is a superscheme and $\bG$ be a sheaf of groups on the site $\Sf$, a $\bG$-supertorsor  $\bP\to \Sc$ over $\Sc$ is trivial if it is isomorphic, as a supertorsor, to  $\bG\times\Sc \to \Sc$.

There is a natural notion of pullback of relative supertorsors by superscheme morphisms $\phi\colon \Sc' \to \Sc$; if $\bF\to\Sc$ is a sheaf on $\Sf/\Sc$, the pullback of $\bF$ is the sheaf $\bF_{\vert \Sc'}:=\phi^\ast \bF$ on $\Sf/\Sc'$ whose sections on a $\Sc'$- superscheme $\Tc \to \Sc'$ are the sections of $\bF$ on the $\Sc$-superscheme $\Tc\to\Sc' \xrightarrow{\phi}\Sc$. Then, if $\bG$ is a sheaf of groups on $\Sf/\Sc$ and $\bP$ is a $\bG$-supertorsor over $\Sc$, $\phi^\ast\bP$ is a $\phi^\ast\bG$-supertorsor over $\Sc'$.  The following result corresponds to \cite[6.4.4]{Alp24}.

\begin{prop}\label{prop:supertorsorsproperies} Let $\Sc$ be a superscheme and $\bG$ a sheaf of groups on $\Sf/\Sc$.
\begin{enumerate}
\item if $\bP\to\Sc$ is a $\bG$-supertorsor over $\Sc$, then $\bP\to\Sc$ is  a sheaf  epimorphism;
\item if $\bP\to\Sc$ is a $\bG$-supertorsor over $\Sc$, then $\Sc$ is equivalent to the quotient sheaf $\bP/\bG$;
\item a $\bG$-supertorsor $\bP\to\Sc$ over $\Sc$ is trivial if and only if there is a section $\Sc\to\bP$ of $\bP\to\Sc$;
\item a sheaf $\bP$ on $\Sf/\Sc$ is a $\bG$-supertorsor if and only if it is locally trivial, that is, if there exists a covering $\{\Sc_i\to\Sc\}$ and isomorphisms $\bP_{\vert \Sc_i}\simeq \bG_{\vert \Sc_i}$ of $\bG_{\vert \Sc_i}$-supertorsors.
\end{enumerate}
\qed
\end{prop}

 As we already mentioned, supertorsors are closely related to principal superbundles. For smooth group algebraic superspaces we have:
\begin{prop} \label{prop:torsorprincipal}  Let $\Sc$ be a superscheme and $\Gf\to\Sc$ a smooth group $\Sc$-algebraic superspace. There is an equivalence of categories between the category of $\Gf$-principal superbundles Definition \ref{def:principalbundle2} and the category of \'etale $\Gf$-supertorsors (Definition \ref{def:relsupertorsor}).
\end{prop}
\begin{proof} Let $\Pf\to\Tc$ be a $\Gf$-principal superbundle on a $\Sc$-superscheme $\Tc$. Then $\Pf\to\Tc$ is smooth, so that it has sections locally for the \'etale topology by Proposition \ref{prop:localsectsmooth}; so $\Pf$ is a supertorsor by Proposition \ref{prop:supertorsorsproperies}. Conversely, if $\bP$ is a $\Gf$-supertorsor over $\Sc$, again by Proposition \ref{prop:supertorsorsproperies} there is an \'etale covering $\{\Sc_i\to \Sc\}$ such that $\bP_{\vert \Sc_i}$ is trivial for every $i$ and therefore  is a $\Gf$-principal superbundle $\bP_{\vert \Sc_i}$. By Proposition  \ref{prop:psbdescent} on descent for principal superbundles, $\bP$ is also a $\Gf$-principal superbundle. Using this, one easily completes the proof.
\end{proof}
 
 When  $\Gf\to\Sc$ is not necessarily smooth we have the following result:
\begin{prop}\label{prop:torsorsPsb} Let $\Sc$ be a superscheme and $\Gf\to\Sc$  an fppf group $\Sc$-algebraic superspace. There is a fully faithful functor from the category of $\Gf$-principal superbundles (Definition \ref{def:principalbundle2}) to the category of $\Gf$-supertorsors for the fppf topology (Definition \ref{def:relsupertorsor}).
\end{prop}
\begin{proof} Let $\Pf\to\Tc$ be a $\Gf$-principal superbundle on an $\Sc$-superscheme $\Tc$. By Proposition \ref{prop:fppfsheaf} $\Pf$ is a sheaf on $\Gf_{fppf}/\Sc$.  Since  $\Pf\times_\Tc\Pf \simeq \Gf\times_\Sc \Tc$, condition (2) of Definition \ref{def:relsupertorsor} is fulfilled, and, as $\Pf \to \Tc$ is a covering for the fppf topology, condition (1) is fulfilled as well, so that $\Pf$ is a $\Gf$-supertorsor. The proof that this defines a fully faithful functor is straightforward.
\end{proof}
  
\begin{remark} In the bosonic case the functor of Proposition \ref{prop:torsorsPsb} is an equivalence of categories \cite[Ex.\! 6.4.5]{Alp24}. A proof in the super setting would need the results mentioned  in Remark \ref{rem:ArtinThm}.
\end{remark}

\subsubsection*{Supergerbes}

The general notion of gerbe on a site is given, for instance, in \cite[Def.\! 6.2.6]{Alp24}. If we apply that definition to the site $\Sf_{fppf}$ of superschemes for the fppf topology,  we get  the notion of supergerbe. We spell out the definition as follows.
\begin{defin}[Supergerbes]\label{def:supergerbe} A supergerbe is an fppf-superstack $\Xf$ fulfilling the following conditions:
\begin{enumerate}
\item for every superscheme $\Tc$, there is an fppf covering $\{\Tc_i\to\Tc\}$  such that all the categories $\Xf(\Tc_i)$ are non-empty;
\item given objects $x$, $y$ of $\Xf(\Tc)$, there exist an fppf covering $\{\Tc_i\to\Tc\}$ and isomorphisms $\rest{x}{\Tc_i} \iso \rest{y}{\Tc_i}$ for every $i$.
\end{enumerate}
A gerbe $\Xf$ is trivial it there is a section $\Sf_{et}\to\Xf$ of the projection functor $\Xf \to \Sf_{et}$.
\end{defin}

In a similar way, we define supergerbes on a superscheme $\Sc$ as gerbes for the small site $\Sf_{et}/\Sc$. One easily sees that a supergerbe over a superscheme $\Sc$ is trivial if and only if the category $\Xf(\Sc)$ is nonempty.

 Classifying superstacks are important instances of supergerbes. We have:
\begin{prop}  Let $\Gsc\to\Sc$ be an fppf   affine $\Sc$-group superscheme over a superscheme. The classifying superstack $\Bf\Gsc$ is a supergerbe.
\end{prop}
\begin{proof} Remember that $\Bf\Gsc$ is an fppf superstack by Corollary \ref{cor:fppfsheaf}. The result now follows from the fact that, since $\Gf\to\Sc$ is fppf, every $\Gf$-principal superbundle is locally trivial for the fppf topology.
\end{proof}

%
 The notion of banded gerbe on a site (see, for instance \cite[Def.\! 6.4.10]{Alp24}) can be applied directly to the site $\Sf_{fppf}$.
The definition is as follows:
\begin{defin}\label{def:bandedsg} 
Let $\bG$ be an abelian sheaf on the site $\Sf_{fppf}$. A supergerbe $\Xf$ is  banded by $\bG$ if for every superscheme $\Tc$ and every object $x\in \Xf(\Tc)$, there exist isomorphisms $\psi_x\colon \rest{\bG}{\Tc} \iso \mathfrak{A}ut_{\Xf(\Tc)}(x)$ such that for every isomorphism $\alpha\colon x \to y$ in $\Xf(\Tc)$ the diagram
$$
\xymatrix{
& \rest{\bG}{\Tc} \ar[ld]^{\phi_x}\ar[rd]^{\phi_y} &
\\
\mathfrak{A}ut_{\Xf(\Tc)}(x) \ar[rr]^{\operatorname{Inn}_\alpha} & & \mathfrak{A}ut_{\Xf(\Tc)}(y)
}
$$
where $\mathfrak{A}ut_{\Xf(\Tc)}(x)(\tau)= \alpha^{-1}\tau\alpha$, is commutative. A morphism of banded $\bG$-supergerbes is a morphism of superstacks compatible with the bands.
\end{defin}

\subsection{Coarse moduli superspaces of algebraic superstacks}

 In the usual theory of stacks, the \emph{coarse moduli space} $X$ of an algebraic stack $\Xf$ 
is in a sense the best approximation to $\Xf$  by 
 a classical geometric object, such as a scheme or an algebraic space. It carries --- actually, it is defined by  ---  a universal morphism $\pi: \Xf  \to X$ that ``collapses'' the   automorphism groups of the objects parameterized by the stack, inducing a   bijection between the isomorphism classes of objects parametrized by the stack (the geometric points of the stack) and the geometric points of $X$.
 The extension of this notion to the realm of superstacks is straightforward.

\begin{defin}\label{def:coarse} A morphism $\pi\colon \Xf \to \Xf_{coar}$ from an algebraic superstack to an algebraic superspace is a coarse moduli superspace if:
\begin{enumerate}
\item   for every algebraically closed field $k$, the induced map $(\Xf(k)/\sim)\to \Xf_{coar}(k)$, from
the set of isomorphism classes of objects of $\Xf$ over $k$, is bijective, and
\item $\pi$ is universal for maps to algebraic superspaces, i.e., every map $\Xf \to \Yf$ to an
algebraic superspace factors uniquely as
$$
\xymatrix{ \Xf \ar[d]^\pi \ar[rd] & \\
\Xf_{coar} \ar@{..>}[r] & \Yf
}
$$
where the dotted arrow is a morphism of algebraic superspaces.
\end{enumerate}
\end{defin}

The existence of coarse moduli superspaces is a key point in a super GIT theory. In the classic, or bosonic case, the Keel-Mori Theorem ensures the existence of a coarse moduli space for an algebraic stack with finite inertia \cite{KeelMori97} (see also \cite[4.4.2]{Alp24}). A Keel-Mori theorem for Deligne--Mumford 
 superstacks has ben recently proved in \cite{Peng}.
  In Section \ref{sec:examples} we  shall review some  examples of coarse moduli superspaces.

\begin{defin}\label{def:geomquot} Let $\Gsc \to \Sc$ be a fppf group superscheme over a superscheme $\Sc$, acting on an $\Sc$-algebraic superspace $\Uf$. A geometric quotient of the  action of $\Gsc$ on $\Uf$ is a $\Gsc$-equivariant morphism $\Uf \to \Xf_{coar}$ of algebraic superspaces over $\Sc$ such that the induced superstack morphism $[\Uf/\Gsc]\to \Xf_{coar}$ is a coarse moduli superspace in the sense of Definition \ref{def:coarse}.
\end{defin}

\subsection{Universal families}
Since the existence of universal families for algebraic stacks \cite[3.1.7]{Alp24} is completely formal, analogous arguments prove the existence and properties of universal families for algebraic superstacks. We then have:

\begin{defin}\label{def:gen2yoneda} Let $\Xf$ be a superstack and $\Yf$ an algebraic superstack.  Consider a smooth presentation $\Ucal\to \Yf$ from a superscheme $\Ucal$ (Definition \ref{def:DMalgsstack}).
The category $\Xf(\Yf)$  is  the equalizer $\Xf(\Yf)$ of
\begin{equation}
 \Xf(\Ucal) \rightrightarrows \Xf(\Ucal\times_\Yf \Ucal) \mathrel{\substack{\textstyle\rightarrow\\[-0.6ex]
                      \textstyle\rightarrow \\[-0.6ex]
                      \textstyle\rightarrow}} \Xf(\Ucal\times_\Yf\Ucal\times_\Yf\Ucal)
\end{equation}
i.e., objects of $\Xf(\Yf)$ are pairs $(a,\alpha)$ where $a \in\Xf(\Ucal)$ and $\alpha\colon p_1^\ast a \iso  p_2^\ast a$ is an isomorphism satisfying the cocycle condition $p_{23}^\ast \alpha\circ p_{12}^\ast \alpha = p_{13}^\ast \alpha$ while morphisms $(a,\alpha)\to (a',\alpha')$ are morphisms $\beta\colon a \to a'$ satisfying  $p_2^\ast \beta \circ \alpha = \alpha'\circ p_1^\ast \beta$.
\end{defin}

Proceeding as in \cite[Lemma 3.1.25]{Alp24} one has:
\begin{lemma}[Generalized 2-Yoneda Lemma] \label{lem:gen2yoneda}
There is a natural equivalence of categories
$$
\Homst_{SSt}(\Yf,\Xf)\to \Xf(\Yf)\,.
$$
 In particular, $ \Xf(\Yf)$ is independent of the smooth presentation of $\Yf$. When $\Yf$ is a superscheme, the above equivalence coincides with the one given by the 2-Yoneda Lemma \ref{lem:2yoneda}. \qed
\end{lemma}

For every algebraic superstack $\Xf$, the identity $\Xf\to\Xf$ corresponds, by the generalized 2-Yoneda Lemma  \ref{lem:gen2yoneda},  to an object $u$ of $\Xf(\Xf)$, unique up to a unique isomorphism.

One can then define:
\begin{defin}\label{def:univfam}  The universal family over  the algebraic superstack $\Xf$ is the aforementioned object $u\in \Xf(\Xf)$.
\end{defin}

If $g\colon \Yf \to \Zf$ is a morphism of algebraic superstacks, by the generalized 2-Yoneda Lemma \ref{lem:gen2yoneda} there is a natural pullback functor $g^\ast\colon \Xf(\Zf) \to \Xf(\Yf)$ given by composition of morphisms of algebraic superstacks. One sees that:

\begin{prop}\label{prop:univfam} If $f_\alpha\colon \Ucal\to \Xf$ is the map from a superscheme corresponding to an object $a\in \Xf(\Ucal)$ by the 2-Yoneda Lemma \ref{lem:2yoneda}, and $f_\alpha^\ast\colon \Xf(\Xf)\to\Xf(\Ucal)$ is the pullback functor,   there is an isomorphism $a\iso f_\alpha^\ast u$ in $\Xf(\Ucal)$. \qed
\end{prop}

\section{Some examples} \label{sec:examples}

In this Section we describe some examples of moduli superstacks.

\subsection{The superstack of coherent sheaves}

We reproduce some results from \cite{LTD26}, where a more complete theory is developed.

\begin{defin}\label{def:cohcategory fibered in grupoids } Let $\Xcal$ be a superscheme.
The \emph{category fibered in grupoids  of coherent sheaves of super $\Oc_\Xcal$-modules}
$$
\mathfrak{C}oh_\Xcal \xrightarrow{p} \Sf_{et}
$$
is given by the following data:
\begin{itemize}
\item
objects are pairs $(\Tc,\M)$ where $\Tc$ is a superscheme, and $\M$ is a coherent sheaf of (super) $\Oc_{\Xcal\times \Tc}$-modules, flat over $\Tc$;
\item morphisms from $(\Tc',\M')$ to $(\Tc,\M)$ are pairs $(\phi,\psi)$ where $\phi\colon \Tc'\to\Tc$ is a superscheme morphism and $\psi\colon \M'\to(1\times\phi)^\ast\M$ is a homogeneous isomorphism of sheaves of (super) $\Oc_{\Xcal\times \Tc'}$-modules.
\end{itemize}
The functor $p$ sends $(\Tc,\M)$ to $\Tc$. In a similar way, we define the categories fibered in grupoids  $\Mf_\Xcal$ of locally free sheaves of $\Oc_\Xcal$-modules of arbitrary rank, and $\Mf_\Xcal(\mathbf{r})=\Mf_\Xcal(r|s)$ of locally free sheaves of rank  
$\mathbf{r}=r\vert s$.
\end{defin}

If $\Xcal\to\Sc$ is a superscheme over a superscheme $\Sc$, we obtain in the same way the category fibered in grupoids  $\mathfrak{C}oh_{\Xcal/\Sc} \xrightarrow{p} (\Sf/\Sc)_{et}$ over the \'etale site of $\Sc$-superschemes.
Moreover, if $\Xcal\to\Sc$ is projective,  for every superpolynomial $\mathbf{P}$ we can also define the category fibered in grupoids  $\mathfrak{C}oh_{\Xcal/\Sc}(\mathbf{P})$ (resp.\! $\Mf_{\Xcal/\Sc}(\mathbf{P})$) of coherent (super)sheaves (resp.\! locally free supersheaves) on $\Xcal\to \Sc$, flat over $\Sc$,  with relative Hilbert superpolynomial $\mathbf{P}$.

Our first task is to prove that all the above categories fibered in grupoids  are  algebraic stacks. To do so we need at first to prove  the existence of a super analog of Grothedieck's Quot scheme. This is proved in \cite{LTD26} with the same techniques used in \cite{BrHRPo20} to construct the Hilbert superscheme. We begin with the description of the super Quot functor.

\begin{defin} Let $f\colon\Xcal \to \Sc$ be a superscheme morphism and $\F$  a coherent sheaf on $\Xcal$. The relative super Quot functor associated with $\Xcal/\Sc$ and $\F$ is the functor on the category of superschemes over $\Sc$ that  to every superscheme morphism $\Tc\to\Sc$ associates  the family 
$\SQuotf_{\F/\Xcal/\Sc}(\Tc)$ of all   sheaf epimorphisms
$$
\F_\Tc \to \M \to 0,
$$
where $\M$ is a coherent sheaf on $\Xcal_\Tc$, flat over $\Tc$,   whose support is proper over $\Tc$.
\end{defin}
When $f\colon \Xcal \to \Sc$ is a projective morphism we can decompose the super Quot functor into subfunctors $\SQuotf_{\F/\Xcal/\Sc}^{\bP}$ parametrizing quotient sheaves flat over $\Sc$, whose restriction to the fibers have  fixed super Hilbert polynomial $\bP=(P_+,P_-)$ \cite[Definition 2.41]{BrHRPo20}.

The following is a particular case of the construction of the  Quot superscheme (actually, the existence of the Quot superscheme can be proved for a more general class than the projective morphisms \cite{LTD26}).

\begin{prop}[Existence of the Quot superscheme]\label{prop:quotrepres} Let $f\colon \Xcal \to \Sc$ be a projective morphism, where $\Sc$ is a noetherian superscheme, and let $\F$ be a coherent sheaf on $\Xcal$.
\begin{enumerate}
\item
For any super Hilbert polynomial $\bP$, the super Quot functor $\SQuotf_{\F/\Xcal/\Sc}^\bP$ is representable by a proper $\Sc$-superscheme $\SQuot^\bP (\F/\Xcal/\Sc) \to \Sc$.
\item
As a consequence, the super Quot functor $\SQuotf_{\F/\Xcal/\Sc}$  is representable by the disjoint union $\SQuot(\F/\Xcal/\Sc)$ of the $\Sc$-superschemes $\SQuot^\bP (\F/\Xcal/\Sc)$ corresponding to the various super Hilbert polynomials. Then, $\SQuotf_{\F/\Xcal/\Sc}$ is locally of finite type and separated over $\Sc$.
\end{enumerate}
\end{prop}

We now have:

\begin{prop}\cite{LTD26} Let $\Xcal\to\Sc$ be a projective superscheme over a noetherian superscheme $\Sc$. Then, $\mathfrak{C}oh_{\Xcal/\Sc}$, $\Mf_{\Xcal/\Sc}$ and $\Mf_{\Xcal/\Sc}(r|s)$ are algebraic superstacks over $\Sc$. For every superpolynomial $\mathbf{P}$ the categories fibered in grupoids $\mathfrak{C}oh_{\Xcal/\Sc}(\mathbf{P})$ and $\Mf_{\Xcal/\Sc}(\mathbf{P})$ are algebraic superstacks.
\end{prop}
\begin{proof} There are open and closed immersions  $\mathfrak{M}(r|s)_{\Xcal/\Sc}\hookrightarrow \mathfrak{M}_{\Xcal/\Sc}$ and an open immersion $\mathfrak{M}_{\Xcal/\Sc}\hookrightarrow \mathfrak{C}oh_{\Xcal/\Sc}$ of categories fibered in grupoids over $(\Sf/\Sc)_{et}$. Moreover, one has  open and closed immersions $\mathfrak{C}oh_{\Xcal/\Sc}(\mathbf{P})\hookrightarrow \mathfrak{C}oh_{\Xcal/\Sc}$ and $\mathfrak{M}_{\Xcal/\Sc}(\mathbf{P})\hookrightarrow \mathfrak{M}_{\Xcal/\Sc}$.
Then, it is enough to prove the statement for $\mathfrak{C}oh_{\Xcal/\Sc}$. We follow the proof of \cite[Thm.\! 4.6.2.1]{LauMor-Bai00}. 

By \cite[Cor.\! 3.18]{BrHRPo20}, given a $\Sc$-superscheme $\Tc$ and two coherent sheaves $\M$ and $\Nc$ on $\Xcal_\Tc$, flat over $\Tc$, the functor of even homomorpisms $\mathbf{Hom}(\M,\Nc)_+$ is representable by a $\Tc$-linear superscheme $$\phi\colon \mathbb{H}{om}(\M,\Nc)=\mathbb{V}(\Qc_{\M,\Nc})=\SSpec(\SSym(\Qc_{\M,\Nc}))\to \Xcal_\Tc$$  associated with a coherent sheaf $\Qc_{\M,\Nc}$ on $\Xcal_\Tc$. Let $\Psi\colon \phi^\ast\M \to \phi^\ast \Nc$ be the universal morphism. The formation of $\coker \Psi$ is compatible with base changes $\Tc'\to \Tc$. Since $\M$ is flat over $\Tc$, the formation of $\ker\Psi$ is also compatible with such   base changes. Then, the  intersection of  open sub-superchemes of $\mathbb{H}{om}(\M,\Nc)$ given by  complements of the supports of $\coker \Psi$ and $\ker\Psi$ represents the functor  of isomorphisms between $\M$ and $\Nc$.
In  particular, the isomorphisms between two coherent sheaves on $\Xcal_\Tc\to \Tc$ form a sheaf for the Zariski topology, and then also for the \'etale topology. Moreover by  Proposition \ref{prop:fpqcdescent} descent data for $\mathfrak{C}oh_{\Xcal/\Sc}$ are effective. This proves that $\mathfrak{C}oh_{\Xcal/\Sc}$ is a stack. It remains only to construct a smooth presentation. One proceeds as in \cite[Thm.\, 4.6.2.1]{LauMor-Bai00} changing the results there  cited  from EGA with the corresponding statements for projective superschemes contained in \cite{BrHRPo20}, and using the Quot superscheme of Proposition \ref{prop:quotrepres}.
\end{proof}

 The superstack $\mathfrak{C}oh_{\Xcal/\Sc}^{n_0-reg}(\mathbf{P})$ of $n_0$-regular coherent sheaves, where regularity is Calstelnuovo regularity  \cite{BrHRPo20}, can be also described as a sub-superstack of a superstack of representations of a certain superalgebra,   generalizing to the super setting the construction of the moduli of coherent sheaves in terms of Kronecker modules developed in \cite{ACK07}. Using this, one can describe
the structure of the bosonic reduction of the superstack  $\mathfrak{C}oh_{\Xcal/\Sc}(\mathbf{P})$; its relationship with a moduli of Kronecker sheaves and various notions of stability and semistability for coherent sheaves on a projective superscheme are dealt with in  \cite{LTD26}.

\subsection{Ramond--Ramond (smooth) SUSY curves}\label{subsec:RR-SUSY}

The superstack of SUSY curves was the first instance of a superstack appearing in the literature.  The construction of an algebraic superspace describing the structure of a (coarse) moduli space of SUSY curves (without punctures) dates back to \cite{LeRoth88}, where superorbifolds  were used, and   \cite{DoHeSa97}, where it was assumed that the genus was greater than or equal to 2 and   level structures were fixed to prevent the existence of automorphisms different from the identity or the parity swap.    
More recently, it was proved in \cite{BrHR21} that the functor of SUSY curves of genus $g\ge 2$ with a level structure and Ramond and Neveu-Schwarz punctures is representable by an algebraic superspace, and the dimension of this moduli space was computed.
If one does not consider the level structure, one can construct a   moduli superstack,  which is only a Deligne--Mumford superstack \cite{CodViv17}, 
and is an open sub-superstack of the moduli superstack of stable SUSY curves constructed by Felder, Kazhdan and Polishchuk
\cite{FKPpubl}; we shall review this latter construction in the next Subsection.  Here we are going to recall the construction of the moduli space of Ramond--Ramond smooth SUSY curves.

We start by recalling the definitions of {\em supercurve} and {\em SUSY curve}. To simplify the exposition, and following  \cite{BrHR21}, we assume that the genus $g$ is not less than two, and that a level $n$ structure, with $n\ge 3$, has been fixed, so that there exists a fine moduli scheme $M_g$ carrying a universal curve $X_g$.

\begin{defin} \label{def:supercurverel}A  smooth supercurve over a superscheme $\Sc$ is a smooth proper morphism $\Xcal\to \Sc$ of superschemes of relative dimension $1\vert 1$.
\end{defin}

\begin{defin}\cite[Def.~3.4]{BrHR21} \label{def:RRrel}
Let $\pi\colon\Xcal\to \Sc$ be a relative smooth supercurve, and $\Zc\hookrightarrow \Xcal$ a positive relative superdivisor. We say that $\pi\colon\Xcal\to \Sc$ is a \emph{smooth Ramond--Ramond SUSY curve (or a smooth Ramond--Ramond SUSY curve) along $\Zc$}  if   there is a locally free submodule $\Dcal\hookrightarrow \Theta_{\Xcal/\Sc}$ of rank $(0,1)$ such that 
the composition map
$$\Dcal\otimes_{\Oc_{\Sc}}\Dcal \xrightarrow{[\,,\,]}  \Theta_{\Xcal/\Sc}\to  \Theta_{\Xcal/\Sc}/\Dcal$$
induces an  $\Oc_{\Xcal}$-module isomorphism
$$
\Dcal\otimes_{\Oc_{\Xcal}}\Dcal\iso  (\Theta_{\Xcal/\Sc}/\Dcal)(-\Zc)\,.
$$
We then say that $\Dcal$ is a
 \emph{Ramond--Ramond conformal structure} for $(\pi\colon\Xcal\to\Sc,\Zc)$ or a  \emph{Ramond--Ramond conformal structure for $\pi\colon\Xcal\to\Sc$ along $\Zc$}.
\end{defin}

 We assume that the Ramond--Ramond superdivisor $\Zc$ is a collection of mutually disjoint Cartier divisors $\Zc_j$ of relative degree 1 ($j=1,\dots,\nf_R$), which we call the   \emph{Ramond--Ramond punctures}. We refer to $\nf_R$ as the number of Ramond--Ramond punctures. Note that the existence of a smooth Ramond--Ramond SUSY structure along $\Zc$ forces $\nf_R$ to be even.

\subsubsection*{Bosonic reduction of the moduli of Ramond--Ramond SUSY curves}

 We know that the bosonic reduction of the moduli of Ramond--Ramond SUSY curves has to be the moduli of the Ramond--Ramond SUSY curves over ordinary schemes (bosonic superschemes). It is known (see for instance \cite[Prop.\! 3.8]{BrHR21}) that  the existence of a Ramond--Ramond conformal structure along $\Zc$ on a smooth supercurve $\Xcal\to S$, where $S$ is a scheme,  is equivalent to the existence of an  $\Oc_X$-module isomorphism
$$
\Lcl\otimes \Lcl \iso \kappa_{X/S}(Z)\,.
$$
where $Z$ is the bosonic reduction of the Ramond--Ramond superdivisor $\Zc$; or, 
in other words, it is equivalent to the structure of a $Z$-spin structure on the bosonic reduction $X\to S$. 
Then, the restriction to schemes of the \'etale sheaf $\F_{\nf_R}$ of SUSY curves is the sheaf $\F^{spin}_{\nf_R}$ of Ramond--Ramond spin curves. A moduli scheme for  singular spin curves was first constructed by Cornalba  \cite{Corna87}.  The related construction 
of a moduli space for spin curves with Ramond--Ramond punctures 
was given in \cite{BrHR21} adapting the one in  \cite{DoHeSa97} where punctures were not considered. We sketch the construction here.

 Let $X_g\to M_g$ be the universal curve over the moduli scheme of genus $g$ curves, and let $M$ be the open subscheme of the relative symmetric pover $S^{\nf_R}(X_g/M_g)$ which is complementary to all diagonals. Let $X$ be the restriction of $X_g$ to $M$, and denote by $J^d\to M$ the degree $d$ relative Jacobian of $X\to M$; also, denote by $\Upsilon_{[d]} \in \operatorname{Pic}^d(X\times_MJ^d/J^d)$ the ``universal invertible sheaf class.''
 Define the scheme $M_{\nf_R}$ as the fiber product 
 $$\xymatrix{
 M_{\nf_R} \ar[r]\ar[d] & J^{g-1+\nf_R/2} \ar^(.45){\mu_2}[d] \\
 M \ar[r]^(.35){\iota} & {J^{2g-2+\nf_R}} 
 } $$
 where  the morphism $\mu_2$ is the operation of  taking  the square, and $\iota$  is defined by the twisted canonical sheaf $\kappa_{X/M}(Z)$, with  $Z\to M$   the universal positive divisor of degree $\nf_R$. Let also $X_{\nf_R}$ be the fiber product
 $X\times_M J^{g-1+m}$ and  $Z_{nf_R}$ the pullback of $Z$.
The restriction $\Upsilon_{\nf_R}$  to $X_{\nf_R}$ of the class $\Upsilon_{[g-1+\nf_R/2]}$ satisfies $\Upsilon_{[g-1+\nf_R/2]}^2\simeq [\kappa_{\X_{\nf_r}/M_{\nf_R}(Z_{\nf_R})]}$ and one has:
\begin{prop}\cite[Thm.\! 4.5]{BrHR21} The bosonic reduction to schemes of the sheaf  $\F_{\nf_R}$ of SUSY curves and, equivalently, the functor $\F^{spin}_{\nf_R}$ of Ramond--Ramond spin curves, are representable by the quasi-projective scheme $M_{\nf_R}$, which has   dimension $3g-3+\nf_R$, and by the ``universal Ramond--Ramond spin curve'' 
$(X_{\nf_R},Z_{\nf_R}, \Upsilon_{\nf_R})$.
\end{prop}

\subsubsection*{Global moduli   and universal Ramond--Ramond SUSY curves.}

We review here   a construction of the moduli space of Ramond--Ramond SUSY curves as an algebraic superspace; 
as a consequence of Theorems  4.13 and  4.14  of \cite{BrHR21},  this can be done by adapting   
 the construction of the supermoduli of SUSY curves given in  \cite{DoHeSa97}. 

 \begin{defin} An \'etale covering $p\colon U \to M_{\nf_R}$ is {\em trivializing} if
 $p^\ast\Upsilon_{\nf_R} = [\mathcal L]$ for a line bundle $\mathcal L$ on $X_U$, and there is an isomorphism $\mathcal L\otimes \mathcal L \simeq \kappa_{X_U/U}(Z_U)$.
 \end{defin}
 Trivializing \'etale coverings of that kind always exist.
 On can define
a  local moduli superscheme of Ramond--Ramond 
SUSY curves as 
$$
\U=(U, \bigwedge \Ec_U)\,,
$$
where $U$ is a trivalizing \'etale covering,  while $$\Ec_U =   \pi_{U\ast} (\kappa_{X_U/U}(Z_U)^{1/2}\otimes\kappa_{X_U/U})\simeq \pi_{U\ast} (\Lcl\otimes \kappa_{X_U/U}),$$ and one has
\begin{equation}\label{eq:dimRR}
\dim\,\U=(3g-3+\nf_R,2g-2+\nf_R/2)\,.
\end{equation}
Since the even component  of the Kodaira-Spencer map of the relative Ramond--Ramond SUSY curve $(\Xcal_U,Z_U,\Lcl)$ over $U$ is an isomorphism,
one    obtains a relative Ramond--Ramond SUSY curve
$$
(\Xf_\U,\Zc_U,\Dcal_U)\to \U,
$$
extending $(\Xcal_U,Z_U,\Lcl)\to U$  
whose Kodaira-Spencer map is an isomorphism \aoe{\cite[Thm.\! 4.13]{BrHR21}}.
If we denote by $F_{\nf_R}$ the restriction of the presheaf of Ramond--Ramond SUSY curves with $\nf_R$ punctures
to the category of schemes, there is a natural presheaf morphism  $U\to F_{\nf_R}$ which maps a scheme morphism $\varphi\colon S\to U$ to the Ramond--Ramond SUSY curve $\varphi^\ast \Xcal_{U}\to S$.
We can then consider the fiber product $\F_{\nf_R}\times_{F_{\nf_R}} U$ as a presheaf of 
superschemes. A section of this presheaf on a supercheme $\Sc$
is a  pair formed by a Ramond--Ramond SUSY curve
$\Xf\to\Sc$ and a scheme morphism  $\varphi\colon S\to U$,
such that $\Xf\to\Sc$ is an extension of $\varphi^\ast \Xcal_{U}\to S$.
By Theorem 4.5 of  \cite{BrHR21}, the sheaf associated with the presheaf $\F_{\nf_R}\times_{F_{\nf_R}} U$ is
$$
\F_{\nf_R}\times_{M_{\nf_R}} U\,.
$$ 

 Eventually one proves:

\begin{thm}\label{thm:globlamoduiRR}\cite[Thm.\! 4.16]{BrHR21} The functor $\F_{\nf_R}$ of relative Ramond--Ramond SUSY curves of genus $g$ along a (nonramified) puncture of degree $\nf_R$ is the categorial quotient
$$
\xymatrix{
\U\times_{M_{\nf_R}}U\ar@<3pt>[r]^(.65){(\rho_1,\rho_2)}\ar@<-3pt>[r]& \U\ar[r]^\rho& \F_{\nf_R}\,,
 }
$$
 of an \'etale equivalence relation of superschemes in the category of sheaves of superschemes. Thus $\F_{\nf_R}$ is identified with a separated   algebraic superspace  $\M_{\nf_R}$ (see Proposition \ref{prop:quot}) whose bosonic reduction is the scheme $M_{\nf_R}$. Moreover,
$$
\dim \M_{\nf_R}=(3g-3+\nf_R,2g-2+\nf_R/2)\,.
$$
\qed
\end{thm}

\subsection{Stable SUSY curves}\label{subsec:Stcurves}

In the same way as the moduli stack of smooth curves is compactified by adding stable curves \cite{DelMum69}, the moduli superstack of smooth SUSY curves, with or without punctures, is compactified by adding stable singular SUSY curves. This raises two problems: first, characterizing what kinds of singularities are allowed in the bosonic reduction of the curve; and second, defining a sensible notion of a SUSY structure on these singular curves. The answer to the first question is straightforward: we should consider stable curves, that is, prestable (meaning nodal) curves with a finite number of automorphisms. To answer the second question one needs an alternative definition of a SUSY structure  that   also  works for singular prestable curves, as in \cite{BrHRMa23} (see also \cite{FKPpubl}). 
One sees that if $f\colon\Xcal \to \Sc$ is a proper smooth supercurve and $\Zc\hookrightarrow\Xcal $ is a positive relative  superdivisor as in Subsection \ref{subsec:RR-SUSY}, the existence  of a superconformal structure $\Dcal\hookrightarrow \Theta_f=\Omega_f^\ast$ with Ramond--Ramond punctures at $\Zc$ is   equivalent to the existence of a sheaf epimorphism 
$$
\Omega_f\xrightarrow{\bar\delta} \Ber_f(\Zc)\to 0
$$
such that the composition
$$
\ker\bar\delta \hookrightarrow \Omega_f\xrightarrow{d} \Omega_f\wedge \Omega_f\xrightarrow{\bar\delta\wedge\bar\delta} \Ber_f^{\otimes 2}(2\Zc)
$$
yields an isomorphism
$$
\ker\bar\delta \iso \Ber_f^{\otimes 2}(\Zc)\,.
$$
The Ramond--Ramond superconformal  distribution $\Dcal$ is recovered as the image of the dual morphism $\bar\delta^\ast\colon \Ber_f^\ast(-\Zc)\hookrightarrow \Theta_f$. 

Since by a classical result \  the relative Berezinian $\Ber_f$ is the dualizing sheaf $\omega_f$ for $f$ \cite{OgPe84, Penk83} \cite[Prop. 3.35]{BrHRPo20}, we can generalize the notion of SUSY-curve with Ramond--Ramond punctures to the singular case. To do that we   need to fix the kind of nonnecessarily smooth supercurves $f\colon\Xcal\to\Sc$ we are going to deal with, the meaning of the Ramond--Ramond punctures, and finally the notion of prestability and stability for such objects.

To begin with, let us consider   a proper  superscheme morphism $f\colon \Xcal\to \Sc$    that is a Cohen-Macaulay relative supercurve \cite{BrHRMa23}. By this we mean:
\begin{enumerate}
\item $f$ is flat and its fibers $\Xcal_s$ are connected and reduced superschemes of pure dimension $1|1$, that is, one has  $\Oc_{\Xcal_s}=\Oc_{X_s}\oplus \Pi\Lcl_s$, where $\Lcl_s$ is a pure sheaf of dimension 1 and rank 1;
\item the dualizing complex of $f$ is of the form $\omega_f[1]$ for a sheaf $\omega_f$ on $X$ flat over $S$.
\end{enumerate}

To define a  Ramond--Ramond structure we take a positive relative superdivisor $\Zc \to \U$ taking values in the smooth locus $\U$ of $f\colon \Xcal \to \Sc$. Moreover, we impose that the various components $\Zc_i$ of $\Zc$ do not intersect each other and that the restriction $f_{| \Zc_i}\colon\Zc_i \to \Sc$ is an isomorphism for every $i$. In other words, $\Zc \to \Sc$ is not ramified. Again, the components $\Zc_i$ are called the Ramond--Ramond punctures.

\begin{defin}\label{def:SUSYCM}\cite[Def.\! 5.8]{BrHRMa23} Let $f\colon\Xcal \to \Sc$ be a proper Cohen-Macaulay relative supercurve and $\Zc\hookrightarrow\U $  a relative positive superdivisor as above.  A superconformal structure with poles along $\Zc$ is an epimorphism
$$
\Omega_f\xrightarrow{\bar\delta} \omega_f(\Zc)\to 0\,,
$$
where $\omega_f$ is the relative dualizing sheaf, such that the composition
$$
\ker\bar\delta \hookrightarrow \Omega_f\xrightarrow{d} \Omega_f\wedge \Omega_f\xrightarrow{\bar\delta\wedge\bar\delta} \omega_f^{\otimes 2}(2\Zc)
$$
yields an isomorphism
$$
\ker\bar\delta \iso \omega_f^{\otimes 2}(\Zc)\,.
$$
We  say that  $(f\colon\Xcal \to \Sc,\Zc,\bar\delta)$ is a Ramond--Ramond SUSY curve   along $\Zc$. 
\end{defin}

Prestability and stability for  Ramond--Ramond SUSY curves are defined in terms of the bosonic reduction. The precise definition is: 

\begin{defin}\label{def:stablecurve}\cite[Def.\! 5.11]{BrHRMa23}. Fix nonnegative integers $g$, $\nf_{NS}$ and $\nf_{RR}$, and write $\nf:=(\nf_{NS}, \nf_{RR})$. 
 A stable (resp.\! prestable) SUSY curve of arithmetic genus $g$ with $\nf$ punctures over a superscheme $\Sc$ is a proper and Cohen-Macaulay $\Sc$-supercurve $f\colon \Xcal \to \Sc$, together with:
\begin{enumerate}
\item a collection of mutually disjoint closed sub-superschemes $\Wc_i\hookrightarrow \U$ ($i=1,\dots,\nf_{NS}$), where $\U$ is the smooth locus of $f$, such that $\pi\colon \Wc_i\to\Sc$ is an isomorphism for every $i$. They are called   \emph{Neveu-Schwarz  punctures};
\item a collection of mutually disjoint Cartier divisors $\Zc_j$ of relative degree 1 ($j=1,\dots,\nf_R$), contained in $\U$ and disjoint from the Neveu-Schwarz  punctures. They are the   \emph{Ramond--Ramond punctures} as defined earlier;  
\item a morphism $\bar\delta\colon \Omega_f\to \omega_f(\Zc)$, where $\Zc=\sum \Zc_j$.
\end{enumerate}
Moreover, these data fulfill the following conditions: 
\begin{enumerate}
\item \emph{(Stability (resp.\! prestability))} $f\colon \Xcal \to \Sc$ is a  stable (resp.\! prestable) relative supercurve with punctures $(\Wc_i,\Zc_j)$, i.e., the bosonic restriction $f_{bos}\colon X \to S$ is a  stable (resp.\! prestable) relative curve with punctures $(W_i,Z_j)$; 
\item \emph{(Superconformal structure)} the morphism $\bar\delta$ is a superconformal structure with poles along $\Zc$ as in Definition \ref{def:SUSYCM}.
\end{enumerate}
\end{defin}

Morphisms of prestable Ramond--Ramond SUSY curves are defined straightforwardly \cite[Def.\! 2.1]{BHRSM25}.

\begin{defin}\label{def:category fibered in grupoids smaps}  The category fibered in groupoids 
$$
\Sf\Mf_{g,\nf} \xrightarrow{p} \Sf
$$
of stable SUSY curves of arithmetic genus $g$ with $\nf=(\nf_{NS}, \nf_{RR})$ punctures is given by the following data: 
\begin{itemize}\item objects are stable Ramond--Ramond SUSY curves $\mathbb X\equiv\{f\colon\Xcal\to\Sc, \{\Wc_i\},\{\Zc_j\}, \bar\delta)$\footnote{We use a simplified notation here; some of the relevant  data are left implicit.} over a superscheme $\Sc$, of arithmetic genus $g$, with $\nf_{NS}$ Neveu-Schwarz punctures and $\nf_{RR}$ Ramond--Ramond punctures; 
\item morphisms are cartesian diagrams
\begin{equation}\label{eq:cartesian}
\xymatrix{
\mathbb X'\ar[r]^\Xi\ar[d]^g & \mathbb X \ar[d]^f\\
\Tc\ar[r]^\xi & \Sc
}
\end{equation}
that is, diagrams inducing an isomorphism $\Xcal'\iso \xi^\ast \Xcal:=\Xcal\times_{\Sc}\Tc$ of superschemes over $\Tc$ compatible with the punctures and the morphisms $\bar\delta$. 
The functor $p$  maps a stable Ramond--Ramond SUSY curve to the base superscheme  and the morphism $\Xi$ to the base morphism $\xi$. The pullback $\xi^\ast\mathbb X\to\Tc$ is given by the fiber product $f_\Tc\colon\Xcal\times_{\Sc}\Tc\to\Tc$ and the fiber products of the data $\{\Xcal_i\},\{\Zc_j\}, \bar\delta$ (thus providing  a natural ``cleavage''). Note that the fiber product of $\bar\delta$ is a derivation of the same kind due to \cite[Prop.\! 2.26]{BrHRMa23}.\end{itemize}
\end{defin}

The following result is proved in \cite{FKPpubl}:

\begin{prop}\label{prop:SUSYsuperstack} The category fibered in grupoids  $\Sf\Mf_{g,\nf} \xrightarrow{p} \Sf$ of stable SUSY curves with punctures is a proper and smooth Deligne--Mumford superstack.
\end{prop}

 We give an outline of  the proof  in several steps; it will be completed at the end of this subsection.  

\paragraph{\bf $\Sf\Mf_{g,\nf} \xrightarrow{p} \Sf$ is a superstack.}
We use  some facts proved in \cite{FKPpubl}. 
Let us first recall that, given a prestable SUSY curve  $f\colon \Xcal\to\Sc$, a Neveu-Schwarz  puncture $\Xcal_i\hookrightarrow \Xcal$  corresponds to a Cartier positive superdivisor $\Wc_i\hookrightarrow \Xcal$ with the same support. 
The first fact  is that  the sheaf
$$
\omega_{\Xcal/\Sc}^2:= i_\ast i^\ast \omega_{\Xcal/\Sc}^{\otimes 2}=i_\ast (\omega_{\Ucal/\Sc}^{\otimes 2})\,,
$$
where $\omega_{\Xcal/\Sc}=\omega_f$ is the relative dualizing sheaf, is locally free, i.e., it  is a {line bundle} \cite[Thm.\! 4.1]{FKPpubl}.
Finally, the second fact is an alternative definition of stability for SUSY curves, described in the next statement.

\begin{prop}\label{prop:stablecurve} A prestable SUSY curve with punctures $\mathbb X\equiv\{f\colon\Xcal\to\Sc, \{\Wc_i\},\{\Zc_j\}, \bar\delta)$ is stable if and only if the  line bundle
$$
\omega_{\Xcal/\Sc}^2(\Wc+\Zc)\,,
$$
where $\Wc=\sum_i\Wc_i$ and $\Zc=\sum_j\Zc_j$, is relatively strongly ample with respect to $f\colon\Xcal\to\Sc$ (\cite[Subsect.\! 2.5]{BrHRPo20} or \cite[Def.\! 4.2]{FKPpubl}). 
\end{prop}

A consequence is that if $\mathbb X\equiv\{f\colon\Xcal\to\Sc, \{\Wc_i\},\{\Zc_j\}, \bar\delta)\}$ is  a stable SUSY curve with punctures,  $f\colon\Xcal\to\Sc$ is flat and strongly superprojective. Then, if $\mathbb X'\equiv\{f'\colon\Xcal'\to\Sc, \{\Wc'_i\},\{\Zc'_j\}, \bar\delta')\}$ is another stable SUSY curve with punctures,  the functor of isomorphisms between $f\colon\Xcal\to\Sc$ and $f'\colon\Xcal'\to\Sc$  is representable by an $\Sc$-superscheme $\Hc\to\Sc$ by 
 \cite[Cor.\! 4.27]{BrHRPo20}. It follows easily that there is a closed sub-superscheme $\sIsom(\mathbb X,\mathbb X)\hookrightarrow \Hc$ representing the functor of isomorphisms of $\mathbb X$ with $\mathbb X'$ (see \cite[5.2]{FKPpubl}). In particular, the functor of isomorphisms of $\mathbb X$ with $\mathbb X'$ is a sheaf.

Moreover, using descent for closed sub-superschemes (Proposition \ref{prop:subdescent}) and for polarized superschemes (Proposition \ref{prop:descpolar}), one proves that descent data for stable SUSY curves with punctures are effective. Thus, $\Sf\Mf_{g,\nf}$ is a \emph{superstack}.

\paragraph{\bf The bosonic reduction of the superstack $\Sf\Mf_{g,\nf} \xrightarrow{p} \Sf$.}
To understand the bosonic reduction of the superstack $\Sf\Mf_{g,\nf}$ one has to determine the categories $\Sf\Mf_{g,\nf}(S)$ for any \emph{ordinary} scheme $S$. The objects of such categories are the stable SUSY curves $\Xf\equiv\{f\colon\Xcal\to S, \{\Xcal_i\},\{\Zc_j\}, \bar\delta)\}$ with arithmetic genus $g$ and $\nf$ punctures over $S$.
The situation here is simpler than in the general case; since the base $S$ is an ordinary scheme, the structure sheaf of $\Xcal$ decomposes as $\Oc_\Xcal=\Oc_X\oplus \Pi \Lcl$ were $\Lcl$ is fiberwise an even sheaf pure of dimension 1, and the relative dualizing sheaf takes the form $\omega_f=\omega_{X/S}\oplus \Pi \,\Homsh_{\Oc_X}(\Lcl, \omega_{X/S})$.
The datum of a superconformal structure on $f\colon \Xcal \to S$ is equivalent to the datum of an isomorphism
$$
\Lcl\iso \Homsh_{\Oc_X}(\Lcl,\omega_{X/S}(Z))\,,
$$
of $\Oc_X$-modules, where $Z=\sum Z_j$ (\cite[Lemma 2.6]{FKPpubl}, \cite[Sec.\! 5.3]{BrHRMa23}).

Then, giving a stable SUSY curve $\mathbb X\equiv\{f\colon\Xcal\to S, \{\Wc_i\},\{\Zc_j\}, \bar\delta)\}$ over $S$ is the same as giving a \emph{spin curve} with punctures, according to the following well-known  definition.
\begin{defin}\label{def:spincurve} A stable (resp.\! prestable) spin curve with arithmetic genus $g$ and $\nf=\{\nf_{NS},\nf_{RR}\}$ punctures consists of:
\begin{enumerate}
\item a stable (resp.\! prestable) relative curve $(\tilde f\colon X \to S, \{W_i\}, \{Z_j\})$ of arithmetic genus $g$;
\item \emph{(Spin structure)} a sheaf $\Lcl$ on $X$ relatively pure of dimension 1 together with an isomorphism $\Lcl\iso \Homsh_{\Oc_X}(\Lcl,\omega_{X/S}(Z))$ of $\Oc_X$-modules, where $Z=\sum Z_j$. Note that $\tilde f$ is  Cohen-Macaulay (actually Gorenstein) so that it has a relative dualizing sheaf $\omega_{X/S}$. 
\end{enumerate} 
\end{defin} 

One can now straightforwardly define the \emph{stack} $\Mf^{spin}_{g,\nf}$ of   spin curves of arithmetic genus $g$ and $\nf=\{\nf_{NS},\nf_{RR}\}$ punctures. It was first studied  in \cite{Corna87} and more recently in \cite{JaKiVa05,MoZh24}.

From the above discussion we deduce the following result: 
\begin{prop} There is a stack isomorphism
$$
(\Sf\Mf_{g,\nf})_{bos} \simeq \Mf^{spin}_{g,\nf}
$$
between the bosonic reduction stack of the superstack of stable punctured Ramond--Ramond  SUSY curves and the stack of the punctured spin curves.
\qed
\end{prop}	 

Thus, using  \cite{Corna87} or \cite[Thm.\! 2.3.2]{JaKiVa05} we have
\begin{prop}\label{prop:spinstack}
The bosonic reduction $(\Sf\Mf_{g,\nf})_{bos}$ of the superstack of stable punctured Ramond--Ramond SUSY curves is a proper and smooth Deligne--Mumford stack.
\qed
\end{prop}

\subsubsection*{The superstack $\Sf\Mf_{g,\nf} \xrightarrow{p} \Sf$ is Deligne--Mumford, proper and smooth.} \label{subsubsec:DMsmoothproper}
This property does not follow from Propostion \ref{prop:spinstack}, because a  superstack (or even superscheme) with smooth bosonic reduction may fail to be smooth. 

To prove that the superstack $\Sf\Mf_{g,\nf}$ is Deligne--Mumford we have to construct an \'etale presentation. This is done in \cite[Sec.\! 5.3]{FKPpubl}; we sketch here the proof.

Since the bosonic reduction $(\Sf\Mf_{g,\nf})_{bos} \simeq \Mf^{spin}_{g,\nf}$ of $\Sf\Mf_{g,\nf}$ is a Deligne--Mumford stack (Proposition \ref{prop:spinstack}), for every stable SUSY curve with punctures (or punctured spin curve) $(X_0,\{w_i\},\{z_j\}, \Lcl_0)$ over the base field $k$, we can find an affine scheme $S$, a point $s_0\in S$ and  a relative punctured spin curve $(\tilde f\colon X\to S, \{W_i\},\{Z_j\}, \Lcl)$ whose fiber over $s_0$ is $(X_0,\{w_i\},\{z_j\}, \Lcl_0)$ and such that the corresponding morphism $S \to \Mf^{spin}_{g,\nf}$ is \'etale. 

Now, we recall an easy property of locally split superschemes:  the structure sheaf is locally the exterior algebra of a  locally free sheaf, and the dual $\Ec$ of this sheaf is precisely the odd part of the bosonic reduction of the (super) tangent sheaf. When the superscheme is a moduli of a geometric structure, the bosonic reduction of its tangent sheaf is the first direct image of the infinitesimal deformations of the universal geometric structure, so that $\Ec$ is the first direct image of the odd infinitesimal deformations.

Following that idea, we consider the  $\Z_2$-graded sheaf \footnote{We follow here the notation of \cite{FKPpubl}, instead of the one in \cite{BrHR21}, where the sheaf is denoted   $\mathcal{G}_{\tilde f}$ as in \cite{LeRoth88}.} 
$\mathcal{A}$ of infinitesimal deformations of the punctured spin curve $X_{sp}:=(\tilde f\colon X\to S, \{W_i\},\{Z_j\}, \Lcl)$, that is, the sheaf of automorphisms of the trivial family induced by $X_{sp}$ over $k[\epsilon,\tau]$, where $\epsilon$ is even, $\tau$ is odd, and $\epsilon^2=\tau^2=\epsilon\tau=0$.
The sheaf $\mathcal{A}_-$ of odd infinitesimal deformations is computed in \cite[Prop.\! 3.11]{FKPpubl} using a partial normalization; the computation in the smooth case is described in Subsection \ref{subsec:RR-SUSY} of these notes. Moreover,  in \cite[Sec.\! 5.3]{FKPpubl} it is proved that
$$
\Ec:= R^1 \tilde f_\ast(\mathcal{A}_- )= R^1\tilde f_\ast ( \omega_{X/S}^{-1}\otimes \Lcl (- \sum_i W_i - \sum_j Z_j))\,.
$$
Note that $\tilde f_\ast(\mathcal{A}_-)=0$ by \cite[Prop.\! 3.11]{FKPpubl}, so that $\Ec$ is locally free. Now, we define the base superscheme $\Sc$ as
\begin{equation}\label{eq:superbase}
\Sc = (S, \bigwedge_{\Oc_S} \Pi \Ec^\vee)\,.
\end{equation}
In \cite{FKPpubl} it was  shown that, possibly after changing $S$ with an \'etale neighborhood of $s_0$, there is a stable punctured SUSY curve $\mathbb X:=(\Xcal \to \Sc, \{\Wc_i\}, \{\Zc_j\}, \bar\delta)$ extending $(\tilde f\colon X\to S, \{W_i\},\{Z_j\}, \Lcl)$ such that the natural superstack morphism $\Sc \to \Sf\Mf_{g,\nf}$ induced by $\mathbb X$ is \'etale. In the smooth case, this is a consequence of the fact that we can construct $\mathbb X$ so that the corresponding Kodaira-Spencer map is an isomorphism. This is proved in  \cite[Thms.\! 4.13, 4.14]{BrHR21}, adapting \cite{LeRoth88} to the presence of Ramond--Ramond punctures. For punctured stable SUSY curves, the proof in \cite[Sec.\! 5.3]{FKPpubl}  uses the smooth case, plus the partial normalization procedure.

Thus $\Sc \to \Sf\Mf_{g,\nf}$ is an \'etale presentation of $\Sf\Mf_{g,\nf}$, so that $\Sf\Mf_{g,\nf}$ is a Deligne--Mumford superstack as claimed.

Moreover one has:
\begin{itemize}
\item as $(\Sf\Mf_{g,\nf})_{bos} \simeq \Mf^{spin}_{g,\nf}$ is a smooth Deligne--Mumford stack by Proposition \ref{prop:spinstack}, the scheme $S$ is smooth. Then $\Sc$ is a smooth superscheme by Equation \eqref{eq:superbase}, so that $\Sf\Mf_{g,\nf}$ \emph{is a smooth superstack} (Definition \ref{def:morphismsprops3});
\item since $(\Sf\Mf_{g,\nf})_{bos}$ is a proper stack again by Proposition \ref{prop:spinstack}, the superstack $\Sf\Mf_{g,\nf}$ \emph{is proper} by Proposition \ref{prop:separpropbos}.
\end{itemize}

We have thus finished the proof of Proposition \ref{prop:SUSYsuperstack}.

\subsection{Stable  supermaps}
\label{subsec:Stmaps}
Maps from SUSY curves to a fixed target superscheme satisfying a stability condition were defined in \cite{BrHRMa23} and studied in more depth in \cite{BHRSM25}.

A first issue  is that we need a theory of supercycles with at least some of the usual properties of ordinary cycles, in particular, the existence of   push-forwards. Supercycles were defined and studied in \cite{BrHRMa23}; we give here a sketch of its principal features (see also \cite{BHRSM25}). 

We start by giving $\Z^2$ a superring structure   writing it as $\Z\oplus\Pi\Z$ (where $\Pi$ is the parity change operator), so that
$$(m-\Pi n)(m'-\Pi n') = (mm'+nn' - \Pi (mn'+m'n)).$$
 
 \begin{defin}  Let $\Xcal$ be a  superscheme.
An $h$-supercycle of $\Xcal$ is a finite sum
 $$
\alpha={\textstyle \sum_i} (m_i -  {\Pi }n_i) [Y_i]
$$
 where $m_i-\Pi n_i \in\Z^2$ and the $Y_i$ are closed subvarieties of $X$ of dimension $h$. The set $Z_h(\Xcal)$  of   $h$-supercycles is a free $\Z_2$-graded module over $\Z^2$.
The group of supercycles of $\Xcal$ is the bi-graded $\Z^2$-module
$$
Z_\bullet (\Xcal)=\bigoplus_{h=0}^m Z_h(\Xcal) = Z_\bullet (X) \oplus \Pi Z_\bullet (X)\,.
$$\end{defin}

Once supercycles have been  defined we can introduce the notion of linear equivalence for them.  
\begin{defin}\ 
\begin{enumerate}
\item A supercycle
  $\alpha\in Z_h(\Xcal)$ is rationally equivalent to zero if there are $t$  sub-supervarieties $\delta_i\colon\Wc_i\hookrightarrow\Xcal$  of even dimension $h+1$ and   odd dimension  $s=0$ or $1$   and nonzero rational even superfunctions $g_i\in \K(\Wc_i)^\ast$ such that
$$
\alpha=\sum_{i=0}^t \delta_{i\ast}\operatorname{div}(g_i)\,.
$$
\item  The $\Z^2$-module of $h$-supercycles modulo rational equivalence  is
 $$
A_h(\Xcal)= Z_h(\Xcal)/W_h(\Xcal)
$$
where $W_h(\Xcal)\subset Z_h(\Xcal)$ is the  graded $\Z^2$-submodule  formed by the 
 $h$-supercycles rationally equivalent to zero.
 \end{enumerate}
 \end{defin}
 
 One has \cite[Def.\! 4.17]{BrHRMa23}:
 \begin{prop}
 If $f\colon \Xcal \to \Ycal$ is a proper superscheme morphism, there is a pushforward morphism $$f_\ast\colon A_\bullet(\Xcal) \to A_\bullet(\Ycal) \,.$$
\end{prop}

Now, using the notion of prestable SUSY curve given in Definition \ref{def:stablecurve}, we can define stable maps as follows:

\begin{defin}\label{def:stablemap2} \cite[Def.\! 2.8]{BHRSM25} Let $\Ycal$ be a projective superscheme and $\beta\in A_1(\Ycal$) a rational equivalence class of supercycles \cite{BrHRMa23}. 
A stable supermap $F$ to $\Ycal$  over a superscheme $\Sc$ of class $\beta$ with $\nf_{NS}$ NS punctures and $\nf_{RR}$ RR punctures is given by the following set of data:
\begin{enumerate}
\item a  prestable SUSY curve $(f\colon\Xcal\to\Sc, \{\Xcal_i\},\{\Zc_j\}, \bar\delta)$ over $\Sc$ with $\nf_{NS}$ Neveu-Schwarz punctures and $\nf_{RR}$ Ramond--Ramond punctures (Definition \ref{def:stablecurve});
\item a superscheme morphism $\phi\colon\Xcal\to \Ycal$ such that $\phi_\ast[\Xcal_s]=\beta$ for every closed point $s\in S$;
\item if we denote as above by $\Wc_i$ the superdivisor corresponding to the Neveu-Schwarz puncture $\Xcal_i$ according to \cite[Def.\! 2.13]{FKPpubl}, the even line bundle
$$
\Lcl_F:=\omega_{\Xcal/\Sc}^2\big(\sum_{1\le i\le \nf_{NS}} \Wc_i + \sum_{1\le j\le \nf_{RR}} \Zc_j\big)\otimes \phi^\ast \Lcl_\Ycal^{\otimes 3},
$$ 
where $\Lcl_\Ycal = \Oc_\Ycal(1)$,
is  strongly relatively ample with respect to $f\colon\Xcal\to\Sc$ (\cite[Subsect.\! 2.5]{BrHRPo20} or \cite[Def.\! 4.2]{FKPpubl}).
\end{enumerate}
\end{defin}

When $\Ycal$ reduces to a   point, we recover the definition of stable SUSY curve with punctures. This is due to the charaterization of stability in terms of the relative strong ampleness of a line bundle given in Proposition \ref{prop:stablecurve}.

We now fix nonnegative integers $g$, $\nf_{NS}$ and $\nf_{RR}$ and a 1-supercycle $\beta\in A_1(\Ycal)$.  
As in \cite[Def.\! 6.5]{BrHRMa23} we have:

\begin{defin}\label{def:category fibered in grupoids smaps2}  The category fibered in groupoids 
$$
\Sf\Mf_{g,\nf_{NS},\nf_{RR}}(\Ycal,\beta) \xrightarrow{p} \Sf
$$
of $\beta$-valued stable supermaps of arithmetic genus $g$ with $\nf_{NS}$ Neveu-Schwarz  punctures and $\nf_{RR}$ Ramond--Ramond punctures into $\Ycal$ 
is given by the following data: 
\begin{itemize}\item objects are $\beta$-valued stable supermaps $F=((f\colon\Xcal\to\Sc, \{\Xcal_i\},\{\Zc_j\}, \bar\delta), \phi)$\footnote{We use a simplified notation here; some of the relevant  data are left implicit.} over a superscheme $\Sc$ into $\Ycal$, of arithmetic genus $g$, with $\nf_{NS}$ Neveu-Schwarz punctures and $\nf_{RR}$ Ramond--Ramond punctures;
\item morphisms are cartesian diagrams
\begin{equation}\label{eq:cartesian2}
\xymatrix{
F'\ar[r]^\Xi\ar[d]^g & F \ar[d]^f\\
\Tc\ar[r]^\xi & \Sc
}
\end{equation}
that is, diagrams inducing an isomorphism $\Xcal'\iso \xi^\ast \Xcal:=\Xcal\times_{\Sc}\Tc$ of superschemes over $\Tc$ compatible with the punctures and the morphisms $\bar\delta$ and such that $\phi'=\phi\circ\Xi$. 
The functor $p$  maps a stable supermap to the base superscheme  and the morphism $\Xi$ to the base morphism $\xi$. The pullback $\xi^\ast F\to\Tc$ is given by the fiber product $f_\Tc\colon\Xcal\times_{\Sc}\Tc\to\Tc$ and the fiber products of the data $\{\Xcal_i\},\{\Zc_j\}, \bar\delta$. Note that the fiber product of $\bar\delta$ is a derivation of the same kind, as a consequence of  \cite[Prop.\! 2.26]{BrHRMa23}.
 \end{itemize}
\end{defin}

Condition (3) of Definition \ref{def:stablemap2} implies that the prestable SUSY curve associated with a stable supermap is projective.

The main result of this Subsection is:
\begin{prop}\label{prop:DMstack}  \cite[Prop.\! 2.18]{BHRSM25} The category fibered in grupoids  $\Sf\Mf_{g,\nf_{NS},\nf_{RR}}(\Ycal,\beta) \xrightarrow{p} \Sf$ of stable supermaps is a Deligne--Mumford superstack.
\end{prop}

The proof is given in \cite{BHRSM25}. It is divided in several parts;
the first is the proof that $$\Sf\Mf_{g,\nf_{NS},\nf_{RR}}(\Ycal,\beta) \xrightarrow{p} \Sf$$ is an algebraic superstack with superschematic and separated diagonal. This is the content of \cite[Thm.\! 2.13]{BHRSM25}. 
 
The second is the proof that  a stable supermap $F=(\Xcal,\Wc,\Zc,\bar\delta,\phi)$ over an algebraically closed field has no global infinitesimal deformations \cite[Lem.\! 2.17]{BHRSM25}. 

The proof of Proposition \ref{prop:DMstack} is then completed as follows.
Let us write for simplicity $\Sf\Mf$ instead of $\Sf\Mf_{g,\nf_{NS},\nf_{RR}}(\Ycal,\beta)$. Since  $\Sf\Mf$  is an algebraic superstack with superschematic diagonal, for every geometric point of $\Sf\Mf$, which is given by a stable supermap $F=(\Xcal,\Wc,\Zc,\bar\delta,\phi)$ over an algebraically closed field $k$, the algebraic superspace $\mathfrak{A}ut_{\Sf\Mf(k)}(F)$ of the automorphisms of $F$ is a superscheme.  The tangent space to  $\mathfrak{A}ut_{\Sf\Mf(k)}(F)$ at the identity  is the space of   global infinitesimal automorphisms of $F$; since this superspace is zero, $\mathfrak{A}ut_{\Sf\Mf(k)}(F)$ is a discrete and reduced group scheme, so that $\Sf\Mf$ is a Deligne--Mumford superstack by the characterization of Deligne--Mumford superstacks by properties of the diagonal (Proposition \ref{prop:unramdiag1}). 

\subsubsection*{The bosonic reduction of the superstack of superstable maps}

A very interesting feature of the supermoduli of stable supermaps is that it is in general nonproper. Since the properness of an algebraic superstack depends only on its bosonic reduction  (Proposition \ref{prop:separpropbos}), to see whether $\Sf\Mf_{g,\nf_{NS},\nf_{RR}}(\Ycal,\beta)$ is proper we have to look at the bosonic reduction $\Sf\Mf_{g,\nf_{NS},\nf_{RR}}(\Ycal,\beta)_{bos}$. 
It happens that this admits an affine morphism $\Ff$ to a proper stack, the stack of stable spin maps we are going to describe afterwords. We will see that the later is proper, so $\Sf\Mf_{g,\nf_{NS},\nf_{RR}}(\Ycal,\beta)_{bos}$ is proper only if the stack morphism $\Ff$ is finite.

The stack of stable spin maps, which we denote   $\Mf^{spin}_{g,\nf_{NS},\nf_{RR}}(Y,\beta_0)$, is readily constructed  after giving a suitable notion of stable spin curve with punctures.
Let $Y$ be a projective scheme, and   fix an ample line bundle $\Lcl_Y$ and a 1-cycle $\beta_0$ on $Y$.
 
\begin{defin}\label{def:spinmap} 
A stable spin map $F$ into $Y$ over a scheme $S$, of genus $g$ and  class $\beta_0$, with $\nf_{NS}$ Neveu-Schwarz punctures and $\nf_{RR}$  Ramond--Ramond punctures, is given by the following data:
\begin{enumerate}
\item a prestable spin curve of genus $g$ over $S$  with $\nf_{NS}$ Neveu-Schwarz punctures and $\nf_{RR}$  Ramond--Ramond punctures $(f\colon X\to S,\{W_i\}, \{Z_j\},  \Lcl,\varpi)$ (Definition \ref{def:spincurve});
\item a scheme morphism $\psi\colon X\to Y$ such that $\psi_\ast[X_s]=\beta_0$ for every closed point $s\in S$.
\end{enumerate}
Moreover, one assumes that the line bundle
$$
\Lcl_F:=\omega_{X/S}(W+Z)\otimes \psi^\ast \Lcl_Y^{\otimes 3}\,,
$$ 
where $W=\sum_i W_i$, is strongly relatively ample with respect to $f\colon X\to S$.
\end{defin}

The stack of stable spin maps was studied in \cite{JaKiVa05,MoZh24}. One has:
\begin{prop}\label{prop:spinstack2}\cite[Thm.\! 2.3.2]{JaKiVa05}
$\Mf^{spin}_{g,\nf_{NS},\nf_{RR}}(Y,\beta_0)$ is a proper Deligne--Mumford stack.
\end{prop}

Now, we can compare the bosonic reduction   $\Sf\Mf_{g,\nf_{NS},\nf_{RR}}(\Ycal,\beta)_{bos}$ with  $\Mf^{spin}_{g,\nf_{NS},\nf_{RR}}(Y,\beta_0)$. First note that  if  $(\Xcal, \{x_i\},\{\Zc_j\},\bar\delta, \phi)$ is a stable supermap to $\Ycal$ of class $\beta$, the class $\beta$ is of the form  $\beta=(1-\Pi)\beta_0$ for a cycle $\beta_0$ on the bosonic reduction $Y$ of $\Ycal$; moreover   the condition $\phi_\ast[\Xcal]=\beta$  is equivalent to $(\phi_{bos})_\ast[X_s]=\beta_0$.

\begin{lemma}\label{lem:stmapsoversch}\cite[Lem.\! 3.4]{BHRSM25} Let $S$ be a scheme. The data of a stable map $(f\colon\Xcal\to S, \{\Xcal_i\},\{\Zc_j\}, \bar\delta, \phi)$ into $\Ycal$ of class $\beta$ over an ordinary scheme $S$ is equivalent to the following data:
\begin{enumerate}
\item
a stable spin map $F_{spin}:=(f_{bos}\colon X \to S, \{X_i\},\{Z_j\}, \varpi, \phi_{bos})$ into $Y$ of class $\beta_0$;
\item
a morphism $\lambda_F\colon\phi_{bos}^\ast \F_\Ycal \to \Lcl$ of $\Oc_X$-modules, where  $\Lcl$ is defined by $\Oc_\Xcal=\Oc_X\oplus\Pi \Lcl$.
\end{enumerate}
\end{lemma}  

Using Lemma \ref{lem:stmapsoversch} we see that there is a forgetful stack morphism between Deligne--Mumford stacks
\begin{align}
\Ff\colon \Sf\Mf_{g,\nf_{NS},\nf_{RR}}(\Ycal,\beta)_{bos} &  \to \Mf^{spin}_{g,\nf_{NS},\nf_{RR}}(Y,\beta_0) \\
  (F_{spin},\lambda_F)& \mapsto F_{spin}.
\end{align} 

\begin{prop}\label{prop:bosredDM}\cite[Prop.\! 3.5]{BHRSM25} The forgetful stack morphism $\Ff$ is schematic and affine. The fiber of $\Ff$ over a stable spin map $F_{spin}=(f\colon X\to S,\{W_i\}, \{Z_j\},  \Lcl,\varpi, \psi)$ is an affine cone $V(\Qc)=\Spec( \Sym (\Qc))\to X$, where $\Qc$ is a coherent sheaf on $X$.
\qed
\end{prop}

Then, since $\Mf^{spin}_{g,\nf_{NS},\nf_{RR}}(Y,\beta_0)$ is proper by Proposition \ref{prop:spinstack2}, the   stack $\Sf\Mf_{g,\nf_{NS},\nf_{RR}}(\Ycal,\beta)_{bos}$ is proper if and only if $\Ff$ is an isomorphism.  
Note, however, that when the target $\Ycal$ is bosonic, i.e., it is an ordinary scheme $Y$,  the stack $\Sf\Mf_{g,\nf_{NS},\nf_{RR}}(Y,\beta)_{bos}$ coincides with $\Mf^{spin}_{g,\nf_{NS},\nf_{RR}}(Y,\beta_0)$, and therefore is always proper.

\appendix

\section{Supergroups and actions on superschemes}\label{sec:supergr}

 In this Appendix we collect some basic definitions about group actions in the category of superschemes.

\subsection{Group superschemes and algebraic supergroups}

{The notion of group scheme (see for instance \cite{Water}) easily generalizes to the super setting.}

\begin{defin}\label{def:supergroup} Let $\Sc$ be a superscheme. A  group $\Sc$-superscheme or a supergroup over $\Sc$ is a group object $\Gsc\to\Sc$ in the category of $\Sc$-superschemes. In other words, it is an $\Sc$-superscheme whose functor of   points is a group functor.
\end{defin}

Thus, for every $\Sc$-superscheme $\Tc\to \Sc$ the set $\Gsc(\Tc)=\Hom_\Sc(\Tc, \Gsc)$ is a group and for every morphism $\phi\colon \Tc'\to \Tc$ of $\Sc$-superschemes, the induced map $\phi^\ast\colon \Gsc(\Tc) \to \Gsc(\Tc')$ is a group morphism. When the groups $\Gsc(\Tc)$ are abelian, one says that group superscheme $\Gc$ is \emph{abelian}. 
There are    morphisms of $\Sc$-superschemes
\begin{align} 
\mu\colon & \Gsc\times_\Sc\Gsc\to\Gsc\ \text{(multiplication), }\\  \beta\colon & \Gsc\iso\Gsc\ \text{(taking the inverse), } \\\ e\colon & \Sc\to \Gsc\ \text{(unit)}
\end{align}
satisfying the usual properties. 

One easily sees that the bosonic reduction $G\to S$ of a group superscheme $\Gsc\to\Sc$ is an $S$-group scheme.

\begin{defin}\label{def:alggr} Let $\Sc$ be a superscheme. An algebraic supergroup over $\Sc$ is a group $\Sc$-scheme $\Gsc \to \Sc$ that is locally of finite type. An algebraic supergroup over a superalgebra $\As$  is a group superscheme over $\SSpec \As$. 
\end{defin}

\begin{defin}\label{def:normal} Let $\Gsc\to \Sc$ be a group superscheme. A subgroup superscheme of $\Gsc\to \Sc$ is a group superscheme $\Hc\to\Sc$ together with a superscheme immersion $\Hc\hookrightarrow\Gsc$ such that $\Hc^\bullet(\Tc)\hookrightarrow \Gsc^\bullet(\Tc)$ is a group morphism for every $\Sc$-superscheme $\Tc\to\Sc$. We say that $\Hc$ is a normal subgroup superscheme of $\Gsc$ if (the image of) $\Hc^\bullet(\Tc)$ is a normal subgroup of $\Gsc^\bullet(\Tc)$ for every $\Tc\to \Sc$.
\end{defin}
 
The subgroup superschemes of an algebraic supergroup are algebraic supergroups as well.

The easiest way to produce examples of algebraic supergroups is to represent group functors. We now describe a few cases.

\begin{example}[The linear supergroup] If $\As$ is a superring, the linear supergroup $\GL_{m|n}(\As)$ is the group of  (even) $\As$-linear automorphisms of $\As^{m|n}$.
 A  {linear} endomorphism of $\As^{m|n}$ is given by an $(m+n)\times (m+n)$    block matrix
$$
M=\begin{pmatrix} A & B \\ C& D\end{pmatrix}\,,
$$
where $A$ and $D$ have even entries, and $B$ and $C$ have odd entries. Moreover $M$ is invertible if and only if both $A$ and $D$ are invertible   matrices, that is, if and only if $\det A$ and $\det D$ are invertible in $\As_0$.

The linear supergroup defines a group functor $\SGLf_{m|n}(\As)$ over the category of $\As$-super\-schemes by letting
$$
\SGLf_{m|n}(\As)(\Tc)=\GL_{m|n}(\Oc_\Tc(\Tc))\simeq \operatorname{Aut}_{\Oc_\Tc}(\Oc_\Tc^{m|n})
$$
for every $\As$-superscheme $\Tc \to \SSpec\As$. This functor is a subfunctor of the even endomorphisms functor $\Tc \to \End_{\Oc_\Tc}(\Oc_\Tc^{m|n})_+$; it is known that the latter is representable \cite{BrHRPo20}: it is the functor of   points of the supervector bundle
$$
\widetilde\Vs(\End_\As(\As^{m|n}))= \SSpec (\SSym_\As(\End_\As((\As^{m|n})^\ast)\to \SSpec\As\,.
$$
If we write $\As^{m|n}=e_1\As\oplus\dots\oplus e_m\As\oplus \theta_1\As\oplus\dots\oplus \theta_n\As$, and $(\omega_1,\dots,\omega_m,\eta_1\dots,\eta_n)$ is the dual basis, then a free basis for $\End_\As((\As^{m|n})^\ast)\iso 
(\As^{m|n})^\ast\otimes \As^{m|n}$ is given by $$(x_{ij}=\omega_i\otimes e_j, y_{KL}=\eta_K\otimes\theta_L, \xi_{iL}=\omega_i\otimes \theta_L, \psi_{Kj}=\eta_K\otimes e_j)$$ where $1\leq i\le m$, $1\leq j\le m$, $1\leq K\le n$, $1\leq L\le n$, and then, 
$$
\widetilde\Vs(\End_\As(\As^{m|n}))= \SSpec \As[ x_{ij}, y_{KL},\xi_{iL}, \psi_{Kj}]\,.
$$
It follows that the supergroup functor $\SGLf_{m|n}(\As)$ \emph{is representable} by the open sub-superscheme $\SGL{m|n}(\As)$ of $\SSpec \As[ x_{ij}, y_{KL},\xi_{iL}, \psi_{Kj}]$, the  complement of the zeroes of the functions $\det X$ and $\det Y$, where $X$ and $Y$ are   even matrices with entries $x_{ij}$ and $y_{KL}$, respectively. That is,
\begin{align*}
\SGL_{m|n}(\As) &= \SSpec (\As[ x_{ij}, y_{KL},\xi_{iL}, \psi_{Kj}]_{\det X\cdot\det Y})\\
&= \SSpec \As[ x_{ij}, y_{KL},\xi_{iL}, \psi_{Kj}, z]/(z\cdot\det X\cdot\det Y-1)\,,
\end{align*}
where the subscript  $\det X\cdot\det Y$ denotes the localization with respect to the powers of the even function 
$\det X\cdot\det Y$.
In this way $\GL_{m|n}(\As)$ can be given the structure of an \emph{affine algebraic supergroup} over $\As$, which we   denote by $\SGL_{m|n}(\As)$.
\end{example}

\begin{example}[The additive supergroup]\label{ex:ga} If $\Sc$ is a superscheme, one defines a functor from the category of $\Sc$-superschemes to the category of abelian groups by
associating with  an $\Sc$-superscheme $f\colon\Tc\to\Sc$ the additive group $\Gamma(\Tc,\Oc_\Tc)$. By \cite[Prop. 2.13]{BrHRPo20} one has $\Gamma(\Tc,\Oc_\Tc)_+\simeq \Hom_\Sc(\Tc,\widetilde\Vs(\Oc_\Sc))$ and also $\Gamma(\Tc,\Oc_\Tc)_-\simeq \Hom_\Sc(\Tc,\widetilde\Vs(\Pi\Oc_\Sc))$. Then,
$$
\Gamma(\Tc,\Oc_\Tc)\simeq \Hom_\Sc(\Tc,\widetilde\Vs(\Oc_\Sc\oplus\Pi\Oc_\Sc))=\Hom_\Sc(\Tc,\As^{1|1}_\Sc)\,.
$$
That is, the relative affine space $\Gs_{a,\Sc}:=\As^{1|1}_\Sc\to\Sc$ is  an affine algebraic supergroup over $\Sc$, called the \emph{additive supergroup over $\Sc$}.
\end{example}

\begin{example}[The multiplicative supergroup]\label{ex:gm} If $\Sc$ is a superscheme, one can define another functor from the category of $\Sc$-superschemes to the category of abelian groups by
associating with an $\Sc$-superscheme $f\colon\Tc\to\Sc$ the multiplicative group $\Gamma(\Tc,\Oc_\Tc^\times)$ of   invertible superfunctions on $\Tc$.
A superfunction on $\Tc$ is invertible if and only if the associated superscheme morphism $f\colon \Tc\to \SSpec \Oc_\Sc[x,\theta]$ factors through the open sub-superscheme $\Gs^{1|1}_{m,\Sc}:=\SSpec \Oc_\Sc[x,x^{-1},\theta]$. Thus
$$
\Gamma(\Tc,\Oc_\Tc^\times)\simeq \Hom_\Sc(\Tc, \Gs^{1|1}_{m,\Sc})
$$
so that  
$\Gs^{1|1}_{m,\Sc}$
is an algebraic supergroup over $\Sc$, called the \emph{multiplicative supergroup over $\Sc$}.

The product and the inverse in the multiplicative supergroup are defined by the superring sheaf morphisms 
\begin{align*}
\Oc_\Sc[x,x^{-1},\theta]&\to \Oc_\Sc[x,x^{-1},\theta]\otimes_{\Oc_\Sc} \Oc_\Sc[x,x^{-1},\theta] \,, & \Oc_\Sc[x,x^{-1},\theta]&\to \Oc_\Sc[x,x^{-1},\theta]\\
x&\mapsto x\otimes x& x&\mapsto x^{-1}\\
\theta&\mapsto \theta & \theta&\mapsto \theta
\end{align*}

If $\As$ is a superring, there is a closed immersion $\Gs^{1|1}_{m,\As}\hookrightarrow \SGL_{1|1}(\As)$ of algebraic supergroups over $\As$ given by
\begin{align*}
\Gs^{1|1}_{m,\As}(\Tc)&\to \SGL_{1|1}(\As)(\Tc)\\
f&\mapsto \begin{pmatrix} f_+ & f_-\\ f_- & f_+ \end{pmatrix}
\end{align*}
for every $\Sc$-superscheme $\Tc$.
\end{example}

\begin{defin}\label{def:character}  A supercharacter on a group $\Sc$-superscheme $\Gsc \to \Sc$ is a morphism of group $\Sc$-superschemes
$$
\chi\colon \Gsc \to  \Gs^{1|1}_{m,\Sc}\,.
$$
\end{defin}

\subsection{Supergroup actions and principal superscheme bundles}\label{subs:psuper}

Let $\Gsc\to\Sc$ be a supergroup over a superscheme $\Sc$.

\begin{defin} \label{def:action} A left action of $\Gsc$ on a $\Sc$-superscheme $\Ycal$ is a morphism 
$$
\psi\colon\Gsc\times_\Sc \Ycal \to \Ycal 
$$
of $\Sc$-superschemes satisfying the usual compatibility requirements, that is, such that for every $\Sc$-superscheme $\Tc\to \Sc$, the induced map
$$
\Gsc(\Tc)\times \Ycal(\Tc) \to \Ycal(\Tc)
$$
is an ordinary group action.  The action is said to be free if the morphism
$$
(\psi\times p_2)\colon \Gsc\times_\Sc \Ycal \to \Ycal\times_\Sc \Ycal
$$
 is  injective, i.e., it is a superscheme monomorphism.
One can define right actions in an analogous way. 
\end{defin}

Every left action $\psi\colon\Gsc\times_\Sc \Ycal \to \Ycal$ gives a right action $\Psi\colon\Ycal\times_\Sc \Gsc \to \Ycal$, which is the composition of the isomorphism $\Ycal\times_\Sc \Gsc\to \Gsc\times_\Sc \Ycal$ given on points by $(y,g)\mapsto (g^{-1},y)$ and the left action $\psi$; this procedure is reversible.

Every action $\Gsc\times_\Sc \Ycal \to \Ycal$ of a supergroup on a superscheme induces an action $G\times_S Y \to Y$ of the group scheme $G$ on the bosonic reduction $Y$ of $\Ycal$.

\begin{defin}\label{def:equivariant} Let $\psi\colon\Gsc\times_\Sc \Ycal \to \Ycal$, $\phi\colon \Gsc\times_\Sc \Xcal\to \Xcal$  be two $\Gsc$-actions. An $\Sc$-morphism $f\colon \Ycal\to \Xcal$ is $\Gsc$-equivariant (or simply equivariant) if its compatible with the actions, that is, if
$$
\phi\circ (\Id\times f)=f\circ \psi\,.
$$
\end{defin}

\begin{defin}\label{def:principalbundle} Let $\Gsc\to\Sc$ be a supergroup. A principal $\Gsc$-superscheme bundle on an $\Sc$-superscheme $\Tc$ is a morphism  $\pi\colon \Psc\to \Tc$ of $\Sc$-superschemes together with a left  action $\psi\colon \Gsc\times_\Tc\Psc \to \Psc$ satisfying the following properties:
\begin{enumerate}
\item if we consider the trivial action of $\Gsc$ on $\Tc$, then $\pi\colon \Psc\to \Tc$ is $\Gsc$-equivariant,
\item the morphism
$(\psi, p_2)\colon\Gsc\times_\Sc \Psc\to \Psc\times_\Tc \Psc$
is an isomorphism.
\end{enumerate}
\end{defin}

 Morphisms of principal $\Gsc$-superscheme bundles over the same superscheme $\Sc$ are $\Gsc$-equivariant  superscheme morphisms  commuting with the projections to $\Sc$.

Principal superscheme bundles are preserved under base change, that is, if $\phi\colon\Tc'\to\Tc$ is a morphism of $\Sc$-superschemes and $\Psc\to \Tc$ is a principal $\Gsc$-superscheme bundle, the pullback $\phi^\ast\Psc:=\Psc\times_\Tc \Tc'\to \Tc'$ is a principal $\Gsc$-superscheme bundle.

\begin{defin}\label{def:trivial}
 A principal $\Gsc$-superscheme bundle $\Psc\to \Tc$ is \emph{trivial} if there is an isomorphism  $\Psc \iso \Gsc\times_\Sc \Tc$ of $\Gsc$-superscheme bundles. This is equivalent to the existence of a section of $\Psc\to\Tc$.
 \end{defin}

As in the bosonic case one has:
\begin{prop}\label{prop:alliso}
Let $\Sc$ be a superscheme. Every morphism $\phi\colon \Psc\to\Psc'$ of principal $\Gsc$-superscheme bundles over $\Sc$ is an isomorphism.
\end{prop}
\begin{proof} We have to prove that for every $\Sc$-superscheme $\Tc$, the map $\phi_\Tc\colon \Psc(\Tc)\to \Psc'(\Tc)$ between the sets of $\Tc$-valued points is a bijection. Let $p_1,p_2$ in $\Psc(\Tc)$ such that $\phi_\Tc(p_1)=\phi_\Tc(p_2)$. Since $\Gsc(\Tc)\times \Psc(\Tc)\iso \Psc(\Tc)\times\Psc(\Tc)$, the exists a unique $g\in\Gsc(\Tc)$ such that $p_2=g\cdot p_1$. Then $\phi_\Tc(p_1)=\phi_\Tc(p_2)=g\cdot \phi_\Tc(p_1)$, so that $g=1$ because $\Gsc(\Tc)\times \Psc'(\Tc)\iso \Psc'(\Tc)\times\Psc'(\Tc)$. It follows that $p_2=p_1$ thus proving the injectivity of $\phi_\Tc$. For the surjectivity, take $p'\in \Psc'(\Tc)$; now consider a point $p\in \Psc(\Tc)$. Since $\Gsc(\Tc)\times \Psc'(\Tc)\iso \Psc'(\Tc)\times\Psc'(\Tc)$, there exists $g\in \Gsc(\Tc)$ such that $p'=g\cdot \phi_\Tc(p)=\phi_\Tc(g\cdot p)$, thus finishing the proof.
\end{proof}

\section{Properties of superscheme morphisms}
\label{sec:app}

We extend to superschemes some more or less well-known facts about scheme morphisms, which we have been using in the main part of these notes.

\subsection{Odd dimension of a superscheme}
While there seems to be only one sensible definition of   even dimension of
a superscheme, i.e., the dimension of the underlying bosonic scheme, for what concerns the odd dimension there are several possibilities.
 Since this may be quite confusing, we shall spend some time trying to clarify the issue.
Let  $\Xcal=(X,\Oc_\Xcal)$ be a superscheme.
  As usual, here we denote by $\Jc$ the ideal generated by the odd sections so that the bosonic reduction of $\Xcal$ is the scheme $(X, \Oc_X=\Oc_\Xcal/\Jc)$. The quotient $\Ec=\Pi(\Jc/\Jc^2)$, where where $\Pi$ is the parity changing functor, is a locally finitely generated quasi-coherent $\Oc_X$-module.

\begin{defin}[1st definition] If $x\in X$ is a point, the odd dimension of $\Xcal$ at $x$ is 
the smallest number $n_x=\operatorname{odd-dim}_1(\Oc_{\Xcal,x})$ of generators of the ideal $\Jc_x$, or, equivalently, the smallest number of generators of the $\Oc_{\Xcal,x}$-module $\Ec_x$. Then $\operatorname{odd-dim}_1$ of $\Xcal$ is the supremum   over the points $x$ of $\Xcal$ of  $n_x$.
\end{defin}

If $\Xcal$ is irreducible with generic point $x_0$, then $n_x\geq n_{x_0}$ for every point $x\in X$, and there is an open $U\subseteq X$ such that $n_x= n_{x_0}$ whenever $x\in U$. 

\begin{defin}[2nd definition]\label{def:dim}  Given a superscheme $\Xcal$,
we call $\operatorname{odd-dim}_2(\Xcal)$   the smallest integer number $s$ such that  $\Jc^{s+1}=0$, that is, $\operatorname{odd-dim}_2$ is the order according to \cite[2.1]{BrHRPo20} or \cite[7.1.4]{We09}.  
\end{defin}

 Both odd dimensions may be infinite, but they are finite if $\Xcal$, or equivalently $X$, is quasi-compact. 
 When this happens, one has $ \operatorname{odd-dim}_2(\Xcal)\leq \operatorname{odd-dim}_1(\Xcal)$ with equality if $\Xcal$ is split.

Actually it  is useful to strengthen the second notion of odd dimension (compare with \cite[Def.\! 2.3]{BrHRMa23}).
\begin{defin}[3rd definition: pure odd dimension]\label{def:pureodddim}
A superscheme $\Xcal$ is of pure odd dimension $s$ if 
\begin{enumerate}
\item $\operatorname{odd-dim}_2(\Xcal)=s$, 
\item the sheaf $\Ec=\Pi(\Jc/\Jc^2)$ of $\Oc_X$-modules has   rank $s$ on every irreducible component of $X$, and
\item  $\Ec$ is a sheaf of pure dimension $m=\dim X$, that is, it has no subsheaves supported on a closed subscheme of dimension strictly smaller than  $m$.
\end{enumerate}
A superscheme $\Xcal$ is of pure dimension $m\vert s$ if it is of pure odd dimension $s$ and $\dim_2(\Xcal)=m\,\vert\,s$.
\end{defin}
 
\subsection{Smooth  superscheme morphisms }
A suitable notion of smooth and \'etale superscheme morphism  turns out to be the following  \cite[Def.~A.16]{BrHRPo20}, which utilizes the first definition of odd dimension.
\begin{defin}\label{def:smooth}  A superscheme morphism $f\colon \Xcal\to\Sc$   is smooth   of relative dimension $\dim_1 f=m\vert n$ if:
\begin{enumerate}
\item it is locally of finite presentation; if the superschemes are locally noetherian, it suffices to ask that $f$ is locally of finite type.
\item $f$ is flat, i.e., $f_\ast\Oc_\Xcal$ is a flat sheaf of $\Oc_\Sc$-modules;
\item If $\Xcal'\subset\Xcal$ and $\Sc'\subset\Sc$ are irreducible components\footnote{An irreducible component of a superscheme $\Xcal$ is a maximal irreducible closed sub-superscheme of $\Xcal$.}
such that $f(\Xcal')\subset \Sc'$,
then $\dim_1 \Xcal' = \dim_1\Sc' + m\vert n$;
\item the sheaf of relative differentials $\Omega_{\Xcal/\Sc}$ is locally free of rank $m\vert n$.
\end{enumerate}
\end{defin}

When $\Sc=\Spec k$  one has the notion of smooth  
superscheme over a field.

\begin{defin} \label{def:etale} An \'etale superscheme morphism $f\colon \Xcal\to\Sc$ is a smooth morphism of relative dimension $0|0$.
\end{defin}

One has the following criterion for smoothness \cite[Prop.\! A.17]{BrHRPo20}.

\begin{prop}\label{prop:smoothsplit}
 A superscheme $\Xcal$  of dimension $\dim_1\Xcal=m\vert n$,
locally of finite type over a field $k$,   is smooth if and only if
\begin{enumerate}
\item $X$ is a smooth scheme over $k$ of dimension $m$, and 
\item  the $\Oc_X$-module $\Ec=\Pi(\Jc/\Jc^2)$ is locally free of rank $n$ and the natural map ${\bigwedge}_{\Oc_X} \Ec\to Gr_{\Jc}(\Oc_\Xcal)$ is an isomorphism.\footnote{The morphism is defined as follows: a section of ${\bigwedge}_{\Oc_X}$ which is homogeneous of degree $p$ in the $\Z$-gradation has a representative in $\Jc^2$; then one takes its class in $Gr_{\Jc}(\Oc_\Xcal)$.}
\end{enumerate}
\end{prop}

\begin{proof} 
Assume that the two conditions hold. Since the question is local, we can assume that $\Xcal$ is affine and $\Ec$ is a free $\Oc_X$-module with
 generators $\bar\theta_1,\ldots,\bar\theta_n$. Choose   lifts $\theta_i\in \Jc$ of $\bar\theta_i$ for $i=1,\ldots,n$. 
 Since $X$ is smooth, there is a splitting $\Oc_X\to \Oc_{\Xcal,+}$ of the nilpotent extension $\Oc_{\Xcal,+}\to \Oc_X$. Thus, we get a homomorphism
$$\Oc_X[\theta_1,\dots,\theta_n]\to \Oc_{\Xcal}.$$ 
This  induces an isomorphism of the assiociated graded algebras and then is an isomorphism.
Hence  $\Xcal$ is locally split; we have $\Omega_{\Xcal/X}\simeq \Oc_X d\theta_1\oplus\ldots\oplus\Oc_X d\theta_n$, and the exact sequence of $\Oc_{\Xcal}$-modules
$$
\Omega_X\otimes_{\Oc_X}\Oc_{\Xcal}\to \Omega_{\Xcal}\to \Omega_{\Xcal/X}\to 0
$$
splits:
$$
\Omega_{\Xcal}\simeq (\Omega_X\otimes_{\Oc_X}\Oc_{\Xcal}) \oplus \Oc_X d\theta_1\oplus\ldots\oplus\Oc_X d\theta_n\,.
$$ 
Since $X$ is smooth, $\Omega_{X,x}$ is free of rank $m$ for every point $x$, and then $\Omega_{\Xcal}$ is locally free, so that $\Xcal$ is smooth.

For the converse, the question is again local, so we can also assume that $X$ is the spectrum of a local ring, that we still denote by $\Oc_X$. The exact sequence induced by $i\colon X\hookrightarrow \Xcal$ yields an exact sequence
$$
\Ec=\Jc/\Jc^2\xrightarrow{\delta} \Omega_{\Xcal}\otimes_{\Oc_{\Xcal}}\Oc_X \to \Omega_X \to 0
$$
where $\delta(\theta)$ is the class of $d\theta$ modulo $\Jc$. Taking  even and odd parts we get an isomorphism
$$\Omega_X  \simeq (\Omega_{\Xcal}\otimes_{\Oc_{\Xcal}}\Oc_X)_+ \,$$
and a surjection
\begin{equation}\label{eq:Ec-odd-Omega-surj}
\Ec\twoheadrightarrow (\Omega_{\Xcal}\otimes_{\Oc_{\Xcal}}\Oc_X)_- .
\end{equation}
Since $(\Omega_{\Xcal}\otimes_{\Oc_{\Xcal}}\Oc_X)_+$ is free of rank $m$ and $X$ has dimension $m$, we get that $X$ is smooth by \cite[17.15.5]{EGAIV-IV}.
On the other hand, by definition of odd dimension, there exists a surjection $\Oc_X^n\to \Ec$. Since $(\Omega_{\Xcal}\otimes_{\Oc_{\Xcal}}\Oc_X)_-$ is free of rank $n$,
this implies that \eqref{eq:Ec-odd-Omega-surj} is an isomorphism, so $\Ec$ is also free $\Oc_X$-modules of rank $n$. 

As $X$ is smooth, as before, we can choose a splitting $\Oc_X\to \Oc_{\Xcal,+}$. Let also $(\theta_1,\dots,\theta_n)$ be sections of $\Jc$ projecting to a basis of $\Ec$,
and let $(z_1,\dots,z_m)$ be a minimal set of  generators of the maximal ideal $\mathfrak m$ of $\Oc_X$. Then 
$(dz_1,\dots,dz_m,d\theta_1,\dots,d\theta_n)$ is a basis of $\Omega_{\Xcal}$.
It remains to prove that the natural epimorphism  
$$\rho:\Oc_X[t_1,\ldots,t_n]\to \Oc_{\Xcal}: t_i\mapsto\theta_i,$$ 
where $(t_i)$ are formal odd variables, is an isomorphism. We proceed by induction on $n$, the case $n=0$ being trivial. Let $\Sc$ be the closed sub-superscheme of $\Xcal$
defined by the ideal $(\theta_n)$. Then $\Omega_{\Sc}$ is a free $\Oc_{\Sc}$-module with the basis $(dz_1,\dots,dz_m,d\theta_1,\dots,d\theta_{n-1})$,
and by induction assumption we have an isomorphism
\begin{equation}\label{eq:Oc-X-t-Oc-mod-th-n}
\Oc_X[t_1,\ldots,t_{n-1}]\to \Oc_{\Sc}=\Oc_\Xcal/(\theta_n).
\end{equation}
Now, if an element $f\in \Oc_X[t_1,\ldots,t_n]$ is in the kernel of $\rho$, reducing modulo $\theta_n$ and applying the above isomorphism, we see that 
$$f=\sum a_{j_1\dots j_{i-1}}t_{j_1}\ldots t_{j_{i-1}} \cdot t_n,$$ 
where the sum runs over  $1\leq j_1<\dots<j_{i-1}\leq n-1$ and the coefficients are in $\Oc_X$. Take differentials in the identity
$$0=\rho(f)=\sum a_{j_1\dots,j_{i-1}}\theta_{j_1}\cdot\dots\cdot\theta_{j_{i-1}} \cdot\theta_n.$$ 
Since $(dz_1,\dots,dz_m,d\theta_1,\dots,d\theta_n)$ is a  basis of $\Omega_{\Xcal}$, looking at the coefficient of $d\theta_n$, we get 
$$0=\sum a_{j_1\dots,j_{i-1}}\theta_{j_1}\cdot\dots\cdot\theta_{j_{i-1}}.$$ 
Hence, using the  isomorphism \eqref{eq:Oc-X-t-Oc-mod-th-n}, we get that all the coefficients are zero, so $f=0$.
\end{proof}
 
\begin{corol}\label{cor:smoothsplit0} Let $\Xcal$ be a superscheme over a field $k$, $k\to k'$ a field extension, and $\Xcal'=\Xcal\times_{\Spec k}\Spec k'$ the superscheme over $k'$ obtained by base change. Then $\Xcal$ is smooth over $k$ if and only if $\Xcal'$ is smooth over $k'$.
\end{corol}
\begin{proof} We have  $\Jc'=\Jc\otimes_k k'$ so that the bosonic reductions are related by $X'=X\times_{\Spec k}\Spec k'$. Then, by a standard classical argument, $X'$ is smooth over $k'$ if and only if $X$ is smooth over $k$. Moreover $\Ec'=\Ec\otimes_k k'$ and then $\Ec'$ is locally free if and only if $\Ec$ is.  Furthermore, the map ${\bigwedge}_{\Oc_{X'}} \Ec'\to Gr_{\Jc'}(\Oc_\Xcal)'$ is the base change of ${\bigwedge}_{\Oc_X} \Ec\to Gr_{\Jc}(\Oc_\Xcal)$, so the former is an isomorphism if and only if the latter is. We finish by Proposition \ref{prop:smoothsplit}.
\end{proof}

\begin{corol}\label{corol:smoothsplit} If $\Xcal$ is smooth, then $\Xcal$ is of pure dimension $\dim_2\Xcal=\dim_1\Xcal$.
\end{corol}

\begin{proof} If $\Ec$ is locally free,  $\Xcal$ is pure of dimension $m$. It remains to be proved that $\operatorname{odd-dim}_2(\Xcal) = \rk \Ec =n$. But indeed
${\bigwedge}_{\Oc_X} \Ec\simeq Gr_{\Jc}(\Oc_\Xcal)$ implies $\Jc^n\ne 0$, $\Jc^{n+1}=0$. 
\end{proof}

 One might think of introducing another notion of smoothness, say 
 ``pure smoothness,''
using $\dim_2$ instead of $\dim_1$. However the two definitions are equivalent, as we are going to show.

\begin{defin}\label{def:smooth2}  A  superscheme morphism $f\colon \Xcal\to\Sc$  is  pure smooth of relative dimension $\dim_2 f=m\vert n$ if
\begin{enumerate}
\item $f$ is locally of finite presentation (if the superschemes are locally noetherian, it is enough to ask that $f$ is locally of finite type);
\item $f$ is flat;
\item condition (3) of Definition \ref{def:smooth} holds with $n=\operatorname{odd-dim}_2(\Xcal') - \operatorname{odd-dim}_2(\Ycal')$, and the fibers of $f\colon \Xcal'\to\Ycal'$ are superschemes of pure $\operatorname{odd-dim}_2$ $n$ (Definition \ref{def:pureodddim}); 
\item the sheaf of relative differentials $\Omega_{\Xcal/\Sc}$ is locally free of rank $m\vert n$.
\end{enumerate}
\end{defin}

\begin{prop}\label{prop:smoothsplit2} A superscheme $\Xcal$ is smooth 
of $\dim_1\Xcal=m\vert n$
if and only if  it  is pure smooth of $\dim_2\Xcal=m\vert n$.
 \end{prop}
\begin{proof} If $\Xcal$ is smooth, it is of pure dimension $\dim_2\Xcal=\dim_1\Xcal$ by Corollary \ref{corol:smoothsplit} and automatically pure smooth.
Assume now, for the converse, that $\Xcal$ is pure smooth. Since the question is local,   we can also assume that $\Xcal$ is the superspectrum of a local superring, which we still denote by $\Oc_\Xcal$. As in the proof of Proposition \ref{prop:smoothsplit} (see \cite[Prop.\! A.17]{BrHRPo20}),   $X$ is smooth of dimension $m$ and   there is a surjection
$$
\Ec\xrightarrow{d} \Pi(\Omega_{\Xcal}\otimes_{\Oc_{\Xcal}}\Oc_X)_- \to 0\,.
$$
Since $\Pi(\Omega_{\Xcal}\otimes_{\Oc_{\Xcal}}\Oc_X)_- $ is free of rank $n=\operatorname{odd-dim}_2(\Xcal)$, and $\Ec$ has rank   $n$, the morphism $d$ is an isomorphism on a dense open subscheme $U$ of $X$. Thus, $\ker d$ is supported in dimension strictly smaller than $m$, so that $\ker d=0$ by (3) of Definition \ref{def:pureodddim}. It follows that $\Ec$ is free of rank $n$ so that $\operatorname{odd-dim}_1(\Xcal)=n$; 
by reasoning as in the second part of the proof of  
 Proposition \ref{prop:smoothsplit} we see that ${\bigwedge}_{\Oc_X} \Ec\simeq Gr_{\Jc}(\Oc_\Xcal)$, so that
 $\Xcal$ is smooth.
\end{proof}

One has an analogous statement for morphisms.

\begin{prop}\label{prop:smoothsplit3} A superscheme morphism $\Xcal \to \Sc$ is smooth of relative  $\dim_1=m\vert n$ if and only if 
it  is pure smooth of  relative $\dim_2=m\vert n$.  \qed
\end{prop}

 For the reader's convenience, we finish this section by clarifying  the  definitions of surjective and faithfully flat superscheme morphisms. These notions are necessary to define the \'etale, fppf and fpqc topologies for superschemes.
 
\begin{defin}\label{def:surj} A superscheme morphism  $f\colon \Xcal \to \Ycal$  is surjective if the associated 
scheme morphism  $f_{\bos} \colon X_{\bos}\to Y_{\bos}$ is surjective as a map of sets. It is faitfhfully flat if it is flat and surjective.
\end{defin}
 
By following the same proof as in the case of schemes (see e.g.~\cite[Ch.~3, Exer.~16]{AtMcD69}), it is easy to show that if $\Xcal=\SSpec \Bs$ and $\Ycal=\SSpec\As$ are affine, a morphism
$f\colon \Xcal \to \Ycal$ is faithfully flat if and only if the corresponding  superring morphism $\As\to\Bs$ is faithfully flat, i.e., it is flat and
for every nonzero $\Z_2$-graded $\As$-module $M$, the $\Z_2$-graded $\Bs$-module $M\otimes_\As\Bs$ is nonzero as well.

\subsection{Formally smooth, \'etale and unramified  superscheme morphisms}

We  have (see \cite[Def.\! A.18]{BrHRPo20}):

 \begin{defin}\label{def:formallysmooth}
A  superscheme morphism  $f\colon \Xcal\to \Sc$  
is formally smooth (resp.\! formally unramified, resp.\! formally \'etale) if for every affine $\Sc$-superscheme $\SSpec(\Bs)$ and every nilpotent ideal $\Nc\subset \Bs$,
the morphism
$$
\Hom_\Sc(\SSpec(\Bs),\Xcal) \to \Hom_\Sc(\SSpec(\Bs/\Nc),\Xcal)
$$
is surjective (resp.\! injective, resp.\! bijective). 
\end{defin}

 A straightforward adaption of \cite[37.6.7]{Stacks} gives:
\begin{prop}\label{prop:formunram} A superscheme morphism   $f\colon \Xcal\to \Sc$  is formally unramified if and only if $\Omega_f=0$.
\qed
\end{prop}

 Taking $\Bs$ as an ordinary commutative ring, we easily see the following property:

\begin{prop}\label{prop:formallysmoothbos} If a superscheme morphism  $f\colon \Xcal\to \Sc$  
is formally smooth (resp.\! formally unramified, resp.\! formally \'etale), then the bosonic reduction  $f_{bos}\colon X \to S$  is a formally smooth (resp.\! formally unramified, resp.\! formally \'etale) scheme morphism. 
\qed
\end{prop}

\begin{corol}\label{cor:localsmooth} If a superscheme morphism $f\colon \Xcal\to \Sc$ is smooth (resp.\! \'etale) then the bosonic reduction  $f_{bos}\colon X \to S$  is a  smooth (resp.\! \'etale)  scheme morphism.
\end{corol}
\begin{proof} Clearly $f_{bos}\colon X \to S$ is locally of finite presentation. Then, the result follows from Proposition \ref{prop:formallysmoothbos}.
\end{proof}

As usual, we say that a superscheme $\Xcal$ over a field $k$ is formally smooth (resp.\! formally unramified, resp.\! formally \'etale) if the structural morphism $\Xcal \to \Spec k$ is so.

Using the definitions of smooth and \'etale superscheme morphism  (Definitions \ref{def:smooth} and \ref{def:etale}) we have:
\begin{prop}\label{prop:localsmooth}
A superscheme morphism   is smooth (resp.\! \'etale) if and only if it is locally of finite presentation and formally smooth (resp.\! formally \'etale).
\end{prop}
\begin{proof} The  smooth case is \cite[Prop.\! A.25]{BrHRPo20}. The \'etale case follows easily.
\end{proof}

As in the bosonic case, Proposition \ref{prop:localsmooth} has the following consequence:
\begin{corol}\label{cor:localsectsmooth} Let $f\colon \Xcal \to \Ycal$ be a smooth superscheme morphism. There exists a surjective (Definition \ref{def:surj}) \'etale  superscheme morphism $\Zc\to \Ycal$  such that $f_\Zc\colon \Xcal\times_\Ycal\Zc \to \Zc$ has a section.\qed
\end{corol}

\subsection{Unramified superscheme morphisms}

 We now define unramified morphisms.
\begin{defin}
A superscheme morphism  $f\colon \Xcal\to \Sc$   is unramified if it is locally of finite type and formally unramified.
\end{defin}

Note that we only require that $f$ is locally of finite type and not locally of finite presentation as in the \'etale and smooth cases. 
 
\begin{prop} \label{prop:smoothunramified} Let $f\colon \Xcal \to \Sc$ be a superscheme morphism.
\begin{enumerate}
\item If $f$ is \'etale, it is unramified.
\item If $f$ is locally of finite presentation, then $f$ is \'etale if and only if it is flat and unramified.
\end{enumerate}
\end{prop}
\begin{proof} (1) If $f$ is \'etale, it is locally of finite presentation and formally smooth by Proposition \ref{prop:localsmooth}. Then, it is also locally of finite type and formally unramified, and then it is unramified. (2) If $f$ is \'etale, it is unramified by (1). Conversely, if $f$ is unramified, it is formally unramified so that $\Omega_f=0$ by Proposition \ref{prop:formunram};  moreover, since $f$ is locally of finite presentation and flat,  it is \'etale by the very definition (Definition \ref{def:etale}).\end{proof}

Once more, one easily checks the following:
\begin{prop}\label{prop:basech} The properties of being formally smooth, formally \'etale and formally unramified are stable under base change and composition.
The same is true for the properties of being  smooth, \'etale and unramified.
\end{prop}
\begin{proof} The proof of  the smooth, \'etale and unramified cases relies mainly on the fact that also in the super case the K\"ahler differentials are compatible with base change.\footnote{Although the proof is an easy extension of the one for the ordinary case, we spell it out for the reader's convenience.
Let $\As \to \Bs$ be a superring morphism. The  $\Bs$-module $\Omega_{\Bs/\As}$ of  K\"ahler differentials  is characterized by its universal property: for every $\Z_2$-graded $\Bs$-module $\Ms$ there is an even functorial isomorphism of $\Bs$-modules
\begin{equation}\label{eq:diff}
 \operatorname{Der}_\As(\Bs,M)\iso \Hom_\Bs(\Omega_{\Bs/\As},M)\,,
\end{equation}
 that is, every graded derivation $\Bs\to \Ms$ over $\As$ factors as the composition of the K\"ahler differential $d\colon \Bs \to \Omega_{\Bs/\As}$ and a unique $\Bs$-module morphism $\varphi\colon \Omega_{\Bs/\As}\to M$.
 If $\phi\colon \As \to \As'$ is a superring morphism and $\Bs'=\As'\otimes_\As\Bs$, for every graded $\Bs'$-module $\Ms'$ one has 
 $$
\Hom_{\Bs'}(\Omega_{\Bs/\As}\otimes_{\Bs}\Bs', \Ms')\iso \Hom_\Bs(\Omega_{\Bs/\As},\Ms)\iso \operatorname{Der}_\As(\Bs,\Ms')\iso  \operatorname{Der}_{\As'}(\Bs',\Ms')\,,
 $$
 where the first isomorphism is adjunction, the second is Equation \eqref{eq:diff} and the third is given by  $D\mapsto 1\otimes D$. Applying the universal property of differentials (Equation \eqref{eq:diff}) for $\As'\to\Bs'$ we get a $\Bs'$-module isomorphism
$$
 \Omega_{\Bs/\As}\otimes_{\Bs}\Bs'  \iso \Omega_{\Bs'/\As'}\,,
 $$
 that is, K\"ahler differentials are compatible with base change.}
Then one proceeds as in the ordinary case (see e.g.~\cite[Thm.\! III.10.2]{Hart77}). 
\end{proof}
 

\begin{prop}\label{prop:fibersunram} Let $f\colon \Xcal\to\Sc$ a superscheme morphism, locally of  finite type. Then $f$ is unramified if and only if for every point $s=\Spec \kappa(s) \to \Sc$, the fiber $f_s\colon \Xcal_s \to \Spec\kappa(s)$ is formally unramified.
\end{prop}
\begin{proof} We have to prove that $f$ is unramified if and only if for every point $s=\Spec \kappa(s) \to \Sc$ $f_s\colon \Xcal_s \to \Spec\kappa(s)$ is formally unramified.  By the super Nakayama Lemma (see, for instance \cite{BBH91,CarcaFi11} or \cite[6.4.5]{We09}) one has that $\Omega_f=0$ is equivalent to $\Omega_{f\vert \Xcal_s}=0$ for every point $s$ of $S$.  Since $\Omega_{f_s} =\Omega_{f\vert \Xcal_s}$ by base change for the relative differentials (see the proof of Proposition \ref{prop:basech}), one concludes by  Proposition \ref{prop:formunram}.
\end{proof}

\begin{prop}\label{prop:unram} \ 
\begin{enumerate}
\item Let  $f\colon\Xcal\to S$ be a superscheme morphism where the target is an ordinary scheme and let  $i\colon X\hookrightarrow \Xcal$ be the natural immersion of the bosonic underlying scheme into $\Xcal$. If $f$ is formally unramified, then $i$ is an isomorphism $X\simeq\Xcal$ so that $\Xcal$ is an ordinary scheme.
\item Every formally unramified superscheme morphism is bosonic (i.e., schematic).
\end{enumerate}
\end{prop}
\begin{proof} (1) Let $\tau\colon \Xcal \to \Xcal$ be the involution that acts on the even part of $\Oc_\Xcal$ as the identity and on the odd part as multiplication by $-1$. Since $S$ is bosonic,  the restriction of $\tau$ to $X$ coincides with $i$, and then it has to be equal to the identity because $f$ is formally unramified, It follows that $\Xcal$ is even, that is, there is an isomorphism $\Xcal \iso \Xcal_0=\Xcal/\tau=\Spec (\Oc_\Xcal)_+$. The latter is an ordinary scheme and then, the  inverse isomorphism $\Xcal_0\iso \Xcal$ factors through the bosonic restriction $X\hookrightarrow \Xcal$.  Thus $i$ is an isomorphism.

(2) Follows from (1) and Proposition \ref{prop:basech}.
\end{proof}

\subsection{fiberwise properties of smooth, \'etale and unramified morphisms}

Let $f\colon \Xcal \to \Sc$ be a superscheme morphism, and let  $i\colon X \hookrightarrow \Xcal$, $j\colon S \hookrightarrow \Sc$, $\tilde{\iota}\colon X \hookrightarrow \Xcal_S$ be  the immersions of the bosonic underlying schemes.
 Consider the diagram
\begin{equation}\label{eq:diagram}
\xymatrix{
X \ar@{^{(}->}[r]
^{\tilde\iota}\ar[rd]_{f_{bos}}& \Xcal_S \ar@{^{(}->}[r] ^{j_\Xcal}\ar[d]^{f_S} & \Xcal\ar[d]^f\\
& S  \ar@{^{(}->}[r] ^{j}&\Sc
}
\end{equation}
where the square is cartesian and $i=j_\Xcal\circ \tilde{\iota}$. 

\begin{prop}\label{prop:unramprop}  Let $f\colon \Xcal\to\Sc$ a superscheme morphism, locally of finite type. The following conditions are equivalent:
\begin{enumerate}
\item $f$ is unramified;
\item $\Xcal_S$ is bosonic and $f_{bos}$ is unramified;
\item $f$ is bosonic and for every scheme $T$ and every morphism $T\to\Sc$, the morphism $f_T\colon \Xcal_T\to T$ is an unramified morphism of ordinary schemes;
\item for every point $s=\Spec k\to \Sc$, the fiber $\Xcal_s$ is an unramified scheme;
\item for every  point $s=\Spec k\to \Sc$, the fiber $\Xcal_s$  is a discrete and reduced scheme, i.e., $\Xcal_s\simeq\coprod_i \Spec K_i$, where each $K_i$ is a finite and separable extension of the residue field $\kappa(s)$.
\end{enumerate}	
\end{prop}
\begin{proof} If $f$ unramified, $f_S$ is  unramified by Proposition \ref{prop:basech}, and then $\Xcal_S$ is bosonic by Proposition \ref{prop:unram} so that $f_{bos}= f_S$, which is unramified. Thus (1) implies (2).

Assume now that (2) is true. Then one has a cartesian diagram
\begin{equation}
\xymatrix{
 X\ar@{^{(}->}[r] ^{i}\ar[d]^{f_S} & \Xcal\ar[d]^f\\
S  \ar@{^{(}->}[r] ^{j}&\Sc
}
\end{equation}
so that $i^\ast \Omega_f\simeq \Omega_{f_{bos}}=0$ as $f_{bos}$ is  unramified.  By \cite[Lemma 2.10]{BrHR21} we have   $\Omega_f=0$ and then $f$ is unramified by Proposition \ref{prop:unram}. Thus (2) implies (1).

(2) and (3) are equivalent, as every morphism from an ordinary scheme to $\Sc$ factors through  $S$, and we can then apply base change (Proposition \ref{prop:basech}). (1) and (3) are equivalent by Proposition \ref{prop:fibersunram}. Finally, for the equivalence between (4) and (5) see, for instance, \cite[Corollary 17.4.2]{EGAIV-IV}.
\end{proof}

Proposition \ref{prop:unramprop} relates the property of being unramified for a superscheme morphism with the corresponding por  the underlying bosonic morphism. We know study that  relationship for  \'etalness. To do that we first need the following result on flatness.

\begin{lemma}\label{lem:flatnilp2} \ 
\begin{enumerate}
\item Let $\As$ be a superring, $I\subset\As$ a nilpotent graded ideal and $M$ a $\Z_2$-graded $\As$-module. Then $M$ is flat over $\As$ if and only if $M/IM$ is flat over $\As/I$ and $\Tor_1^\As(\As/I,M)=0$.
\item
Let $f\colon \Xcal \to \Sc$ be a superscheme morphism. Then $f$ is flat if and only if it $f_S$ is flat and $\mathcal{T}{or}_1^{\Oc_\Sc}(\Oc_S,f_\ast\Oc_\Xcal)=0$.
\end{enumerate}
\end{lemma}

\begin{proof} (1) If $M$ is flat, then $M/IM=\As/I\otimes_\As M$ is flat over $\As/I$ by base change, and $\Tor_1^\As(\As/I,M)$ $=0$. Conversely, write $\bar\As=\As/I$ and $\bar M=M/IM$, and let $N$ be a graded $\bar\As$-module. Take an exact sequence $0\to K\to F\to N\to 0$ of graded $\bar\As$-modules with $F$ free. Since $\Tor$ commutes with direct sums we have
$$ \Tor_1^\As(F,M)= \Tor_1^\As(\oplus_j \bar\As,M) = \oplus_j \Tor_1^\As(\bar\As,M) = 0 .$$
Thus  $\Tor_1^\As(F,M)=0$, and then $\Tor_1^\As(N,M)=\ker(K\otimes_\As M\to F\otimes_\As M)$. For every graded $\bar\As$-module $P$ one has $P\otimes_\As M\simeq P\otimes_{\bar\As}\bar M$, so that this kernel is the kernel of $K\otimes_{\bar\As}\bar M\to F\otimes_{\bar\As}\bar M$, which vanishes because $\bar M$ is $\bar\As$-flat. Hence $\Tor_1^\As(N,M)=0$ for every graded $\bar\As$-module $N$. If now $N$ is a graded $\As$-module with $I^kN=0$, the exact sequence $0\to IN\to N\to N/IN\to 0$ and induction on $k$ give $\Tor_1^\As(N,M)=0$. Since $I$ is nilpotent, this applies to every graded $\As$-module $N$, and $M$ is flat.

(2) The problem is local both on target and source; then we can assume that  $\Sc=\SSpec \As$ and $\Xcal=\SSpec \Bs$ are affine. The result follows from (1) by taking $M=\Bs$ and $I=J_\As$.
\end{proof}

Now we have:

\begin{prop} \label{prop:etaleifbosis} Let $f\colon \Xcal \to \Sc$ be a locally of finite presentation superscheme morphism. Then $f$ is \'etale if and only if $\Xcal_S$ is bosonic, $f_{bos}$ is \'etale, and $\mathcal{T}{or}_1^{\Oc_\Sc}(\Oc_S,f_\ast\Oc_\Xcal)=0$.
\end{prop}
\begin{proof} If $f$ is \'etale, it is unramified by Proposition \ref{prop:smoothunramified}, and then $\Xcal_S$ is bosonic by Proposition \ref{prop:unramprop}, so that $f_{bos}=f_S$ which is \'etale because \'etalness is stable by  base change (Proposition \ref{prop:basech}). Moreover, the flatness of $f$ implies the vanishing of the Tor sheaf of the statement. For the converse, if the conditions of the statement are fulfilled, one has that $f_S=f_{bos}$ and then $f$ is flat by Lemma \ref{lem:flatnilp2}. Moreover, $f$ is unramified by Proposition \ref{prop:unramprop} and one finishes by applying Proposition \ref{prop:smoothunramified}. 
\end{proof}

We shall use the following fact: if $\Bs=\wt\Bs/K$, where $\wt\Bs$ is a polynomial superalgebra over $\As$ (in even and odd variables, possibly infinitely many) and $\Bs$ is formally smooth over $\As$, then the conormal map
\begin{equation}\label{eq:conormal}
\delta\colon K/K^2\to \Omega_{\wt\Bs/\As}\otimes_{\wt\Bs}\Bs\,,\qquad k\mapsto dk\otimes 1\,,
\end{equation}
is a split injection of graded $\Bs$-modules. This is the super version of \cite[Tag 031I, Lemma 10.138.7]{Stacks}, and the proof is the same: the identity of $\Bs$ lifts to an $\As$-algebra morphism $s\colon \Bs\to\wt\Bs/K^2$ because $\wt\Bs/K^2\to \Bs$ is a square-zero extension, and $D(p)=p-s(\bar p) \bmod K^2$ is an even $\As$-derivation $\wt\Bs\to K/K^2$ with $D(k)=k$ for $k\in K$; by the universal property \eqref{eq:diff} it defines a retraction of $\delta$.

\begin{lemma}\label{lem:lifting} Let $\As$ be a superring, $I\subset \As$ a graded ideal, and $\Bs$ a flat $\As$-superalgebra such that $\Bs/I\Bs$ is formally smooth over $\As/I$. Let $\Ss\to \Ss/\Nc$ be a surjection of $\As$-superalgebras, where $\Nc$ is a graded ideal with $\Nc^2=0$ and $I\Nc=0$. Then every $\As$-algebra morphism $g\colon\Bs\to\Ss/\Nc$ lifts to an $\As$-algebra morphism $\Bs\to\Ss$.
\end{lemma}
\begin{proof} Write $\bar\As=\As/I$, $\bar\Bs=\Bs/I\Bs$, and $\Bs=\wt\Bs/K$ with $\wt\Bs$ a polynomial superalgebra over $\As$ and $\pi\colon\wt\Bs\to\Bs$ the projection. Since $\wt\Bs$ is free, $g\circ\pi$ lifts to an $\As$-algebra morphism $\tilde h\colon\wt\Bs\to\Ss$. One has $\tilde h(K)\subset\Nc$ and, since $\Nc^2=0$, $\tilde h$ induces a morphism $\phi\colon K/K^2\to\Nc$, which is $\Bs$-linear when $\Bs$ acts on $\Nc$ through $g$. As $I\Nc=0$, $\Nc$ is actually a $\bar\Bs$-module, so that
$$
\phi\in\Hom_\Bs(K/K^2,\Nc)=\Hom_{\bar\Bs}(K/K^2\otimes_\Bs\bar\Bs,\Nc)\,,\qquad K/K^2\otimes_\Bs\bar\Bs=K/(K^2+IK)\,.
$$
We now use that $\Bs$ is flat over $\As$: since $\Tor_1^\As(\bar\As,\Bs)=0$, tensoring $0\to K\to\wt\Bs\to\Bs\to0$ by $\bar\As$ yields an exact sequence $0\to K/IK\to\wt\Bs/I\wt\Bs\to\bar\Bs\to 0$. Thus the kernel $\bar K$ of the presentation $\wt\Bs/I\wt\Bs\to\bar\Bs$ of $\bar\Bs$ as a quotient of a polynomial superalgebra over $\bar\As$ is $\bar K=K/IK$, and
$$
\bar K/\bar K^2=K/(K^2+IK)=K/K^2\otimes_\Bs\bar\Bs\,.
$$
Therefore $\phi$ is a $\bar\Bs$-linear map $\bar\phi\colon\bar K/\bar K^2\to\Nc$. Since $\bar\Bs$ is formally smooth over $\bar\As$, the conormal map $\bar\delta\colon\bar K/\bar K^2\to\Omega_{(\wt\Bs/I\wt\Bs)/\bar\As}\otimes\bar\Bs$ is a split injection by \eqref{eq:conormal}, so that $\bar\phi$ extends to a $\bar\Bs$-linear map $\Omega_{(\wt\Bs/I\wt\Bs)/\bar\As}\otimes\bar\Bs\to\Nc$, that is, by \eqref{eq:diff}, to an $\bar\As$-derivation $\wt\Bs/I\wt\Bs\to\Nc$. Composing with $\wt\Bs\to \wt\Bs/I\wt\Bs$ we get an even $\As$-derivation $D\colon\wt\Bs\to\Nc$ such that $D(k)=\phi(k)$ for every $k\in K$. Since $\Nc^2=0$ and $\wt\Bs$ acts on $\Nc$ through $\tilde h$, the map $h=\tilde h-D\colon\wt\Bs\to\Ss$ is again an $\As$-algebra morphism; it vanishes on $K$, hence factors through $\Bs$, and $h\equiv\tilde h\equiv g\circ\pi$ modulo $\Nc$. Thus $h$ is the required lifting.
\end{proof}

Without flatness one only knows that $\bar K$ is a quotient of $K/IK$, and $\phi$ need not factor through $\bar K/\bar K^2$; see Remark \ref{rem:affinetale}.

Now in analogy with Propositions \ref{prop:unramprop}  and \ref{prop:etaleifbosis}, which deal with the properties of being unramified or smooth, we study the  fiberwise properties of a smooth superscheme morphism.

\begin{prop} \label{prop:affinetale} Let $f\colon \Xcal\to\Sc$ be a superscheme morphism locally of finite presentation. 
\begin{enumerate}
\item If $f$ is flat, then $f$ is formally smooth if and only if $f_S$ is formally smooth.
\item $f$ is smooth if and only if $f_S$ is smooth and $\Tor_1^{\Oc_\Sc}(\Oc_S,f_\ast\Oc_\Xcal)=0$.
\end{enumerate}
\end{prop}
\begin{proof}
(1) If $f$ is formally smooth, so is $f_S$ by base change (Proposition \ref{prop:basech}). It is enough to prove the converse when both $\Xcal$ and $\Sc$ are affine. Assume that $f$ is flat and $f_S$ is formally smooth. By induction on the nilpotency index, it is enough to check the lifting property of Definition \ref{def:formallysmooth} for surjections $\Ss\to\Ss/\Nc$ of $\As$-superalgebras with $\Nc^2=0$. Let $g\colon\Bs\to\Ss/\Nc$ be an $\As$-algebra morphism, let $n$ be such that $J_\As^n=0$, and put $\Nc_k=\Nc\cap J_\As^k\Ss$, so that $\Nc=\Nc_0\supset\Nc_1\supset\dots\supset\Nc_n=0$. Each $\Ss/\Nc_{k+1}\to\Ss/\Nc_k$ is a surjection with square-zero kernel $\Nc_k/\Nc_{k+1}$, and this kernel is killed by $J_\As$ because $J_\As\Nc_k\subset\Nc\cap J_\As^{k+1}\Ss=\Nc_{k+1}$. Applying Lemma \ref{lem:lifting} with $I=J_\As$ to these surjections we successively lift $g$ to $\Ss/\Nc_1,\dots,\Ss/\Nc_n=\Ss$. Hence $f$ is formally smooth.

(2) It follows from (1) and Proposition \ref{prop:localsmooth}.
\end{proof}

\begin{remark} The proof of (2) of Proposition \ref{prop:affinetale} gives an alternative proof of Proposition \ref{prop:etaleifbosis}.
\end{remark}
 
\begin{remark}\label{rem:affinetale} \
\begin{enumerate}
\item The condition on $\mathcal{T}{or}_1^{\Oc_\Sc}(\Oc_S,f_\ast\Oc_\Xcal)$ cannot be dropped in (2) of Proposition \ref{prop:affinetale} nor in Proposition \ref{prop:etaleifbosis}. Let $\As=k[\theta]$ with $\theta$ odd, $\Bs=k=\As/(\theta)$ and $f\colon \Xcal = \SSpec \Bs \to \Sc=\SSpec\Bs$ the morphism induced by the superring morphism $\As \to \Bs$ that maps $\theta\mapsto 0$. Then $\Bs_A=k$ and $f_S$ is the identity, 
so that $f$ is formally unramified by (2); but $J_\As=k\theta\simeq\As/(\theta)$ as an $\As$-module, so that $J_\As\otimes_\As\Bs\simeq k\ne 0=J_\As\Bs$, and $f$ is neither flat nor formally smooth. Indeed, the identity $k\to\As/(\theta)$ does not lift to an $\As$-algebra morphism $k\to\As$.
\item  In (2) of Proposition \ref{prop:affinetale} one cannot replace $f_S$ by $f_{bos}$. Let $\As=k$ and $\Bs=k[\theta]$ with $\theta$ odd. Then $f$ is flat, $\Tor_1^\As(A,\Bs)=0$ trivially, and $f_{bos}\colon\Spec k\to\Spec k$ is the identity; but $\Omega_{\Bs/k}=\Bs\,d\theta\ne0$, so that $f$ is not \'etale. Here $\Xcal_S=\Xcal$ is not bosonic, in accordance with (5). In the same vein, the bosonic reduction of a flat morphism need not be flat: for $\As=k[t]$ and $\Bs$ the free $\As$-module with basis $1,u$ (even) and $\theta_1,\theta_2$ (odd), with $u^2=u\theta_i=0$ and $\theta_1\theta_2=tu$, the morphism $f$ is flat while $B=\Bs/J_\Bs=k[t]\oplus (k[t]/(t))\,u$ has $t$-torsion.
\end{enumerate}
\end{remark}

We finish this Subsection with a result that provides a quite useful alternative definition of a smooth superscheme morphism.
\begin{prop} \label{prop:smooth2}
A  superscheme morphism $f\colon \Xcal\to\Sc$  is smooth if and only if it is  locally of finite presentation, flat, and for every point $s=\Spec\kappa(s) \to \Sc$, the fiber $\Xcal_s\to \Spec \kappa(s)$ is a smooth superscheme over $\kappa(s)$.
\end{prop}

\begin{proof} A smooth morphism is locally of finite presentation and flat, and  its fibers are smooth by Proposition \ref{prop:basech}. For the converse, since smoothness is Zariski local on source, we are reduced  to show that $f$ is smooth at every point $x\in\Xcal$, that is, in an open neighbourhood of $x$. We can then assume that both $\Xcal$ and $\Sc$ are affine. Let us write $\Xcal=\SSpec\Bs$, $\Sc=\SSpec\As$ and $s=f(x)$.
We have
$$
\Bs=\As[t_1,\dots,t_p,\theta_1,\dots,\theta_q]/(f_1,\dots,f_r,\phi_1,\dots,\phi_u)
$$
where $t_i$, $f_j$ are even and $\theta_i $, $\phi_j$ are odd. The  fiber of $s=\Spec \kappa(s)$ is $\Xcal_s=\SSpec \Bs\otimes_\As\kappa(s)$ and
$$
\Bs\otimes_\As\kappa(s)=\kappa(s)[t_1,\dots,t_m,\theta_1,\dots,\theta_q]/(\bar f_1,\dots,\bar f_p,\bar \phi_1,\dots,\bar\phi_r)\,.
$$
Since $\Xcal_s$ is smooth over $\kappa(s)$ of some dimension $m|n$, we can reorder  variables and equations and select lifted polynomials $f_1,\dots,f_h,\phi_1,\dots,\phi_k$ with $h=p-m$, $k=q-n$ such that the even determinants
$$
J(t)= \det \left( \frac{\partial f_i} {\partial t_j}\right)_{1\leq i,j\leq h} \qquad
J(\theta)= \det \left(\frac{\partial \phi_i}{\partial \theta_j}\right)_{1\leq i,j\leq k} 
$$
do not vanish at $x$ (see \cite[Prop.\! A.20]{BrHRPo20}).  

Consider now the $\As$-superalgebra 
$$
\Bs'=\\As[t_1,\dots,t_p,\theta_1,\dots,\theta_q]/(f_1,\dots,f_h,\phi_1,\dots,\phi_k)\,.
$$
and the localization $\Bs'_J$ where $J=J(t)\cdot J(\theta)$.
Now $i\colon\Xcal'= \SSpec \Bs'_J \to \Sc$ is a standard smooth superscheme morphism of relative dimension $m|n$ in an open neighborhood of $x$ \cite[Prop.\! A.20]{BrHRPo20}.

Moreover,  we have a surjection $\pi\colon \Bs'_J \to \Bs_J=\Bs\otimes_{\Bs'}\Bs'_J$ and this induces a closed immersion $i\colon\Xcal'=\SSpec \Bs'_J \hookrightarrow \Xcal$. Passing to the fiber on $s$, since $\Xcal_s$  and $\Ucal'_s$ are smooth superschemes at the point $x$ over $\kappa(s)$ of the same dimension $m|n$, the morphism $i_x\colon\Xcal'_x \hookrightarrow \Xcal_s$  is  \'etale at $x$. Then $\pi$ induces an isomorphism $\Oc_{\Xcal',x} \iso \Oc_{\Xcal,x}$ ol local superrings, so that the kernel $\Ic=\ker\pi$ satisfies $\Ic_s\cdot\Oc_{\Xcal',x}=0$. We now apply flatness: since $\Bs_J$ is flat over   $\As$, one has   $\Ic_s\simeq \Ic\otimes_\As\kappa(s)$, and then $(\Ic\otimes_\As\kappa(s))_x=0$. Furthermore, the ideal $\Ic$ is finitely generated because both $\Bs_J$ and $\Bs'_J$ are $\As$-superalgebras of finite presentation; then the super Nakayama Lemma (\cite{BBH91,CarcaFi11} or \cite[6.4.5]{We09})  gives   $\Ic_x=0$. Therefore, the closed immersion $i\colon \Xcal' \to \Xcal$ is an isomorphism in an open neighborhood of   $x$. Since $\Xcal'\to \Tc$ is standard smooth at $x$, so is $f\colon \Xcal \to \Tc$ and one finishes.
\end{proof}

\subsection{Quasi-finite and radicial superscheme morphisms}

\begin{defin}\label{def:qf} A  superscheme  morphism $f\colon \Xcal\to\Ycal$  is locally quasi-finite (resp.\! quasi-finite) if it is locally of finite type and has  discrete (resp.\! finite) bosonic fibers. This is equivalent to $f_{bos}\colon X \to Y$ being a locally quasi-finite (resp.\! quasi-finite)  scheme morphism.
\end{defin}


By Proposition \ref{prop:unramprop} one has:
\begin{corol}\label{cor:unramprop} An unramified superscheme morphism is locally quasi-finite.
\qed\end{corol}
 
We also have:
\begin{prop}\label{prop:unramcomp}
Let $f\colon \Xcal\to\Ycal$, $g\colon\Ycal\to \Zc$ be morphisms of superchemes. If the composition $g\circ f\colon \Xcal \to \Zc$ is unramified, so is $f$.
\qed
\end{prop}

\begin{defin} A  superscheme morphism $f\colon \Xcal\to\Sc$  is radicial if the diagonal morphism $\Delta_f\colon \Xcal \to \Xcal\times_\Sc\Xcal$ is 
surjective.
\end{defin}

Since relative differentials are stable by base change, we obtain:
\begin{prop}\label{prop:basechrad} The property of being radicial is stable by base change.
\end{prop}

\begin{prop}\label{prop:mono} Let $f\colon \Xcal\to \Sc$  be a superscheme morphism,  locally of finite type. 
The following conditions are equivalent:
\begin{enumerate} 
\item $f$  is a monomorphism (Definition \ref{def:injective});
\item $f$ is  radicial and formally unramified;
\item for every point $s\colon \Spec \kappa(s) \to S\hookrightarrow\Sc$ the fiber $\Xcal_s=\Spec \kappa(s)\times_\Sc \Xcal$ is either empty or there is an isomorphism $\Xcal_s\iso\Spec \kappa(s)$ of $\kappa(s)$-schemes. 
\end{enumerate}

\end{prop}
\begin{proof}  
If $f$ is a monomorphism,  the diagonal $\Delta_f$ is an isomorphism by Proposition \ref{prop:isodiag}. Then, it is surjective, so that $f$ is radicial, and the conormal sheaf $\Omega_f$ to the diagonal $\Delta_f$ is zero,
 which implies that $f$ is formally unramified by  Proposition \ref{prop:formunram}.  Thus, (1) implies (2).  To prove the converse, if $f$ is  radicial and formally unramified, the diagonal morphism $\Delta_f$ is surjective, and its conormal sheaf $\Omega_f$ vanishes. As in the bosonic case (see \cite[Cor.\! 17.4.2]{EGAIV-IV}), the  condition $\Omega_f=0$ implies that $\Delta_f$ is an open immersion; being surjective, $\Delta_f$  has to be an isomorphism, and then $f$ is a monomorphism by Proposition \ref{prop:isodiag}.

We now prove that (2) implies (3). 
If $f$ is radicial and formally unramified, the morphism $f_s\colon \Xcal_s=\Xcal\times_\Sc\Spec \kappa(s) \to \Spec \kappa(s)$ is also radicial and formally unramified. Assume that $\Xcal_s$  is nonempty.  
By Proposition \ref{prop:unram}, on has $X_s\simeq \Xcal_s$. Thus, $\Xcal_s=X_s\to \Spec \kappa(s)$ is a radicial and unramified scheme morphism, so that $\Xcal_s\simeq \Spec \kappa(s)$ by \cite[Prop.\! 17.2.6]{EGAIV-IV}. Thus (2) implies (3).

Let us prove that (3) implies (2). If (3) is true, for every point $x\in X$ there is an isomorphism $\Xcal_s=X_s \simeq \Spec\kappa(s)$ where $s=f(x)$. It follows that $\Delta_f$ is surjective, and then $f$ is radicial. Moreover, the stalk $(\Omega_{\Xcal_s})_x$ is zero, and then we also have $(\Omega_f)_x=0$, so that $f$ is formally unramified by Proposition \ref{prop:formunram}, and we finish. 
\end{proof}

 We conclude this part with a result on affine and finite morphisms.
\begin{lemma}\label{lem:sections} Let $\Xcal$ be a superscheme, $a\in \Gamma(X,\Oc_\Xcal)$ an even section, and $X_a$ the open subset of points where $a$ does not vanish. If $\Xcal$ is quasi-compact and separated, the canonical superring morphism
$$
\Gamma(X,\Oc_\Xcal)_a \to \Gamma(X_a,\Oc_\Xcal)
$$
 is an  isomorphism.
\end{lemma}
\begin{proof} Write $\As=\Gamma(X,\Oc_\Xcal)$ and consider the canonical superscheme morphism $\phi\colon \Xcal \to \SSpec \As$. If $U_a$ is the standard open subscheme of $\SSpec \As$ associated with $a$, one has $X_a=\phi^{-1}(U_a)$. Since $\phi$ is quasi-compact and separated, $\phi_\ast \Oc_\Xcal$ is a quasi-coherent sheaf of superrings, and then $\phi_\ast \Oc_\Xcal\simeq \widetilde \Gamma(X,\Oc_\Xcal)=\widetilde \As$. Thus, 
$$
\Gamma(X_a, \Oc_\Xcal)=\Gamma(U_a, \phi_\ast \Oc_\Xcal)=\Gamma(U_a, \widetilde\As)=\As_a\,.
$$
\end{proof}

\begin{prop}\label{prop:finitebos} Let  $f\colon\Xcal\to\Ycal$ be a superscheme morphism and $f_{bos}\colon X \to Y$ its bosonic reduction. 
Then, $f$ is affine (resp.\! finite) if and only i$f_{bos}$ is affine (resp.\! finite).
\end{prop}
\begin{proof} It is clear that if $f$ is affine (resp.\! finite), so is $f_{bos}$. For the converse statement assume first that $f_{bos}$ is affine. To prove that $f$ is affine, we can assume that $\Ycal$ is affine. We are reduced to see that a superscheme $\Xcal$ whose bosonic reduction $X$ is affine, is affine as well. 
Note that since $X$ is affine, it is quasi-compact and separated, and then $\Xcal$ is also quasi-compact and separated (see \cite[Prop.\! A.13]{BrHRPo20}). Then, 
\begin{equation}\label{eq:global}
\Gamma(X,\Oc_\Xcal)_a \iso \Gamma(X_a,\Oc_\Xcal)
\end{equation}
 for every even global section a of $\Oc_\Xcal$ by Lemma \ref{lem:sections}.

Let $\Jc$ be the ideal of $\Oc_\Xcal$ generated by the odd sections, so that $\Oc_X=\Oc_\Xcal/\Jc_\Xcal$. For every integer $k$, $\Jc^k/\Jc^{k+1}$ is a quasi-coherent $\Oc_X$-module (actually of finite type) and then it is acyclic.
Since $X$ quasi-compact, one has $\Jc^{n+1}=0$ for some $n$. Now, from the exact sequences
$$
0 \to \Jc^{k+1}\to \Jc^k \to \Jc^k/\Jc^{k+1}\to 0
$$
we see by descending induction that the sheaves $\Jc^k$ are acyclic for every $k\geq 0$. In particular, $\Jc$ is acyclic and then the morphism $\Gamma(X,\Oc_\Xcal) \to \Gamma(X,\Oc_X)$ is surjective, so that the open subsets $X_a$ of $X$ for the even global sections $a$ of $\Oc_\Xcal$ form a covering of $X$, since the same is true for the global sections of $\Oc_X$ because $X$ is affine. Thus, Equation \eqref{eq:global} implies that there is an isomorphism of sheaves of superrings
$$
\Gamma(X,\Oc_\Xcal)\widetilde{\ } \iso \Oc_\Xcal\,,
$$
where $\Gamma(X,\Oc_\Xcal)\widetilde{\ }$ is the sheaf associated with the presheaf of localizations of the superring $\Gamma(X,\Oc_\Xcal)$. This proves that $\Xcal$ is affine.

Assume now that $f_{bos}$ is finite. Then $f$ is affine by the first part of this Proposition. As before, we can assume that $\Ycal$ is affine. Moreover
  $f_\ast\Oc_{X}=(f_{bos})_\ast\Oc_X$ is a finite $\Oc_Y$-algebra, and then also a finite $\Oc_\Ycal$-algebra. Since $\Jc_\Xcal^{n+1}=0$ for some $n$,  there is a finite filtration of $\Oc_\Xcal$ whose quotients $\Jc_\Xcal^k/\Jc_\Xcal^{k+1}$ are finite $\Oc_X$-modules. Then, their direct images  are finite over $\Oc_Y$ since $f_{bos}$ is finite and thus they are also finite over $\Oc_\Ycal$. It follows that $f_\ast \Oc_\Xcal$ is finite over $\Oc_\Ycal$ and we finish.     
\end{proof}

\subsubsection*{Formally unramified group superschemes}

Let $f\colon \Gsc \to \Sc$  be an $\Sc$-group superscheme and $\Omega_{\Gsc/\Sc}$ the relative cotangent sheaf. Proceeding as in \cite[Tag 04.7.I]{Stacks} one has:
\begin{prop}\label{prop:diffree}  There is a natural isomorphism, of $\Oc_\Gsc$-modules
$$
\Omega_{\Gsc/\Sc}\iso f^\ast e^\ast \Omega_{\Gsc/\Sc}\,,
$$
where $e\colon \Sc \to \Gsc$ is the unit section. In particular, if $\Sc$ is the spectrum of a field $k$ (that is, $\Gsc$ is a group superscheme over $k$), the relative cotangent sheaf  $\Omega_{\Gsc/\Sc}$ is a free $\Oc_{\Gsc}$-module.
\qed\end{prop}

Now let $\Gsc$ be a supergroup locally of finite type over a field $k$ and let $$T_e(\Gsc)=\Der_k(\Oc_{\Gsc,e}, \kappa(e))$$ be  the tangent superspace at the origin. One has $T_e(\Gsc)\simeq (\Omega_{\Oc_{\Gsc,e}} \otimes_{{\Oc_{\Gsc,e}}}\kappa(e))^\ast\simeq (e^\ast \Omega_{\Gsc/\Spec k})^\ast$, and then $e^\ast \Omega_{\Gsc/\Spec k}\simeq T_e(\Gsc)^\ast$, because both are finite dimensional superspaces.
\begin{corol} \label{cor:group} If $T_e(\Gsc)=0$, there is an isomorphism 
$\Gsc\simeq\coprod_i \Spec K_i$, where each $K_i$ is a finite and separable extension of $k$.
\end{corol}
\begin{proof} One has $\Omega_{\Gsc/\Spec k}=0$ by Proposition \ref{prop:diffree}, so that $\Gsc$ is an  unramified superscheme.  Then $\Gsc$ is bosonic by
Proposition \ref{prop:unramprop}, and we finish.
\end{proof}
 
\subsection{Quasi-affine superscheme morphisms}

We start by defining quasi-affine superschemes and morphism and studying basic properties of them.

\begin{defin}\label{def:quasiaffine} A superscheme is quasi-affine if it is a quasi-compact open sub-superscheme of an affine superscheme. A superscheme morphism $f\colon \Xcal \to \Sc$  is \emph{quasi-affine} if the pre-image of every affine open sub-superscheme $\Ucal\hookrightarrow \Sc$ is a quasi-affine open sub-superscheme of $\Xcal$.
\end{defin}

 One easily proves the following:
\begin{lemma}\label{lem:qaffine1}
A quasi-affine superscheme is separated. Analogously, a quasi-affine superscheme morphism is separated as well.
\end{lemma}
 
Proceeding as in \cite[Sections 28.18, 29.13]{Stacks} one {can characterize quasi-affine superschemes and morphisms.}

\begin{prop}\label{prop:qaffine}\
\begin{enumerate}
\item A superscheme $\Xcal$ is quasi-affine if and only if  the canonical morphism 
$$\Xcal \to \SSpec (\Gamma(X,\Oc_\Xcal))$$
 is a quasi-compact open immersion.
\item
Let $f\colon \Xcal \to \Sc$ be a superscheme morphism. If $f$ is quasi-affine, $f_\ast\Oc_\Xcal$ is a quasi-coherent sheaf of $\Oc_\Sc$-algebras, and in the natural factorization
\begin{equation}\label{eq:qaffine}
\Xcal \xrightarrow{g} \SSpec f_\ast\Oc_\Xcal \to \Sc
\end{equation}
of $f$, the  morphism $g$ is an open immersion (necessarily quasi-compact). Conversely, if $f$ is quasi-compact and separated, so that $f_\ast\Oc_\Xcal$ is a quasi-coherent sheaf of $\Oc_\Sc$-algebras, and in the factorization of $f$  given in Equation \eqref{eq:qaffine} the morphism $g$ is an open immersion, then $f$ is quasi-affine.
\end{enumerate}
\end{prop}
\begin{proof}
(1) Let us write $\As=\Gamma(X,\Oc_\Xcal)$. If $\Xcal \to \SSpec \As$
 is a quasi-compact open immersion, then $\Xcal$ is quasi-affine. Conversely, if $\Xcal$  is quasi-affine, it is an open sub-superscheme of an affine superscheme $\SSpec \Bs$. There is a canonical superring morphism $\Bs\to\As$ so that the open immersion $\Xcal\hookrightarrow \SSpec \Bs$ factors as
 $$
 \Xcal \to \SSpec\As \to \SSpec \Bs\,.
 $$
 The superscheme $\Xcal$ can be covered by  basic open sub-superschemes $\Xcal'_b =\SSpec \Bs_b$ of $\SSpec\Bs$ where $b$ is an even element of $\Bs$.  If $a$ is the image of $b$ in $\As$,  and $\Xcal_a$ is the open sub-superscheme  of $\Xcal$ complementary to the zeroes of $a$, then $\Xcal_a=\Xcal'_b$ and  thus  $\As_a\simeq \Bs_b$ by Lemma \ref{lem:sections}. Therefore the morphism $\Xcal \to \SSpec \As$ induces a superscheme isomorphism $\Xcal_a\iso \U_a$, where $\U_a$ is the open sub-superscheme of $\SSpec \As$ complementary to  the zeroes of $a$. This concludes the proof.
 
 (2) follows straightforwardly from (1).
\end{proof}
 
We now study \'etale localization of quasi-finite superscheme  morphisms. Our plan is to extend to the super setting the corresponding results for superschemes morphisms; we focus on \cite[A.7.1, A.7.2]{Alp24} for the statements to generalize. To do so, we need a preliminary result that we now state.
As in the ordinary case, if $\Sc$ is a superscheme and $z\in S$ a point, an \emph{elementary \'etale neighborhood} of $(\Sc,s)$ is an \'etale  superscheme morphism $f\colon\Sc'\to\Sc$ and a point $s'\in S'$ such that $f(s')=s$ and $\kappa(s')\simeq\kappa(s)$ (see \cite[Tag 02LD]{Stacks} for the scheme case).

 The following Lemma is the key point for determining the structure of quasi-affine superscheme  morphisms.

\begin{lemma}\label{lem:etaleextension}
Let $\Sc$ be a superscheme and $s\in S$ a point. If $(S',s') \to (S,s)$ is an elementary \'etale neighborhood at $s$ of the bosonic reduction $S$, there exists an elementary \'etale neighborhood $(T,t) \to (S',s')$ such that the composition $(T,t)\to (S',s') \to (S,s)$ can be extended to an elementary \'etale neighborhood $(\Tc, t) \to (\Sc,s)$, whose bosonic reduction is $(T,t) \to (S',s')$.
\end{lemma}

\begin{proof}
The question is  Zariski local on $\Sc$, so that we can assume that $\Sc=\SSpec \As$ is affine. Then $S=\Spec A$ where $A$ is the bosonic reduction of $\As$. Now, we can consider an affine neighborhood $T=\Spec B\hookrightarrow S'$ of $s'$; then, the composition $T \hookrightarrow S'\to S$  is induced by an \'etale ring morphism $A \to B$. By the local structure of the \'etale morphisms of schemes (see, for instance, \cite[Section 4]{Knut71} or \cite[Tag 00TA]{Stacks} for the smooth case), by conveniently   localizing  $S$ and $T$ we can assume that the $A$-algebra $B$ has the standard form $B\simeq A[x_1,\dots,x_n]/(f_1,\dots, f_n)$ for polynomials $f_i$ such that the matrix $\left ( \partial f_i/\partial x_j\right)$ is invertible in $B$. Now, take polynomials $F_i\in \As[x_1,\dots,x_n]$ that  are mapped  to   $f_i$ by the surjection $ \As[x_1,\dots,x_n] \to A[x_1,\dots,x_n]$. If we consider the $\As$-superalgebra  $\Bs= \As[x_1,\dots,x_n]/(F_1,\dots, F_n)$, the matrix $\left ( \partial F_i/\partial x_j\right)$ maps to $\left ( \partial f_i/\partial x_j\right)$ under bosonic reduction. Then, it is invertible in $\Bs$, and   the induced morphism $\Tc=\SSpec \Bs\to \Sc$ is \'etale by the local structure of the \'etale morphims of superschemes (see \cite[Prop.\! A.20]{BrHRPo20} for the smooth case and note that the noetherian hypothesis written there is not necessary) and its bosonic reduction is $T\to S$. 
\end{proof}

We can now generalize  \cite[A.7.1 and A.7.2]{Alp24} to the super setting.

\begin{prop}\label{prop:qfinitelocal}
Let $f\colon\Xcal \to \Sc$ 
be a quasi-finite and separated superscheme morphism. For every point $s\in S$ there exists an elementary \'etale neighborhood $(\Tc,t) \to (\Sc,s)$  and a decomposition $\Xcal\times_\Sc\Tc =\Zc \coprod \Wc$ as a disjoint union of sub-superschemes
such that $f_{\Tc\vert\Zc}\colon \Zc \to \Tc$ is finite and the fiber $\Wc_t$ is empty. 
\end{prop}
\begin{proof} By \cite[A.7.1]{Alp24} applied to the bosonic reduction $f_{bos}\colon X \to S$ of $f$, there exist an elementary \'etale neighborhood $(S',s') \to (S,s)$ and a decomposition $X\times_S S' =Z' \coprod W'$ as disjoint union of subschemes
such that $Z' \to S$ is finite and the fiber of $W'$ over $s'$ is empty. These properties are preserved under a base change by an elemetary \'etale neighborhood $(T,t)\to (S',s')$.  By Lemma \ref{lem:etaleextension} we can assume that the composition $(T,t)\to (S',s')\to (S,s)$ is the bosonic reduction of an elementary \'etale neighborhood $(\Tc,t)  \to (\Sc,s)$. Then, the decomposition of $X \times_S T =Z \coprod W$ (where $Z=Z'\times_{S'} T$ and $W=W'\times _{S'}T$) gives rise to a decomposition $\Xcal\times_\Sc\Tc =\Zc \coprod \Wc$ as superschemes 
such that $\Wc_t=\emptyset$ and $Z \to T$ is a finite scheme morphism. Since $Z\to T$ is the bosonic reduction of $f_{\Tc\vert\Zc}\colon \Zc \to \Tc$, Proposition \ref{prop:finitebos} implies that $f_{\Tc\vert\Zc}$  is a finite superscheme morphism, which fact completes the proof.\end{proof}

\begin{prop}\label{prop:qaffine2} Let  $f\colon \Xcal\to\Sc$  be a quasi-finite and separated superscheme morphism. There is a natural factorization of $f$ as a composition
$$
\Xcal \xrightarrow{g} \SSpec f_\ast\Oc_\Xcal \to \Sc
$$
where $g$ is an open immersion (necessarily quasi-compact), and then $f $is quasi-affine (see Proposition B.35).
\end{prop}
 
\begin{proof} The proof is analogous to that of \cite[A.7.2]{Alp24}. We first note that the formation of $f_\ast\Oc_\Xcal$ and the decomposition are preserved by \'etale base changes. Then it is enough to prove that for every point $s\in S$ there is an \'etale neighborhood $(\Tc,t)\to (\Sc,s)$ such that the statement is true for $\Xcal_\Tc \to \Tc$. 
 By Proposition \ref{prop:qfinitelocal} 
we can then assume  that there is a decomposition $\Xcal=\Zc\coprod \Wc$ as a disjoint union of superschemes such that $\Zc\to \Sc$ is finite and $\Wc_s=\emptyset$. 
Now, it is enough to prove the statement in the case $\Sc=\SSpec f_\ast\Oc_\Xcal$. In this case, since $\Oc_\Xcal=\Oc_\Zc\times \Oc_\Wc$, one has $\Oc_\Sc= f_\ast\Oc_\Zc\times f_\ast\Oc_\Wc$ so that $\Sc$ decomposes as $\Sc_1\coprod \Sc_2$ with $s\in \Sc_1$ and $f(Z)\subseteq S_1$, $f(W)=S_2$. We can replace $\Sc$ with $\Sc_1$ and then we have that $g\colon \Zc\iso \Xcal$ and that $f$ is finite.
 \end{proof}

\begin{corol}\label{cor:qaffine} A  superscheme monomorphism $f\colon \Xcal \to \Sc$,  locally of finite type, is quasi-affine.
\end{corol}
\begin{proof} $f$ is separated because $\Delta_f$ is an isomorphism and, being locally of finite type, is quasi-finite by  Proposition \ref{prop:mono}. Then it is quasi-affine by Proposition \ref{prop:qaffine2}.
\end{proof}

As in the bosonic case, $f_\ast\Oc_\Xcal$ may not be a finite algebra over $\Oc_\Sc$. For a quasi-finite and separated scheme morphism  $f\colon X\to S$,   Zariski's Main Theorem states that we can modify the above factorization to get another one
$$
X \xrightarrow{g} \Spec \Ac \xrightarrow{h} S,
$$
where $\Ac$ is an $\Oc_S$-algebra, such that $g$ is an open immersion and $h$ is \emph{finite}. Actually, one takes $\Ac$ as the integral closure of $\Oc_S$ in $f_\ast\Oc_X$. It is not clear at present  whether a similar result is true for superschemes. A proof would require a   definition of integral closure for superrings and the proof of some finiteness results. Let us state this as a conjecture.
\begin{conj}[Super Zariski's Main Theorem]
Let $f\colon \Xcal \to \Sc$ be a quasi-finite and separated superscheme morphism. There exists a finite $\Oc_\Sc$-superalgebra $\Ac$ such that $f$ factors as a composition
$$
\Xcal \xrightarrow{g} \SSpec \Ac \xrightarrow{h} \Sc
$$
where $g$ is an open immersion and $h$ is the finite morphism induced by the superalgebra morphism $\Oc_\Sc \to \Ac$.
\end{conj}

\subsection{Local properties of superscheme morphisms}

\begin{defin}\label{def:localproperties} A property $\mathbf P$ of  superscheme  morphisms  is:
\begin{enumerate}
\item stable by base change, if whenever a morphism $\Xcal\to\Ycal$ has $\mathbf P$ and $\Zc\to\Ycal$ is any morphism, then $\Xcal_\Zc:=\Xcal\times_\Ycal\Zc \to \Zc$ has $\mathbf P$;
\item stable by composition, if the composition of two morphisms that have $\mathbf P$, has also $\mathbf P$;
\item
\'etale (resp.\! smooth, fppf, fpqc) locally on the target, if for any surjective \'etale (resp.\! smooth, fppf, fpqc) morphism $\Vc\to\Ycal$, a morphism $\Xcal\to\Ycal$  has   property $\mathbf P$ if and only if the morphism  $\Xcal_\Vc:=\Xcal\times_\Ycal \Vc\to \Vc$  has property $\mathbf P$;
\item \'etale (resp.\! smooth, fppf, fpqc) locally on the source, if for any surjective \'etale (resp.\! smooth, fppf, fpqc) morphism $\Ucal\to\Xcal$, a morphism $\Xcal\to\Ycal$  has   property $\mathbf P$ if and only if  the composition morphism  $\Ucal \to \Xcal \to \Ycal$  has property $\mathbf P$.
\end{enumerate}
\end{defin}

The next result is straightforward; for easy of reference, we have included here the contents of Proposition \ref{prop:basech}.
\begin{prop}\label{prop:localproperties}\
\begin{enumerate}
\item The following properties are stable under base change and composition and fpqc  local on the target  (and then, fppf, smooth and \'etale local on the target): being surjective, flat, locally of finite presentation,  fppf, locally of finite type,  quasi-finite, locally quasi-finite, separated,  affine, quasi-affine, smooth, \'etale and unramified.
\item The following properties are \'etale local on the source: being surjective, flat, locally of finite presentation,  locally of finite type, locally quasi-finite, smooth, \'etale and unramified.
\end{enumerate}
\end{prop}

  \begin{proof} (1) Surjectivity 
(Definition \ref{def:surj}) 
is clearly invariant under base change and composition (meaning that the composition of two surjective maps is surjective). The fact that surjectivity is fpqc local on the target follows from the fact that this property holds true for scheme morphisms.

Invariance of flatness under base change and composition comes from the following basic algebraic fact :

\begingroup \em If $M$ is a flat $\Z_2$-graded $\As$-module, where $\As$ is a superring, and $\As\to \As'$ is a superring morphism,  then $M\otimes_\As\As'$ is a flat $\Z_2$-graded $\As'$-module (actually, if $\As\to \As'$ is faithfully flat, $M\otimes_\As\As'$ is $\As'$-flat if and only if $M$ is $\As$-flat).   \endgroup

Also the proof that  flatness is fpqc local on the target may be reduced to this basic fact, exactly as in \cite[Tag 02YJ, Lemma 35.23.17]{Stacks}.

Base change invariance of the locally finite type property  follows from the same property for superring morphisms. Invariance of the 
 locally finite type property under composition is proved directly. Fpqc locality on the target is proved as in the ordinary case, see e.g.~\cite[Tag 02YJ, Lemma 35.23.12]{Stacks}. A completely similar situation prevails for the locally finite presentation property
 (\cite[Tag 02YJ, Lemma 35.23.13]{Stacks}).
 
Since a superscheme morphism is locally quasi-finite if its bosonic reduction is locally quasi-finite, it is quite clear that both being locally quasi-finite or quasi-finite is invariant under base change and composition. Also,  being locally quasi-finite is fpqc local if and only if the same property holds for the bosonic reduction. So eventually we  are left with the ordinary case, which is well known and ensues from the following fact: let $f\colon X\to S$ be a scheme morphism, $\phi\colon S'\to S$ a surjective morphism and $f'\colon X'=X\times_S S' \to S'$. Then, for every point $s\in S$ there exists a point $s'\in \phi^{-1}(s)$ and one has $f'\colon X'_{s'}\simeq X_s$, so that if $X'\to S'$ has discrete fibers, so has $f\colon X\to S$.
  The quasi-finite case follows  because  ``quasi-compact'' is fpqc local on the base.
  
 The base change of a separated morphism is separated because taking the diagonal morphism of the base-changed morphism is the base change of the diagonal, and  the property of being a closed immersion is invariant under base change. Invariance under composition is proved by computing the diagonal of a composition of morphisms. Fpqc locality is proved as in the ordinary case, see \cite[Tag 02YJ]{Stacks}, Lemma 35.23.6.

To prove the statements about the affine and the quasi-affine properties for a superscheme morphism $f\colon\Xcal\to\Sc$, we can assume that $f$ is quasi-compact and separated, so that $f_\ast\Oc_\Xcal$ is a quasi-coherent sheaf of $\Oc_\Sc$-algebras, and we have a factorization $\Xcal\xrightarrow{g} \SSpec f_\ast\Oc_\Xcal \to \Sc$ (Equation \eqref{eq:qaffine}) that commutes with flat base change $\Sc'\to\Sc$. Then we use that $f$ is affine (resp.\! quasi-affine) if and only if $g\colon \Xcal \to \SSpec f_\ast\Oc_\Xcal$ is an isomorphism (resp.\! open immersion) and that the same happens for the base change morphism $f'\colon \Xcal\times_\Sc\Sc' \to \Sc'$. 

 The statements about compatibility with base change and composition for smooth, \'etale and unramified morphisms are proved in Proposition \ref{prop:basech}. Since the property of being locally of finite type has already seen to be fpqc local on the target, the case of unramified follows directly from 
 Proposition \ref{prop:formunram}. Analoguosly, using that being locally of finite presentation and flat are fpqc local, the \'etale case is a consequence of Proposition \ref{prop:smoothunramified}.
 
Finally, using again that being locally of finite presentation and flat are fpqc local, the proof that smoothness is local on the target for the fpqc topology follows straighforwardly from Proposition \ref{prop:smooth2} and Corollary \ref{cor:smoothsplit0}: let $f\colon \Xcal \to \Sc$ be  a superscheme morphism and $\Sc'\to \Sc$ a faithfully flat quasi-compact morphism such that $f'\colon \Xcal'=\Xcal\times_\Sc \Sc \to \Sc'$ is smooth. By Proposition \ref{prop:smooth2}, to prove that $f$ is smooth we have to show that the fiber $\Xcal_s$ is a smooth superscheme over $\kappa(s)$ for every point $s\in \Sc$. Since $\Sc'\to\Sc$ is surjective, there is a point $s'\in\Sc'$ lying over $s$. Then $\Xcal'_{s'}=\Xcal_s\times_{\Spec \kappa(s)}\Spec \kappa(s')$, and since this is a smooth superscheme over $\kappa(s')$, $\Xcal_s$ is smooth over $\kappa(s)$ by Corollary \ref{cor:smoothsplit0}.
 
 The claims in (2) follow easily from the definitions.
\end{proof}

\subsection{Fpqc descent for quasi-coherent sheaves on superschemes}
 In this Subsection we provide a detailed proof of fpqc descent for quasi-coherent sheaves on superschemes, as it was  stated in \cite[Prop.\! A.29 Erratum]{BrHRPo20}.
Our aim is to prove the following descent result for quasi-coherent modules on superschemes:

\begin{prop}\label{prop:fpqcdescent} Let $p\colon \Tc\to \Sc$ be a faithfully flat quasi-compact superscheme morphism, and let  $\Rcal=\Tc\times_\Sc\Tc$ and $(p_1,p_2)\colon \Rcal\rra\Tc$ be  the projections. Let $\Nc$ be a graded quasi-coherent sheaf on $\Tc$ with  descent data, that is, an even isomorphism  $\psi\colon p_1^\ast\Nc\iso p_2^\ast \Nc$ such that,
if $\U = \Tc\times_\Sc\Tc\times_\Sc\Tc$, and $\varphi_1$, $\varphi_2$, $\varphi_3$ are the three projections onto $\Rcal$ that omit the first, the second, and the third factor, respectively, the condition
$\varphi_1^\ast \psi \circ \varphi_3^\ast \psi = \varphi_2^\ast\psi$
holds. Then, there is a graded quasi-coherent sheaf $\M$ on $\Sc$ such that $\Nc\simeq p^\ast\M$.
\end{prop}

For the proof we adapt to the supersetting the descent theory developed in  \cite{Vi07}.
The  classical version for schemes is the fpqc (faithfully flat quasi-compact) descent theorem
in SGA1 \cite[Thm.~VIII.1.1]{SGA1}.  

 We start by establishing Proposition \ref{prop:fpqcdescent} is categorical terms. To do so, we give a working  definition.
 
 \begin{defin}\label{def:descentcact} Let $p\colon \Tc\to \Sc$ be a  superscheme morphism. The category  of descent data of quasi-coherent sheaves with respect to $p\colon \Tc\to \Sc$ is  is the category $\Qcoh_{\Tc\to\Sc}$ described as follows:
 \begin{enumerate}
 \item objects are pairs $(\Nc,\psi)$ where $\Nc$ is a quasi-coherent sheaf of $\Oc_\Tc$-modules and $\psi$ is an even isomorphism $p_1^\ast \Nc \iso p_2^\ast\Nc$of $\Oc_\Rcal$-modules such that the descent condition
$\varphi_1^\ast \psi \circ \varphi_3^\ast \psi = \varphi_2^\ast\psi$
holds;
 \item morphisms $\beta\colon (\Nc,\psi)\to (\Nc,\psi')$ are even morphisms $\beta\colon \Nc \to \Nc'$ such that the diagram
 $$
  \xymatrix{
 p_1^\ast\Nc \ar[r]^\psi \ar[d]^{p_1^\ast\beta} & p_2^\ast \Nc \ar[d]^{p_2^\ast \beta} 
 \\
p_1^\ast \Nc' \ar[r]^{\psi'} & p_2^\ast \Nc'
 }
 $$
  commutes, where we are using the notation of Proposition \ref{prop:fpqcdescent}.
 \end{enumerate}
 \end{defin}
 
There is a natural isomorphism $\iota_{\Tc/\Sc}(\Nc): p_1^\ast\Nc \iso p_2^\ast \Nc$  holding for every quasi-coherent sheaf $\Nc$ on $\Tc$,   locally given by $n\otimes b \mapsto (-1)^{|b| |n|} b\otimes n$. Also, for every quasi-coherent sheaf $\M$ on $\Sc$, there is  a morphism $\alpha_\M\colon \M \to p_\ast p^\ast \M$ of $\Oc_\Sc$-modules locally defined  by $m\mapsto 1\otimes m$.
 
 Let $\Qcoh_\Sc$ be the category of quasi-coherent sheaves of $\Oc_\Sc$-modules. There is  a functor
\begin{align}\label{eq:fpqcdescent}
F_p\colon \Qcoh_\Sc & \to \Qcoh_{\Tc\to\Sc}
\\
\M & \mapsto (p^\ast \M, \psi_\M:=\iota_{\Tc/\Sc}(p^\ast\M))\,.
 \end{align}

 Now Proposition \ref{prop:fpqcdescent} is a consequence of the following one:
 \begin{prop}\label{prop:fpqcdescent2} If $p\colon \Tc\to \Sc$ is faithfully flat and quasi-compact, the functor $F_{\Tc\to\Sc}$ in equation \eqref{eq:fpqcdescent} is an equivalence of categories.
 \end{prop}

The proof of Proposition \ref{prop:fpqcdescent2} is split into three parts, described in the next subsections.

 \subsubsection*{The case when the target and the source are affine: descent for (graded) modules over superrings}
 
 When the target and the source are affine, that is $\Sc=\SSpec \As$ and $\Tc=\SSpec \Bs$, a morphism $p\colon \Tc \to \Sc$ corresponds to an even morphism $\As\to\Bs$ of superrings. 
The functor $F_p$ of equation \eqref{eq:fpqcdescent}   takes the form
\begin{align}\label{eq:affinedescent}
 F\colon \operatorname{Mod}_\As:=\Qcoh_\Sc & \to \operatorname{Mod}_{\As\to\Bs}:=\Qcoh_{\Tc\to\Sc} \\
 M & \mapsto (\Bs\otimes_\As M, \psi_M:=1\otimes \iota_M)\,,
 \end{align}
where $\iota_M\colon M\otimes_\As\Bs \to \Bs\otimes_\As M$ is given by $m\otimes b \mapsto (-1)^{|b| |m|}b\otimes m$. 
Proposition \ref{prop:fpqcdescent2} may be  rephrased as
\begin{prop}\label{prop:affinedescent} If $\As \to \Bs$ is a faithfully flat morphism of superrings, the functor $F$ of equation \eqref{eq:affinedescent} is an equivalence of categories.
 \end{prop}

To prove this we need a preliminary result.  Let us write $\Bs^{\otimes 2} =\Bs\otimes_\As\Bs$ and $\Bs^{\otimes 3}=\Bs\otimes_\As\Bs\otimes_\As\Bs$ for simplicity. 

\begin{lemma}\label{lem:exact} If $\As \to \Bs$ is faithfully flat, for every graded $\As$-module $M$ the sequence 
$$
0 \to M \xrightarrow{\alpha_M} \Bs\otimes_\As M \xrightarrow{e\otimes \Id_M} \Bs^{\otimes 2}\otimes_\As M\,,
$$
where  $e(b)=b\otimes 1 - 1\otimes b $, is exact.
\end{lemma}
\begin{proof} Since $\As\to\Bs$ is faithfully flat, it is enough to prove that the sequence
$$
0 \to \Bs\otimes_\As M \xrightarrow{1\otimes\alpha_M} \Bs^{\otimes 2} \otimes_\As M \xrightarrow{1\otimes e\otimes \Id_M} \Bs^{\otimes 3}\otimes_\As M\,,
$$
obtained by changing base to $\Bs$ is exact. We have

\noindent (1) $1\otimes\alpha_M$ is injective: this follows because the morphism $\eta\otimes 1\colon \Bs^{\otimes 2} \otimes_\As M \to \Bs\otimes_\As M$, where $\eta\colon \Bs^{\otimes 2} \to \Bs$ is the multiplication  $\eta(b\otimes b')= bb'$, is a retraction of $1\otimes\alpha_M$.

\noindent (2) A direct computation shows that $(1\otimes e\otimes \Id_M)\circ (1\otimes\alpha_M)=0$. To conclude, take  $\sum_i b_i\otimes b'_i\otimes m_i\in \ker (1\otimes e\otimes \Id_M)$. Then $0= \sum_i b_i\otimes b'_i\otimes 1\otimes m_i- \sum_i b_i\otimes 1\otimes b'_i\otimes m_i$. Applying $\eta\otimes 1\otimes \Id_M$, we have $0= \sum_i b_ib'_i\otimes 1\otimes m_i- \sum_i b_i\otimes b'_i\otimes m_i$, so that $\sum_i b_i\otimes b'_i\otimes m_i=(1\otimes \alpha_M)(\sum_i b_ib'_i \otimes m_i)$, and we finish.
\end{proof}

 \noindent\emph{Proof of Proposition \ref{prop:affinedescent}.} Define a functor $G\colon \operatorname{Mod}_{\As\to\Bs}\to \operatorname{Mod}_\As$ by sending an object $(N,\psi)$ to the $A$-module $G(N)\subseteq N$ consisting of the elements $n\in N$ such that $\psi(n\otimes 1)=1\otimes n$, and of an arrow $\beta\colon (N,\psi)\to (N',\psi')$ to the induced morphism $\beta\colon N\to N'$.

Note that $(e\otimes\Id_M)(b\otimes m)= \psi_M(b\otimes m\otimes 1)-1\otimes b\otimes m$, so that $G( \Bs\otimes_\As M)= \ker (e\otimes \Id_M)$. By  Lemma \ref{lem:exact}  the morphism $\alpha_M\colon M\to \Bs\otimes_\As M$ gives a natural isomorphism $M\iso G(M\otimes_\As\Bs)=(GF)(M)$, thus proving that the composition $GF$ is isomorphic to the identity o $\operatorname{Mod}_\As$.

It remains  only to prove that  $FG$ is isomorphic to the identity on $\operatorname{Mod}_{\As\to\Bs}$. Take an object $(N,\psi)$ of $\operatorname{Mod}_{\As\to\Bs}$. Then $M:=G(N,\psi)$ is a graded $\As$-submodule of $N$ and thus there is a natural even morphism $\theta\colon \Bs\otimes_\As M \to N$ of $\Bs$-modules given by $\theta(b\otimes m)=bm$. A straightforward computation shows that $\theta$ is a morphism in the category $\operatorname{Mod}_{\As\to\Bs}$, namely, that the diagram
$$
\xymatrix{
\Bs\otimes_\As\Bs\otimes_\As M \ar[r]^(.6){1\otimes\theta} \ar[d]^{1\otimes\Id_M} & \Bs\otimes_\As N \ar[d]^\psi
\\
\Bs\otimes_\As M \otimes_\As \Bs \ar[r]^(.6){\theta\otimes 1}& N\otimes_\As \Bs
}
$$
commutes. It follows that $\theta$ defines a natural transformation $\Id \to FG$. We have to check that $\theta$ is an isomorphism. 

Consider the morphism $\beta\colon \Bs\otimes_\As N$ given by $\beta(n)=\psi(n\otimes 1)$, so that $M =\ker (\alpha_N-\beta)$.  We have a diagram with exact rows
$$
\xymatrix{
0 \ar[r] & M\otimes_\As\Bs \ar[r]^{i\otimes 1} \ar[d]^{\theta\circ \iota_M} & N\otimes_\As\Bs \ar[rr]^(.4){(\alpha_N-\beta)\otimes 1} \ar[d]^\psi & & \Bs\otimes_\As N \otimes_\As\Bs \ar[d]^{\psi_1}
\\
0 \ar[r] & N \ar[r]^{\alpha_M} & \Bs\otimes_\As N \ar[rr]^(.4){ e\otimes\Id_N} & & \Bs\otimes_\As \Bs \otimes_\As N
}
$$
where $i\colon M \hookrightarrow N$ is the inclusion.
Assume we have proved that the diagram   commutes. Then, since both $\psi$ and $\psi_1$ are isomorphisms, $\theta\circ \iota_M$ is an isomorphism as well, so that $\theta$ is an isomorphism and we conclude.

So, all that remains is to prove the commutativity of the diagram. The commutativity of the first square is immediate. For the second we use the descent property $\psi_2=\psi_1\circ\psi_3$. The computation is as follows: on the one hand we have $\psi_1\circ (\alpha_N\otimes 1)= (e_2\otimes \Id_N)\circ \psi$; on the other hand
\begin{align*}
(\psi_1 (\beta\otimes 1))(n\otimes b)& = \psi_1(\psi(n\otimes 1)\otimes b) = \psi_1\psi_3(n\otimes 1\otimes b) \\
& = \psi_2(n\otimes 1\otimes b)= ((e_1\otimes \Id_N) \psi )(n\otimes b)\,,
\end{align*}
thus finishing the proof of Proposition \ref{prop:affinedescent}. \qed
 
 \subsubsection*{The case when the target is affine}\label{ss:affinetarget}
 
We are going to prove Proposition \ref{prop:fpqcdescent2} when the target $\Sc$ is affine.

Since $\Tc$ is quasi-compact, there is a finite cover $\Tc=\bigcup_i \Tc_i$  by affine open subsuperschemes $\Tc_i$. Then we have two fpqc covers $q\colon\Tc'= \coprod_i \Tc_i \to \Tc$ and $p'=p\circ q\colon \Tc' \to \Sc$.

\begin{lemma}\label{lem:zariskidescent} Descent (i.e. Proposition \ref{prop:fpqcdescent2}) is true for $q\colon \Tc' \to \Tc$, so that $F_q\colon \Qcoh_{\Tc}\to \Qcoh_{\Tc'\to\Tc}$ is an equivalence of categories.
\end{lemma}
\begin{proof} This is simply 
glueing of sheaves and morphisms for the Zariski open cover $\Tc=\bigcup_i \Tc_i$.
\end{proof}

Now, we prove descent when the target is affine in two steps.

\noindent (1) \emph{The functor $F_p\colon \Qcoh_\Sc \to \Qcoh_{\Tc\to\Sc}$ is essentially surjective}.
Let $(\Nc,\psi)$ an object of $\Qcoh_{\Tc\to\Sc}$. Then $(q^\ast\Nc, (q\times q)^\ast \psi)$ is an object of $\Qcoh_{\Tc'\to\Sc}$, so that descent in the affine case (Proposition \ref{prop:affinedescent}) implies the existence of a quasi-coherent sheaf $\M$ on $\Sc$ such that $q^\ast\Nc \simeq p'^\ast \M$ and $\psi_{q^\ast \Nc}=\iota_{\Tc'/\Sc}(p'^\ast\M)=\iota_{\Tc'/\Sc}(q^\ast p^\ast\M)$. Then $F_q(p^\ast\M)\iso F_q(\Nc)$ and by Lemma \ref{lem:zariskidescent}, there is an isomorphism $\Nc \iso p^\ast\M$.

\noindent (2) \emph{The functor $F_p$ is fully faithful}.
 Given quasi-coherent sheaves of $\Oc_\Sc$-modules $\M$, $\M'$, we have a commutative diagram
$$\small{
\xymatrix{
\Hom_{\Qcoh_\Sc}(\M, \M') \ar[r]^(.4){F_p} \ar[rd]_{F_{p'}}^\simeq & \Hom_{\Qcoh_{\Tc\to\Sc}}(F_p(\M), F_p(\M')) \ar[d]^{q^\ast} \ar@{^(->}[r] & \Hom_{\Qcoh_\Tc}(p^\ast\M, p^\ast\M') \ar[d]^{F_q}_\simeq\\
&\Hom_{\Qcoh_{\Tc'\to\Sc}}(F_{p'}(\M), F_{p'}(\M')) \ar@{^(->}[r]& \Hom_{\Qcoh_{\Tc'\to\Tc}}(F_{p'}(\M), F_{p'}(\M'))
}}
$$
where the bijections are due to the fact that both $F_{p'}$  and $F_q$ are fully faithful. It follows that the first horizontal arrow is also bijective and then  the functor $F_p$ is fully faithful as claimed.

 \subsubsection*{The general case}
 
Take an affine open cover $\Sc=\bigcup_i \Sc_i$ and consider the induced morphisms $p_i\colon \Tc_i:=p^{-1}(\Sc_i) \to \Sc_i$.  

\noindent (1) \emph{The functor $F_p\colon \Qcoh_\Sc \to \Qcoh_{\Tc\to\Sc}$ is essentially surjective}.
Given an object $(\Nc,\psi)$ of $\Qcoh_{\Tc\to\Sc}$,  the isomorphism $\psi\colon p_1^\ast \Nc \to p_2^\ast \Nc$ yields by restriction an isomorphism from $p_1^\ast \Nc_i \to p_2^\ast \Nc_i$ where $\Nc_i=\Nc_{\vert \Tc_i}$. One sees that they satisfy descent, that is, that  $(\Nc_i,\psi_i)$ is an object of $\Qcoh_{\Tc_i\to\Sc_i}$. By the previous case (Subsection \ref{ss:affinetarget}), there exist quasi-coherent sheaves $\M_i$ over $\Sc_i$ and unique isomorphisms $\alpha_i\colon F_{p_i}(\M_i)\iso(\Nc_i,\psi_i)$ in $\Qcoh_{\Tc_i\to\Sc_i}$.

By restricting those isomorphisms to $\Tc_{ij}=\Tc_i\cap\Tc_j$, we get isomorphisms
$$
F_{p_{ij}}(\M_{i\vert\Sc_{ij}})= F_{p_i}(\M_i)_{\vert \Tc_{ij}} \overset{\alpha_i} {\simeq}(\Nc_{i\vert \Tc_{ij}},\psi_{i | \Tc_{ij}})
$$
and then isomorphisms
$$
\alpha_i^{-1}\circ\alpha_j\colon F_{p_{ij}}(\M_{j\vert\Sc_{ij}})\simeq F_{p_{ij}}(\M_{i\vert\Sc_{ij}})
$$
Since $F_{p_{ij}}$ is an equivalence of categories, there exists a unique isomorphism $\phi_{ij}\colon \M_{j\vert\Sc_{ij}} \simeq \M_{i\vert\Sc_{ij}}$ such that $F_{p_{ij}}\circ \phi_{ij}= \alpha_i^{-1}\circ\alpha_j$. One checks that the cocycle condition $\phi_{ik}=\phi_{ij}\circ\phi_{jk}$ is fulfilled. Thus, there is a quasi-coherent sheaf $\M$ on $\Sc$ and isomorphisms $\M_{| \Sc_i}\simeq \M_i$. Now, we have isomorphisms $p^\ast \M_{| \Tc_i}\simeq p_i^\ast \M_i\overset{\alpha_i}{\simeq}\Nc_i$ that  coincide on the intersections $\Tc_{ij}$, and then   glue   to an isomorphism $p^\ast \M\simeq \Nc$, so that $F_p(\M)\simeq (\Nc,\psi)$.

\noindent (2) \emph{The functor $F_p\colon \Qcoh_\Sc \to \Qcoh_{\Tc\to\Sc}$ is fully faithful}.
The proof is an easy computation using glueing of sheaves and morphisms for the Zariski topology and the similar property in Subsection \ref{ss:affinetarget}.

This finishes the proof of Proposition \ref{prop:fpqcdescent2} on fpqc descent for quasi-coherent sheaves on superschemes.
\qed

\subsection{Fppf descent for superschemes} \label{subsec:fppf}
We complete some statements about fppf  descent for superschemes given in \cite[Props.\! A.28, A.30]{BrHRPo20}.
The first is that superscheme morphisms satisfy effective descent for the fppf topology.
\begin{prop}\label{prop:fpfpmordescent} A superscheme is a sheaf for the fppf topology. That is, if $\Xcal$ is a superscheme, $\pi\colon \Tc\to \Sc$ is an fppf covering, and $\phi\colon \Tc\to \Xcal$ is a superscheme morphism such that the two arrows $p_1^\ast\phi\colon \Tc\times_\Sc\Tc$ and  $p_2^\ast\phi\colon \Tc\times_\Sc\Tc$ coincide, there is a unique morphism $\psi\colon \Sc \to \Xcal$ such that $\phi=\pi^\ast\psi$.
\end{prop}
\begin{proof} The proof is a straightforward generalization of the similar statement in SGA1 \cite[Thm.~VIII.5.2]{SGA1}.
\end{proof}

As in the bosonic case, quasi-affine superschemes morphism   verify  {effective} descent.  We prove this fact following \cite[Sec.\! 2.1.3]{Alp24}.

\begin{prop}[Fpqc descent for open or closed sub-superschemes]\label{prop:subdescent}
Let $f\colon \Tc\to\Sc$ be an  fpqc superscheme morphism  and $p_i\colon \Tc\times_\Sc\Tc\to\Tc$  the two projections. If $\Vc$ is a closed (resp.\! open) sub-superscheme of $\Tc$ such that $p_1^{-1}(\Vc)=p_2^{-1}(\Vc)$ as sub-superschemes of $\Tc\times_\Sc\Tc$, there exists a closed (resp.\! open) sub-superscheme $\Zc$ of $\Sc$ such that $\Vc=f^{-1}(\Zc)$.
\end{prop}
\begin{proof} The case of closed sub-superschemes is \cite[Prop.\! A.30]{BrHRPo20}. For the case of open sub-superschemes we pass to the complementary closed sub-superschemes (with the reduced structure) and apply the result for the closed case.
\end{proof}
 
\begin{prop}[{Fpqc affine and quasi-affine descent}]\label{prop:qadescent} Let $\{\Sc_i \to \Sc\}$ be an fpqc covering of superschemes and let $\{\Xcal_i/\Sc_i, \varphi_{ij}\}$ be descent data relative to $\{\Sc_i \to \Sc\}$. If every morphism $\Xcal_i\to\Sc_i$ is affine (resp.\! quasi-affine), the descent data are effective, that is, there exist an affine (resp.\! quasi-affine) superscheme morphism $\Ycal\to \Sc$ and isomorphisms $\Xcal_i \iso \Sc_i\times_\Sc \Ycal$.
\qed
\end{prop}

\begin{proof}
We consider first the affine case. Then $\Xcal_i=\SSpec \Ac_i$ for some quasi-coherent sheaf $\Ac_i$ of $\Oc_{\Sc_i}$-superalgebras. By the  fpqc descent for quasi-coherent sheaves (Proposition \ref{prop:fpqcdescent}), we can descend the $\Ac_i$'s to a coherent sheaf $\Ac$ on $\Sc$. Moreover, we can descend the multiplication maps $\Ac_i\otimes_{\Oc_{\Sc_i}}\Ac_i\to \Ac_i$ to a morphism $\Ac\otimes_{\Oc_\Sc}\Ac\to\Ac$ which makes $\Ac$ into a superalgebra over $\Oc_\Sc$. We then take $\Ycal=\SSpec \Ac$ to finish. In the quasi-affine case, we consider the factorization of $f_i\colon\Xcal_i\to\Sc_i$ as the composition of a quasi-compact open immersion $\Xcal_i\to \Ycal_i:=\SSpec f_{i\ast}\Oc_{\Xcal_i}$ and an affine morphism $\Ycal_i\to \Sc_i$ (Proposition \ref{prop:qaffine}) and use the affine case we have just proved and fpqc descent for open subschemes (Proposition \ref{prop:subdescent}). 
\end{proof}

 {As in \cite{Alp24},  for every superscheme morphism $\Xcal\to\Sc$ and every base change $f\colon \Tc\to\Sc$ we denote by $f^\ast\Xcal$ the base change $\Xcal\times_{\Sc, f}\Tc\to\Tc$. Now we have:}

\begin{prop}[Fppf locally quasi-finite and separated descent]\label{prop:qfsepdescent} Let $\{\Sc_i \to \Sc\}$ be an fppf covering of superschemes and let $\{\Xcal_i/\Sc_i, \varphi_{ij}\}$ be descent data relative to $\{\Sc_i \to \Sc\}$. If every morphism $\Xcal_i\to\Sc_i$ is locally quasi-finite and separated, the descent data are effective, that is, there exist a  {locally quasi-finite and separated} superscheme morphism $\Ycal\to \Sc$ and isomorphisms $\Xcal_i \iso \Sc_i\times_\Sc \Ycal$.
\end{prop}
\begin{proof} We can use a similar argument as in the proof of \cite[Theorem 2.1.14]{Alp24}, using quasi-affine descent for  superscheme morphisms (Proposition \ref{prop:qadescent}).  {We can assume to have a single fppf covering $\Tc \to \Sc$ and  locally quasi-finite and separated descent data over $\Tc$, namely, a locally quasi-finite and separated morphism $\Xcal\to\Tc$ and an isomorphism $\alpha\colon p_1^\ast(\Xcal) \iso p_2^\ast(\Xcal)$ over $\Tc\times_\Sc\Tc$ satisfying $p_{23}^\ast\alpha\circ p_{12}^\ast\alpha=p_{13}^\ast\alpha$. For every quasi-compact open sub-superscheme $\Ucal$ of $\Xcal$, the composition $\Ucal\hookrightarrow \Xcal \to \Sc$ is quasi-affine by Proposition \ref{prop:qaffine2}. Now, since $\Tc\to\Sc$ is fppf, $p_2\colon \Tc\times_\Sc\Tc \to \Tc$ is also fppf, and then the topological image $V=p_2(\alpha(p_1^\ast(U))$ is open and quasi-compact and contains $U$. If we consider the open sub-superscheme $\Vc=(V,\Oc_{\Xcal | V})$ of $\Xcal$, $\alpha$ induces an isomorphism $p_1^\ast(\Vc)\iso p_2^\ast(\Vc)$ satisfying the cocycle condition. By fpqc quasi-affine descent (Proposition \ref{prop:qadescent}), $\Vc\to \Tc$ descends to a quasi-affine morphism $\Wc\to\Sc$. We can cover $\Xcal$ with quasi-compact open sub-superschemes $\Ucal_i$. The corresponding open sub-superschemes $\Vc_i$ also cover $\Xcal$, and since each $\Vc_i$ descends to a superscheme $\Wc_i$ that is quasi-affine over $\Sc$, we can glue the superschemes $\Wc_i$ to get a quasi-affine superscheme morphism $\Wc\to\Sc$ such that $\Wc\times_\Sc\Tc\simeq \Xcal$. Moreover, $\Wc\to\Sc$  is locally quasi-finite and separated by Proposition \ref{prop:localproperties}.}
\end{proof}

\begin{prop}[Descent for principal superscheme bundles]\label{prop:pbdescent}
Let $\Gsc$ be an fppf   affine group superscheme over $\Tc$, and $\Sc' \to\Sc$  an fpqc   superscheme morphism over $\Tc$. If $\Psc' \to\Sc'$ is a principal $\Gsc$-superscheme bundle and $\alpha\colon p_1^\ast \Psc'\to p_2^\ast \Psc'$ is an isomorphism of principal $\Gsc$-superscheme bundles over $\Sc'\times_\Sc \Sc'$ satisfying the cocycle condition $p_{12}^\ast\alpha\circ p_{23}^\ast\alpha=p_{13}^\ast\alpha$, then there exists a principal $\Gsc$-superscheme bundle $\Psc\to\Sc$ and an isomorphism $\phi\colon \Psc' \to f^\ast\Psc$ of principal $\Gsc$-superscheme bundles such that $p_1^\ast\phi=p_2^\ast\phi \circ \alpha$.
\end{prop}
\begin{proof} We can proceed as in \cite[Proposition 2.1.17]{Alp24} using quasi-affine descent for   superscheme morphisms  (Proposition \ref{prop:qadescent}).
\end{proof}

Define a category  \text{\sl Pol}  as follows:    objects are pairs $(f \colon \Xcal\to \Sc ,\Lcl)$, where $f$ is a proper and flat superscheme morphism and $\Lcl $ is a relatively ample even line bundle on $\Xcal $ \cite{BrHRPo20};  morphisms $(f'\colon \Xcal'\to \Sc' ,\Lcl)\to (f \colon \Xcal\to \Sc,\Lcl)$ are triples $(g, h, \epsilon)$, where
$g\colon\Sc'\to\Sc$, $h\colon \Xcal'\to\Xcal$ are  superscheme  morphisms  such that the diagram
$$
\xymatrix{\Xcal'\ar[d]^{f'}\ar[r]^h & \Xcal\ar[d]^f \\
\Sc'\ar[r]^g& \Sc
}
$$ 
is cartesian, and $\epsilon\colon \Lcl'\to h^\ast\Lcl$ is an isomorphism of even line bundles on $\Xcal'$.  \text{\sl Pol}  is a fibered category over the category $\Sf$ of superschemes via  the functor $(f \colon \Xcal\to \Sc ,\Lcl)\to \Sc$.

\begin{prop}[Descent for polarized superschemes] \label{prop:descpolar} The fibered category $ \text{\sl Pol} \to\Sf$ satisfies fppf effective descent.
\end{prop}
\begin{proof}
The proof of \cite[Prop.\! 4.4.12]{Ol16} is also valid for superschemes after replacing the reference  \cite{EGAIII-II} about cohomology base change with  \cite[Prop.\! 3.7 and Thm.\! 2.35]{BrHRPo20}.
\end{proof}

\frenchspacing
 
 \def\cprime{$'$}

\end{document}